\documentclass[12pt]{amsart}

\usepackage{graphicx,amsmath,amscd,amsfonts,amsthm,amssymb,verbatim,fullpage,enumerate,color}

\usepackage{xcolor}
\usepackage[backref=page]{hyperref}
\definecolor{couleur_cite}{rgb}{0.05,.4,0.05}
\definecolor{couleur_link}{rgb}{0.05,0.05,0.4}
\hypersetup{
bookmarksopen,bookmarksnumbered,
colorlinks,
linkcolor=couleur_link,citecolor=couleur_cite,
}

\newtheorem{theorem}{Theorem}[section]
\newtheorem{lemma}[theorem]{Lemma}
\newtheorem{prop}[theorem]{Proposition}

\newtheorem{cor}[theorem]{Corollary}
\newtheorem{definition}[theorem]{Definition}

\theoremstyle{remark}
\newtheorem*{remark}{Remark}

\newcommand{\R}{\mathbb R}
\newcommand{\C}{\mathbb C}
\newcommand{\N}{\mathbb N}
\newcommand{\Z}{\mathbb Z}
\newcommand{\Q}{\mathbb Q}

\newcommand{\A}{\mathbb A}
\newcommand{\bbP}{\mathbb P}
\newcommand{\bbH}{\mathbb H}

\newcommand{\p}{\mathfrak p}
\newcommand{\g}{\mathfrak g}
\newcommand{\gq}{\mathfrak q}
\newcommand{\gk}{\mathfrak k}
\newcommand{\ga}{\mathfrak a}

\newcommand{\gu}{\mathfrak u}
\newcommand{\gh}{\mathfrak h}
\newcommand{\gt}{\mathfrak t}

\newcommand{\fre}{\mathfrak e}

\newcommand{\gsl}{\mathfrak{sl} }
\newcommand{\gsu}{\mathfrak{su} }
\newcommand{\gso}{\mathfrak{so} }

\newcommand{\gG}{\mathfrak G}
\newcommand{\gH}{\mathfrak H}
\newcommand{\gU}{\mathfrak U}
\newcommand{\gT}{\mathfrak T}
\newcommand{\gX}{\mathfrak X}
\newcommand{\gY}{\mathfrak Y}

\newcommand{\cO}{\mathcal O}

\newcommand{\cL}{\mathcal L}

\newcommand{\cS}{\mathcal S}
\newcommand{\cT}{\mathcal T}
\newcommand{\cM}{\mathcal M}
\newcommand{\cP}{\mathcal P}
\newcommand{\cQ}{\mathcal Q}

\newcommand{\cH}{\mathcal H}

\newcommand{\bG}{{\underline{G} }}

\newcommand{\bH}{{\underline{H} }}
\newcommand{\bT}{{\underline{T} }}
\newcommand{\bL}{{\underline{L} }}
\newcommand{\bU}{{\underline{U} }}
\newcommand{\bB}{{\underline{B} }}

\newcommand{\bX}{{\underline{X} }}
\newcommand{\bV}{{\underline{V} }}

\newcommand{\bY}{{\underline{Y} }}

\newcommand{\bI}{\underline{I} }

\newcommand{\hG}{\widehat{G}}
\newcommand{\hT}{\widehat{T}}

\newcommand{\tK}{\widetilde{K}}
\newcommand{\tO}{\widetilde{\cO}}
\newcommand{\tk}{\widetilde{k}}

\newcommand{\be}{\begin{equation}}
\newcommand{\ee}{\end{equation}}
\newcommand{\bes}{\begin{equation*}}
\newcommand{\ees}{\end{equation*}}

\newcommand{\ba}{\begin{eqnarray}}
\newcommand{\ea}{\end{eqnarray}}
\newcommand{\bas}{\begin{eqnarray*}}
\newcommand{\eas}{\end{eqnarray*}}

\title{Upper bounds for Maass forms on semisimple groups}

\author{Simon Marshall}
\address{Department of Mathematics\\
University of Wisconsin Madison\\
480 Lincoln Drive\\
Madison\\
WI 53703, USA}
\email{marshall@math.wisc.edu}
\thanks{Supported by NSF grant DMS-1501230.}

\begin{document}

\begin{abstract}
We prove a power saving over the local bound for the $L^\infty$ norm of Hecke-Maass forms on any quasi-split semisimple real group that is not isogenous to a product of odd special unitary groups
\end{abstract}

\maketitle
\setcounter{tocdepth}{1}
\tableofcontents

\section{Introduction}
\label{sec1}

\subsection{Bounds for arithmetic eigenfunctions}

Let $M$ be a compact Riemannian manifold of dimension $n$, and $\psi$ a function on $M$ satisfying $(\Delta + \lambda^2) \psi = 0$ and $\| \psi \|_2 = 1$.  A classical theorem of Avacumovi\'c \cite{Av} and Levitan \cite{Le} states that

\be
\label{avac}
\| \psi \|_\infty \ll \lambda^{(n-1)/2},
\ee
that is, the pointwise norm of $\psi$ is bounded in terms of its Laplace eigenvalue.  This bound is sharp on the round sphere $S^n$ or a surface of revolution, but is far from the truth on flat tori.  It is an interesting problem in semiclassical analysis to find conditions on $M$ under which (\ref{avac}) can be strengthened, and such conditions often take the form of a non-recurrence assumption for the geodesic flow on $M$.  One result of this kind is due to B\'erard \cite{Be}, who proves that if $M$ has negative sectional curvature (or has no conjugate points if $n=2$) then we have

\be
\label{Berard}
\| \psi \|_\infty \ll \frac{\lambda^{(n-1)/2}}{\sqrt{\log \lambda}}.
\ee
See also \cite{SZ,TZ} for other theorems bounding $\| \psi \|_\infty$ under assumptions on the geodesic flow of $M$.  The problem of strengthening (\ref{avac}) for negatively curved $M$ is an interesting one, because for generic $M$ we expect that $\| \psi \|_\infty \ll_\epsilon \lambda^\epsilon$, whereas the strongest upper bound that is known in general is (\ref{Berard}).

In \cite{IS}, Iwaniec and Sarnak introduced a different condition on $M$ and $\psi$ which allows them to deduce quite a strong bound for $\| \psi \|_\infty$.  They assume that $M$ is a congruence hyperbolic manifold, in particular the quotient of $\bbH^2$ by the group of units in an order in a quaternion division algebra over $\Q$, and that $\psi$ is an eigenfunction of the Hecke operators on $M$.  They then prove that $\| \psi \|_\infty \ll_\epsilon \lambda^{5/12 + \epsilon}$.  Moreover, one expects that the assumption on $\psi$ is not necessary because the spectral multiplicities of negatively curved manifolds are always observed to be bounded.  This bound is the strongest that is known for the supremum norm of an eigenfunction on a negatively curved surface, with the next strongest being (\ref{Berard}).

We are interested in extending the methods of Iwaniec and Sarnak to higher dimensional mainfolds, which requires considering eigenfunctions on general locally symmetric spaces.  We shall only consider spaces of noncompact type, although the method of proof would apply equally well to spaces of compact type.  We make this restriction partly for convenience, and partly because the multiplicities of the Laplace spectrum on such manifolds are expected to be bounded as in the hyperbolic case.  Although these manifolds have zero sectional curvature in certain directions, their eigenfunctions are expected to exhibit essentially the same chaotic behaviour that is observed on negatively curved manifolds.

We recall that locally symmetric spaces of noncompact type are constructed by taking a noncompact semisimple real Lie group $G$, a maximal compact subgroup $K \subset G$, and a lattice $\Gamma \subset G$, and defining $Y = \Gamma \backslash G / K$.  We do not assume that $Y$ is compact.  We let $n$ and $r$ be the dimension and rank of $Y$.  We consider functions $\psi \in L^2(Y)$ that are eigenfunctions of the full ring of invariant differential operators, which is isomorphic to a finitely generated polynomial ring in $r$ variables.  This ring contains $\Delta$, and we continue to define $\lambda$ by $(\Delta + \lambda^2) \psi = 0$.

If $\Omega \subset Y$ is compact, Sarnak proves in \cite{Sa} that $\psi$ satisfies

\be
\label{Sarnak}
\| \psi|_\Omega \|_\infty \ll \lambda^{(n-r)/2}.
\ee
The analogous problem to the one solved by Iwaniec and Sarnak for $\bbH^2$ is to improve the exponent in this bound, under the assumptions that $\Gamma$ is congruence arithmetic, and that $\psi$ is an eigenfunction of the ring of Hecke operators.  (Note that when $r \ge 2$, $\Gamma$ is automatically arithmetic by a theorem of Margulis.)  This is often referred to as the problem of giving a subconvex, or sub-local, bound for the sup norm of a Maass form in the eigenvalue aspect.  Besides the original work of Sarnak and Iwaniec, the pairs $\Gamma \subset G$ for which it has previously been solved are $SL_2(\cO_F) \subset SL_2(F_\infty)$ for any number field $F$ by Blomer, Harcos, Maga, and Mili\'cevi\'c \cite{BHMM,BHM}, $Sp_4(\Z) \subset Sp_4(\R)$ by Bomer and Pohl \cite{BP}, $SL_3(\Z) \subset PGL_3(\R)$ by Holowinsky, Ricotta, and Royer \cite{HRR}, and $SL_n(\Z) \subset PGL_n(\R)$ for any $n$ by Blomer and M\'aga \cite{BM1,BM2}.  There are also results bounding eigenfunctions on the round spheres $S^2$ and $S^3$ equipped with Hecke algebras \cite{BM3,BM4}.

We note that much work has been done on variants of the sup-norm problem.  One may consider Maass forms of varying level and eigenvalue as in \cite{BH, HT1, HT2, Te1, Te2}.  There are also results bounding the $L^2$ norm of the restriction of $\psi$ to a submanifold of positive dimension \cite{Ma2, Ma3}.

\subsection{Statement of results}

We first state our result in a simple case.

\begin{theorem}
\label{simplemain}

Let $F$ be a totally real number field, and let $v_0$ be a real place of $F$.  Let $\bG/F$ be connected and semisimple.  We make the following assumptions on $\bG$.

\begin{itemize}

\item $\bG_v$ is compact for all real $v \neq v_0$

\item $\bG_{v_0}$ is $\R$-almost simple\footnote{We recall that a real group is $\R$-almost simple if it does not have a nontrivial proper connected normal subgroup.}, quasi-split, and not isogenous to $SU(n,n-1)$ for any $n$.

\end{itemize}
\noindent
Let $Y$ be a congruence manifold associated to $\bG$ as in Section \ref{sec2adelic}, and let $\Omega \subset Y$ be compact.  Let $\psi$ be a Hecke-Maass form on $Y$ satisfying $\| \psi \|_2 = 1$ and $(\Delta + \lambda^2) \psi = 0$.  We then have $\| \psi|_\Omega \|_\infty \ll \lambda^{(n-r)/2 - \delta}$.

\end{theorem}

We note that the implied constant here, and in Theorem \ref{main} below, is ineffective, though this could probably be fixed with additional work.  This is due to the application of some ineffective bounds on the complexity of algebraic sets, described in Section \ref{sec:complexity}, when proving our main bound on Hecke returns in Section \ref{Ksmallest}.

We shall deduce Theorem \ref{simplemain} from the following more general result.  To state it, it will be convenient to make two definitions.  The first is a condition on a real semisimple group $G$:

\medskip

$(\mathsf{WS})$:  $G$ is quasi-split, and not isogenous to a product of odd special unitary groups.

\medskip
\noindent
The second condition will be applied to the spectral parameters of our Maass form, to simplify the application of a theorem of Blomer-Pohl \cite[Thm. 2]{BP} and Matz-Templier \cite[Prop. 7.2]{MT} in the proof.

\begin{definition}

Let $\g$ be a real semisimple Lie algebra with Cartan decomposition $\g = \gk + \p$ and maximal abelian subspace $\ga \subset \p$.  Let $\g_i$ be the $\R$-simple factors of $\g$.  We say that $\lambda \in \ga^*_\C$ is $(A,\sigma)$-balanced for $A, \sigma > 0$ if its projections $\lambda_i$ to $\g_{i,\C}^*$ satisfy $\| \lambda_i \| \le A \| \lambda_j \|^\sigma$.

\end{definition}

We may now state the general form of our main theorem.

\begin{theorem}
\label{main}

Let $F$ be a number field, and let $v_0$ be a real place of $F$.  Let $\bG/F$ be connected and semisimple, and let $Y$ be a congruence manifold associated to $\bG$ as in Section \ref{sec2adelic}.  Let $\psi$ be a Hecke-Maass form on $Y$ satisfying $\| \psi \|_2 = 1$, with spectral parameter $\lambda \in \ga^*_\C$.  We make the following assumptions:

\begin{itemize}

\item $\bG_{v_0}$ satisfies $(\mathsf{WS})$.

\item The component of $\lambda$ at $v_0$, denoted $\lambda_0$, is $(A,\sigma)$-balanced in $\textup{Lie}(\bG_{v_0})$.

\end{itemize}
\noindent
Then if $\Omega \subset Y$ is compact, there exists $\delta = \delta(\bG,\sigma)$ and $C = C(\Omega,A,\sigma)$ such that

\be
\label{sublocal}
\| \psi|_{\Omega} \|_\infty \le C D(\lambda)^{1/2} (1 + \| \lambda_0 \|)^{-\delta},
\ee
where $D(\lambda)$ is defined in (\ref{Dlambda}).

\end{theorem}

As Theorem \ref{main} is rather general, we now give some examples of what one my prove by specializing it in various ways.  First, Theorem \ref{main} solves the sup norm problem for split groups over any number field $F$, subject to the balance condition on the spectral parameter.

\begin{cor}
\label{mainsplit}

Let $\bG/F$ be split.  Let $\psi$, $\lambda$, and $\Omega$ be as in Theorem \ref{main}.  Assume that $\lambda$ is $(A,\sigma)$-balanced in $\text{Lie}(\bG_\infty)$.  Then there exists $\delta = \delta(\bG,\sigma)$ and $C = C(\Omega,A,\sigma)$ such that

\bes
\| \psi|_{\Omega} \|_\infty \le C D(\lambda)^{1/2} (1 + \| \lambda \|)^{- \delta}.
\ees

\end{cor}

\begin{proof}

If $F$ has a real place, the corollary follows directly from Theorem \ref{main}.  If $F$ has only complex places, the $\Q$-group $\text{Res}_{F/\Q} \bG$ satisfies $(\mathsf{WS})$ at infinity so we may apply Theorem \ref{main} to it.

\end{proof}

As a second example, we may apply Theorem \ref{simplemain} to groups with $\bG(F_{v_0}) = SL(2,\C)$ so that the associated symmetric spaces are congruence arithmetic hyperbolic 3-manifolds.

\begin{cor}
\label{H3}

Let $Y$ be a compact congruence arithmetic hyperbolic 3-manifold.  If the invariant trace field $F$ of $Y$ has a subfield of index 2, then any Hecke-Laplace eigenfunction $\psi$ on $Y$ that satisfies $(\Delta + \lambda^2)\psi = 0$ and $\| \psi \|_2 = 1$ also satisfies $\| \psi \|_\infty \ll (1 + \lambda)^{1 - \delta}$ for some $\delta > 0$ depending only on $Y$.

\end{cor}

\begin{proof}

We recall the basic properties of the invariant trace field $F$ and quaternion algebra $D/F$ associated to $Y$, see \cite{MR} for details.  There is exactly one infinite place $w_0$ of $F$ that is complex, and $D$ is ramified at all real places of $F$.  In addition, $Y$ is an arithmetic manifold associated to the algebraic group $D^\times / F^\times$.  Let $L$ be the index 2 subfield of $F$, and let $v_0$ be the place of $L$ below $w_0$, which must be real by the uniqueness of $w_0$.  If we define $\bG = \text{Res}_{F/L} (D^\times / F^\times)$, then $Y$ is an arithmetic manifold associated to $\bG$, and $\bG$ satisfies ($\mathsf{WS}$) at $v_0$ so that we may apply Theorem \ref{main}.

\end{proof}

\subsection{Structure of the paper}

We now give a plan of the proof, and describe some of its new features.  The core of the amplification argument is in section \ref{sec3}.  The first step is to construct a function $k$ to insert into the pre-trace formula, which has two components:

\setlength{\leftmargini}{1em}
\begin{itemize}

\item {\bf Constructing a spectral projector ($\S$\ref{sec32}):}  We avoid assuming that $\psi$ is tempered at infinity by using the method developed in \cite[Section 7]{BM}.

\item {\bf Constructing an amplifier ($\S$\ref{sec7}):}  Our construction of an amplifier at finite places works on any split semisimple group.

\end{itemize}

After applying a theorem of Blomer-Pohl \cite{BP} and Matz-Templier \cite{MT} to prove that the spectral projector decays away from the maximal compact $K_\infty \subset \bG_\infty$, it remains to bound the number of Hecke returns in our amplifier.  The first step in doing this is to show that there is a number field $E/F$ that embeds into $F_{v_0}$, and a subgroup $\bH < \bG$ defined over $E$, such that $\bH(F_{v_0})$ is a maximal compact connected subgroup of $\bG(F_{v_0})$, which we do in section \ref{sec21}.  We next make three main steps, of which the first is:

\begin{itemize}

\item {\bf Diophantine approximation ($\S$\ref{sec5}):} We show that all $\gamma \in \bG(F) \cap \text{supp}(k)$ that map $x$ close to itself lie in a $F$-subgroup $\bL < \bG$.  $\bL$ can be thought of as a good rational approximation to the stabilizer of $x$ in $\bG(F_\infty)$.  We also show that $\bL$ is of the form $\bigcap_{\sigma \in \text{Gal}(\overline{F} / F)} (y \bH y^{-1})^\sigma$, where $y \in \bG(\overline{F})$ is controlled.  We in fact prove a more general statement, which provides a similar structure theorem for all $\gamma$ of bounded height lying near a conjugate of a fixed subvariety of $\bG$.

\end{itemize}

We next need to estimate $\# \bL(F) \cap \text{supp}(k)$, and Lemma \ref{heckelocal} reduces this to the following local problem:

\begin{itemize}

\item {\bf Estimating intersections in buildings ($\S$\ref{sec6}):}  Suppose $v$ is a finite place at which $\bG$ is split.  Let $T_v$ be a maximal split torus in $\bG_v$, and $K_v$ a hyperspecial maximal compact compatible with $T_v$.  If $\mu \in X_*(T_v)$ is a cocharacter, we wish to estimate the number of cosets in $\bG_v / K_v$ that lie in both the Hecke double coset $K_v \mu(\varpi_v) K_v$ and the image of $\bL_v$.  This is the natural local analogue of estimating $\# \bL(F) \cap \text{supp}(k)$.

\end{itemize}

After these steps, we obtain a bound for the number of Hecke returns given purely in terms of the characters and cocharacters of $\bH$ and $\bG$.  Moreover, this bound will be sharp in certain cases, e.g. if $\bH$ is defined over $F$ and the point $x$ corresponds to $\bH$.  We must determine whether this bound is good enough for us to amplify, which involves the study of:

\begin{itemize}

\item {\bf Weakly small subgroups ($\S\S$\ref{sec8}--\ref{algsec}):} Weak smallness is a condition on a reductive subgroup of a semisimple algebraic group, and is defined by a cocharacter inequality, see Definition \ref{small} and Section \ref{sec:WSdiscuss}.  If $\bH$ is weakly small in $\bG$, then the bound we obtain on Hecke returns is strong enough for amplification to work.  In these sections we show that $\bH$ is weakly small in $\bG$ if and only if $\bG$ satisfies condition ($\mathsf{WS}$).  There are links between (variants of) the weak smallness condition and the spectra of symmetric varieties, see \cite{BM}.

\end{itemize}

\setlength{\leftmargini}{2.5em}

There is another technical difference between this paper and previous amplification arguments that should be emphasized.  The arguments of section \ref{sec6} require $\bL$ to be `unramified' in a certain sense, and so we must choose our amplifier to avoid these places of ramification.  Because $\bL$ depends on $x$, the set of places used in our amplifier must also depend on $x$.

\begin{remark}
We note that it may be possible to relax condition ($\mathsf{WS}$) if $\bG$ doesn't have a maximal compact subgroup defined over $\Q$, as this would allow us to place additional restrictions on $\bL$, but we have not pursued this here.  A good example of such a $\bG$ is the multiplicative group of a division algebra of prime degree.
\end{remark}

\subsection{The significance of condition ($\mathsf{WS}$)}
\label{sec:WSdiscuss}

We now give a second outline of the proof, with the aim of explaining the appearance of condition ($\mathsf{WS}$) in Theorem \ref{main}.  All unexplained notation is standard and defined in Section \ref{sec2}.

Assume we are working over $\Q$, so that we only need to consider one infinite place.  We also assume that there is a connected $\Q$-group $\bH < \bG$ such that $\bH_\infty$ is a maximal compact connected subgroup of $\bG_\infty$.  Let $\bT_H < \bT$ be maximal $\Q$-tori in $\bH$ and $\bG$.  Let $\cP$ be the set of finite places at which $\bT$ and $\bT_H$ split, and all data are unramified.  Let $\| \cdot \|^*$ be the function on $X_*(\bT)$ given by $\| \mu \|^* = \underset{w \in W}{\max} \langle \mu, \rho \rangle$.  We define $\| \cdot \|^*_H$ on $X_*(\bT_H)$ similarly.  These functions are seminorms, and norms if the groups are semisimple.  If $v \in \cP$ and $\mu \in X_*(\bT)$, we define the approximately $L^2$-normalized Hecke operator $\tau(v,\mu) = q_v^{-\| \mu \|^*} 1_{K_v \mu(\varpi_v) K_v}$.

Let $Y$ be an arithmetic congruence manifold associated to $\bG$.  Let $\psi$ be our Hecke-Maass form, and $x \in Y$ some point at which we want to bound $\psi$.  If $\tau$ is a Hecke operator on $Y$, i.e. a weighted correspondence, we let $\tau \cdot x$ denote the weighted set of points obtained by applying $\tau$ to $x$.  Let $N > 0$ be the length of our amplifier, and let $\cP_N = \{ v \in \cP : N/2 < q_v < N \}$.  For simplicity, we shall assume that the Hecke operator that we use to amplify is $\cT \cT^*$, where $\cT = \sum_{v \in \cP_N} \tau(v,\mu)$ for some carefully chosen $\mu \in X_*(\bT)$.  We assume that:

\smallskip
\noindent
$\bullet$ $\cT \cT^*$ acts on $\psi$ with eigenvalue $N^{2+o(1)}$,
\smallskip

\noindent
which implies:

\smallskip
\noindent
$\bullet$ Amplification succeeds if $\cT \cT^* \cdot x$ has mass at most $N^{2-\delta}$ near $x$.
\smallskip

\noindent
This condition is related to the seminorms $\| \cdot \|^*$ and $\| \cdot \|^*_H$ by the following result:

\smallskip
\noindent
$(\star)$ $\tau(v,\mu) \cdot x$ has mass at most $\displaystyle \sum_{\lambda \in W\mu \cap X_*(\bT_H)} q_v^{2\| \lambda \|^*_H - \| \lambda \|^*}$ near $x$.
\smallskip

\noindent
Moreover, for products $\tau(v,\mu) \tau(w,\nu) \cdot x$, the bound is the product of the individual bounds for $\tau(v,\mu)$ and $\tau(w,\nu)$.  Because of our assumption that $\bH$ was defined over $\Q$, this bound is sharp for some $x$, e.g. if $x$ corresponds to the image of $\bH_\infty$ in $Y$.  The appearance of $2\| \lambda \|^*_H - \| \lambda \|^*$ in ($\star$) leads us to make the following two definitions:

\smallskip
\noindent
$\bullet$ We say that $\bH$ is quasi-small in $\bG$ if $\| \lambda \|^* \ge 2\| \lambda \|^*_H$ for all $\lambda \in X_*(\bT_H)$.

\noindent
$\bullet$ We say that $\bH$ is weakly small in $\bG$ if it is quasi-small, and either $\dim \bT_H < \dim \bT$, or there is $\mu \in X_*(\bT)$ such that $\| \mu \|^* > 2 \, \underset{w \in W}{\max} \| w \mu \|^*_H$.
\smallskip

\noindent
Moreover, in Sections \ref{sec8} and \ref{algsec} we prove:

\smallskip
\noindent
$\bullet$ $\bH$ is quasi-small in $\bG$ if and only if $\bG_\infty$ is quasi-split.

\noindent
$\bullet$ $\bH$ is weakly small in $\bG$ if and only if $\bG_\infty$ satisfies ($\mathsf{WS}$).\\

We now show how ($\mathsf{WS}$) and ($\star$) allow us to bound $\cT \cT^* \cdot x$, and in fact why ($\mathsf{WS}$) is almost necessary for doing so.  When expanding $\cT \cT^*$, we obtain diagonal and off-diagonal terms.  There are only $N^{1 + o(1)}$ diagonal terms $\tau(v,\mu) \tau(v,\mu)^*$, and so it suffices to show that $\tau(v,\mu) \tau(v,\mu)^* \cdot x$ has mass $O(1)$ near $x$ for every $v$.  The operator $\tau(v,\mu) \tau(v,\mu)^*$ is still roughly $L^2$ normalized, so when it is expanded in elementary Hecke operators $c_\lambda \tau(v,\lambda)$ we have $c_\lambda \ll 1$.  As a result, combining ($\star$) with the inequality defining quasi-smallness and summing over $\lambda$ provides the required bound for $\tau(v,\mu) \tau(v,\mu)^* \cdot x$.

Note that if $\bG$ is absolutely almost simple, then it is very hard to control the set of $\lambda$ that occur in the expansion of $\tau(v,\mu) \tau(v,\mu)^*$, other than saying that they lie in some ball of radius determined by $\mu$.  Therefore, if quasi-smallness failed then one would probably encounter a $\lambda$ with $2 \| \lambda \|^*_H > \| \lambda \|^*$, and the contribution this made to our bound for $\tau(v,\mu) \tau(v,\mu)^* \cdot x$ could easily be of order greater than $N$.

There are $N^{2+o(1)}$ off-diagonal terms, and so for these we must save a power of $N$.  This is exactly what the two extra conditions in weak smallness allow us to do.  The first one lets us choose $\mu$ so that $W\mu \cap X_*(\bT_H)$ is empty, and the sums in the bound ($\star$) vanish, while the second gives us the required power saving in ($\star$).  As in the diagonal case, it is difficult to see how to proceed without one of these conditions.

\begin{remark}
The fact that one obtains such a clean equivalence between the conditions on $\bG_\infty$ and our seminorm inequalities is interesting, and is related to the spectra of symmetric varieties.  See Section \ref{sec8} and the references there, or \cite[Section 2]{BM} for a more general discussion of the link between amplification and these spectra.
\end{remark}

\begin{remark}
More generally, our methods establish the bound ($\star$) for the number of $\gamma \in \bG(\Q)$ that lie in the support of $\tau(v,\mu)$, and near a general reductive subgroup $\bH_\infty < \bG_\infty$.  Moreover, we can rule out `concentration near $\bH_\infty$ of Hecke eigenfunctions on $\bG(\Q) \backslash \bG(\A)$' exactly when $\bH_\infty$ is weakly small in $\bG_\infty$.  We hope that this will have applications to other problems in the asymptotics of automorphic forms, such as the Quantum Unique Ergodicity problem, bounds for higher-dimensional periods, and Kakeya norms of Maass forms.
\end{remark}

\begin{remark}
The results of \cite{BM} provide a partial converse to the above discussion.  In particular, it is proved there that if $\bG(\R)$ has a maximal compact subgroup defined over $\Q$, $\bG$ is anisotropic over $\Q$, and $\bG(\R)$ is not split, one can prove power growth of Maass forms by using the large number of Hecke returns.
\end{remark}

{\bf Acknowldegements:}  We would like to thank Farrell Brumley, Florian Herzig, Roman Holowinsky, Alireza Salehi Golsefidy, Steven Sam, Lior Silberman, and Akshay Venkatesh for helpful conversations.

\section{Notation}
\label{sec2}

We now introduce the notation used in the proof of Theorem \ref{main} in Section \ref{sec3}, including the subgroup $\bH$ that gives a maximal compact subgroup of $\bG(F_{v_0})$.

\subsection{Existence of a maximal compact subgroup over a number field}
\label{sec21}

\begin{lemma}
\label{Hdef}

There is a number field $F \subset F_1 \subset F_{v_0}$, and a connected subgroup $\bH < \bG$ defined over $F_1$, such that $\bH(F_{v_0})$ is a maximal connected compact subgroup of $\bG(F_{v_0})$.

\end{lemma}

\begin{proof}

We shall denote $v_0$ simply by $v$ throughout the proof.  Let $\g_F$ be the Lie algebra of $\bG$ over $F$.  We first prove the existence of $F \subset F_1 \subset F_v$ such that $\g_F \times F_v$ has a Cartan decomposition defined over $F_1$.

By \cite[Proposition 3.7]{Borel63}, there exists a Lie algebra $\g_F^0/F$ with an $F$-involution $\theta^0$ such that $\theta^0$ is a Cartan involution of $\g_F^0 \times F_v$, and there is an isomorphism $\iota : \g_F^0 \times F_v \to \g_F \times F_v$.  Choose bases $X_1, \ldots, X_r$ of $\g_F^0$ and $Y_1, \ldots, Y_r$ of $\g_F$, and define $a_{ij} \in F_v$ for $1 \le i, j \le r$ by

\bes
\iota(X_i) = \sum_{j=1}^r a_{ij} Y_j.
\ees
Let $R$ be the $F$-subalgebra of $F_v$ generated by $a_{ij}$.  We have a map $F[x_{ij}] \to R$ sending $x_{ij}$ to $a_{ij}$, and we let $I$ be its kernel so that $R \simeq F[x_{ij}] / I$.  Any homomorphism of $F$-algebras $\phi : R \to F_v$ gives a homomorphism $\psi : \g_F^0 \to \g_F \times F_v$ of Lie algebras over $F$ by

\bes
\psi( X_i) = \sum_{j=1}^r \phi(x_{ij}) Y_j.
\ees

If we let $V \subset F_v^{r^2}$ be the set defined by the equations in $I$, then homomorphisms $\phi : R \to F_v$ are in bijection with points on $V$.  There is a point $p_0 \in V$ corresponding to the natural embedding $R \to F_v$, and by Lemma \ref{algdense} we may find a point $p \in V$ arbitrarily close to $p_0$ whose coordinates are algebraic.  Let $F_1 \subset F_v$ be the field generated by these coordinates, and let $\g_{F_1}^0 = \g_F^0 \otimes F_1$ and $\g_{F_1} = \g_F \otimes F_1$.  Then $p$ gives a homomorphism $\psi : \g^0_{F_1} \to \g_{F_1}$.

If $p$ is sufficiently close to $p_0$ then $\psi$ will be an isomorphism.  If we let $\theta$ be the involution of $\g_{F_1}$ obtained by transferring $\theta^0$ via $\psi$, this implies that $\theta$ is a Cartan involution of $\g_{F_1} \times F_v$.  We let $\gk_{F_1} \subset \g_{F_1}$ be the fixed subspace of $\theta$, so that $\gk_{F_1} \times F_v$ is a maximal compactly embedded subalgebra of $\g_{F_1} \times F_v$.

$\theta$ induces an involution of $\bG^\text{ad} \times F_1$, and we let $\bH^\text{ad}$ be its fixed point subgroup.  Let $\bH < \bG \times F_1$ be the neutral component of the preimage of $\bH^\text{ad}$.  We have $\text{Lie}(\bH) = \gk_{F_1}$, and $\text{Lie}(\bH(F_v)) = \gk_{F_1} \times F_v$.  This implies that $\bH(F_v)$ is compact, hence connected by the result at the top of p. 277 of \cite{Bo2}, and that $\bH(F_v)$ is a maximal compact connected subgroup of $\bG(F_v)$.

\end{proof}

\begin{lemma}
\label{algdense}

Let $F \subset \R$ be a number field, and let $V \subset \R^n$ be an algebraic set defined by polynomials with coefficients in $F$.  Then the set of points in $V$ with algebraic coefficients is dense in $V$ with the real topology.

\end{lemma}

\begin{proof}

This follows from Tarski's theorem \cite{Ta} that the first-order theory of real-closed fields is complete, which implies in particular that a first order statement is true over $\R$ if and only if it is true over $\R \cap \overline{\Q}$.  Let $y \in V$, and let $B$ be a box around $y$ whose corners have rational coordinates.  Then the statement ``$V \cap B \neq \emptyset$'' can be expressed in the first-order language of ordered fields.  To see this, assume for simplicity that $V$ is defined by a single polynomial $f(x) = \sum_\gamma a_\gamma x^\gamma$ where $x = (x_1, \ldots, x_n)$.  Let $p_\gamma$ be a minimal polynomial for $a_\gamma$ over $\Q$, and let $q_\gamma, q_\gamma' \in \Q$ be such that $a_\gamma$ is the unique real root of $p_\gamma$ in $(q_\gamma, q'_\gamma)$.  Consider the first-order statement

\bes
\forall \gamma \forall b_\gamma : p_\gamma(b_\gamma) = 0, q_\gamma < b_\gamma < q_\gamma' \implies \exists x : x \in B, \sum b_\gamma x^\gamma = 0.
\ees
(Of course, this is a grammatically incorrect abbreviation of what one would actually write.)  This statement is true over $\R$.  Indeed, our choice of $p_\gamma$ and $q_\gamma, q_\gamma'$ imply that the only $b_\gamma$ that one quantifies over is $a_\gamma$, and then the existence statement is true because of our chosen $y$.  It follows that the statement is also true over $\R \cap \overline{\Q}$.  Over this field, one must also have $b_\gamma = a_\gamma$, and so one obtains algebraic points in $V \cap B$.

\end{proof}

\subsection{Algebraic groups}
\label{alggps}

Let $F_1$ and $\bH$ be as in Lemma \ref{Hdef}.  Let $\cO$ and $\cO_1$ be the integers of $F$ and $F_1$.  If $v$ is a finite place of $F$ we let $F_v$, $\cO_v$, and $\varpi_v$ be the completion, the ring of integers, and a uniformizer.  We let $k_v$ be the residue field, and denote its order by $q_v$.

Choose maximal tori $\bT < \bG$ and $\bT_H < \bH$, defined over $F$ and $F_1$ respectively, with $\bT_H < \bT$.  Fix an embedding $\rho : \bG \to SL_N$, and define $\gG$ and $\gT$ to be the schematic closures of $\bG$ and $\bT$ inside the group scheme $SL_N / \cO$.  We likewise define $\gH$ and $\gT_H$  to be the closures of $\bH$ and $\bT_H$ inside $SL_N/\cO_1$. These closures are group schemes over $\cO$ (resp. $\cO_1$), and all the inclusions between them over $F$ (resp. $F_1$) extend to closed embeddings over $\cO$ (resp. $\cO_1$).  By \cite[3.9]{T}, we may choose $D > 0$ such that all of these group schemes have connected reductive fibers over $\cO[1/D]$ or $\cO_1[1/D]$ respectively.

Let $X^*(\bT)$ and $X_*(\bT)$ denote the group of characters and cocharacters of $\bT$ over $\overline{F}$.  We define a norm on $X_*(\bT)$ that we shall use throughout the paper.  Let $\Delta$ be the set of roots of $\bT$ in $\bG$, and let $\Delta^+$ be a choice of positive roots.  Let $W$ be the Weyl group of $\bG$, and $\rho$ the half-sum of $\Delta^+$.  We define a norm $\| \cdot \|^*$ on $X^*(\bT)$ by

\bes
\| \mu \|^* = \underset{w \in W}{\max} \langle w \mu, \rho \rangle.
\ees
(Note that we have $\| \mu \|^* = \| -\mu \|^*$ because $\rho$ and $-\rho$ lie in the same Weyl orbit.)  We define a seminorm $\| \cdot \|^*_H$ on $X_*(\bT_H)$ in a similar way, which is a norm if $\bH$ is semisimple.

\subsection{Metrics}

For $\gamma \in \bG(F)$, let $\| \gamma \|_f$ be the LCM of the denominators of the norms of the matrix entries of $\rho(\gamma)$.  Fix a left-invariant Riemmanian metric on $\bG(F_{v_0})$.  Let $d( \cdot, \cdot)$ be the associated distance function.  We define $d(x,y) = \infty$ when $x$ and $y$ are in different connected components of $\bG(F_{v_0})$ with the topology of a real manifold.

\subsection{Adelic groups}
\label{sec2adelic}

Let $\A$ and $\A_f$ be the adeles and finite adeles of $F$.  We choose a compact subgroup $K = \prod_{v} K_v$ of $\bG(\A)$ such that $K_{v_0} = \bH(F_{1,w_0})$, where $w_0 | v_0$ is the place of $F_1$ corresponding to the embedding $F_1 \subset F_{v_0}$ of Lemma \ref{Hdef}, $K_v$ is maximal connected compact at the other infinite places, and $K_v = \rho^{-1}( SL_N(\cO_v))$ for $v < \infty$.  Because $\gG$ was closed in $SL_N / \cO[1/D]$, this implies that $K_v$ is the hyperspecial subgroup $\gG(\cO_v)$ for $v \nmid D \infty$.  For each $v < \infty$, let $dg_v$ be the Haar measure on $\bG(F_v)$ that assigns $K_v$ measure 1.  Choose a Haar measure $dg_\infty$ on $\bG(F_\infty)$, and let $dg = \otimes_{v} dg_v$.  All convolutions on $\bG$ will be defined with respect to these measures.  If $f \in C^\infty_0(\bG(\A))$, we define the operator $\pi(f)$ on $L^2(\bG(F) \backslash \bG(\A) )$ by

\be
\label{regrep}
[\pi(f)\phi](x) = \int_{\bG(\A)} \phi(xg) f(g) dg.
\ee
If $f \in C^\infty_0(\bG(\A))$, we define $f^*$ by $f^*(g) = \overline{f}(g^{-1})$, so that $\pi(f)$ and $\pi(f^*)$ are adjoints.  We define $Y = \bG(F) \backslash \bG(\A) / K$.  Choose compact sets $\Omega_Y \subset Y$ and $\Omega = \prod_v \Omega_v \subset \bG(\A)$ such that the projection of $\Omega$ to $Y$ contains $\Omega_Y$.  We assume that $\Omega_v = K_v$ for $v \nmid D \infty$.

\subsection{Hecke algebras}

Let $\cH$ be the convolution algebra of functions on $\prod_{v \nmid D \infty} \bG(F_v)$ that are compactly supported and bi-invariant under $\prod_{v \nmid D \infty} K_v$.  We identify $\cH$ with a subalgebra of $C^\infty_0(\bG(\A_f))$ in the natural way.  For each $v \nmid D \infty$, we let $\cH_v$ be the algebra of functions on $\bG(F_v)$ that are compactly supported and bi-invariant under $K_v$, which we identify with a subalgebra of $\cH$.

We let $\cP$ be the set of finite places $v$ of $F$ such that that $v \nmid D$, $\bT_v$ is split, and there is a split place $w|v$ of $F_1$ with $\bT_{H,w}$ split (so in particular $\bG_v$ and $\bH_w$ are also split).  If $v \in \cP$, our assumptions imply that $\gT \times \cO_v$ is a smooth closed subgroup scheme of $\gG \times \cO_v$ with reductive fibers, and \cite[5.1.33]{BT2} (combined with e.g. Proposition \ref{stdscheme} to show that $\gT \times \cO_v$ is a standard $\cO_v$-torus) then implies that the apartment of $\bT_v$ contains the point in the building of $\bG_v$ corresponding to $K_v$.  It follows that $\bG(F_v)$ has a Cartan decomposition with respect to $K_v$ and $\bT_v$.

\subsection{Lie algebras}

Let $\g$ be the real Lie algebra of $ \bG(F_\infty)$, and let $\g = \gk + \p$ be the Cartan decomposition associated to $K_\infty$.  Let $\ga \subset \p$ be a maximal abelian subspace.  We let $\Delta_\R$ be the roots of $\ga$ in $\g$, and let $\Delta_\R^+$ be a choice of positive roots.  We let $W_\R$ be the Weyl group of $\Delta_\R$.  For $\alpha \in \Delta_\R$, we let $m(\alpha)$ denote the dimension of the corresponding root space.  We denote the Killing form on $\g$ and $\g^*$ by $\langle \cdot, \cdot \rangle$.  We equip $\ga$ and $\ga^*$ with the norm obtained by restricting the Killing form.  For $\lambda \in \ga^*_\C$, we define

\be
\label{Dlambda}
D(\lambda) = \prod_{\alpha \in \Delta^+_\R} ( 1 + | \langle \alpha, \lambda \rangle| )^{m(\alpha)}.
\ee

\subsection{Spherical functions}

Let $\bG_\infty^0$ denote the connected component of the identity in $\bG(F_\infty)$.  If $\mu \in \ga^*_\C$, we define $\varphi_\mu$ to be the corresponding spherical function on $\bG_\infty^0$.  If $k \in C^\infty_0(\bG_\infty^0)$, we define its Harish-Chandra transform by

\bes
\widehat{k}(\mu) = \int_{\bG_\infty^0} k(g) \varphi_{-\mu}(g) dg_\infty.
\ees
If $k$ is $K_\infty$-biinvariant, we have the inversion formula

\bes
k(g) = \frac{1}{|W_\R|} \int_{\ga^*} \widehat{k}(\mu) \varphi_\mu(g) |c(\mu)|^{-2} d\mu
\ees
where $c(\mu)$ is Harish-Chandra's $c$-function; see \cite[Ch. II $\mathsection$3.3]{He2}.

\subsection{Maass forms}

Let $\psi \in L^2(Y)$ be an eigenfunction of the ring of invariant differential operators on $Y$ and the Hecke algebra $\cH$.  If $\cT \in \cH$, we define $\cT(\psi)$ by the equation $\pi(\cT) \psi = \cT(\psi) \psi$.  We assume that $\| \psi \|_2 = 1$.  We define the spectral parameter $\lambda \in \ga^*_\C / W_\R$ of $\psi$ to be the unique element such that $\psi$ and $\varphi_\lambda$ have the same eigenvalues under the invariant differential operators.  The Laplace eigenvalue of $\psi$ is given by $(\Delta + C_1(G) + \langle \lambda, \lambda \rangle) \psi = 0$ for some $C_1(G) \in \R$.  We have the trivial bound

\bes
\| \psi|_{\Omega_Y} \|_\infty \ll_{\Omega_Y} D(\lambda)^{1/2},
\ees
which is equivalent to (\ref{Sarnak}) when $\lambda$ lies in a regular cone in $\ga^*$, but is stronger for singular $\lambda$.

\section{Amplification}
\label{sec3}

This section contains the proof of Theorem \ref{main}.  In Sections \ref{sec32}-\ref{Ksmallest} we carry out preliminary steps, namely constructing the amplifier at infinite places, proving a pre-trace inequality for $\psi$, and deriving an estimate for Hecke returns from the results of Sections \ref{sec5} and \ref{sec6}.  In Section \ref{sec3amp} we prove Theorem \ref{main} by combining these ingredients with the amplifier at finite places constructed in Section \ref{sec7}.  In this section, all implied constants will be assumed to depend on $\bG$.

\subsection{The amplifier at infinite places}
\label{sec32}

Define $\ga^*_\text{un} = \{ \xi \in \ga^*_\C : W_\R \xi = W_\R \overline{\xi}, \| \Im \xi \| \le \| \rho \| \}$.  By \cite[Prop. 3.4]{DKV}, $\ga^*_\text{un}$ contains the spectral parameter of any spherical unitary representation of $\bG_\infty^0$.  We denote the component of $\xi \in \ga_\C^*$ in $\g_{v_0, \C}$ by $\xi_0$.  We shall construct bi-$K_\infty$-invariant functions $k_\xi, k_\xi^0 \in C^\infty_0( \bG^0_\infty)$ for all $\xi \in \ga^*_\text{un}$ with the following properties:

\begin{enumerate}

\item
\label{kxi0}
$k_\xi = k_\xi^0 * k_\xi^0$.

\item
\label{kxi1}
$| \widehat{k}^0_\xi(-\xi) | \ge 1$.

\item
\label{kxi2}
There is a fixed compact set $B \subset \bG^0_\infty$ such that $\text{supp}(k_\xi) \subset B$.

\item
\label{kxi3}
If $\xi_0$ is $(A,\sigma)$-balanced in $\g_{v_0}$, then $k_\xi(x) \ll_{A,\sigma} D(\xi) (1 + \| \xi_0 \|^\sigma d(x, K_{v_0}) )^{-1/2}$.

\end{enumerate}

\noindent
For this we first take a real function $h_0 \in C^\infty( \ga^* )$ of Paley-Wiener type, which we implicitly extend to a function on $\ga^*_\C$.  Let
\[
h_\xi^0(\nu) = \sum_{w \in W_\R} h_0( w\nu - \Re \xi),
\]
and let $k_\xi^0$ be the bi-$K_\infty$-invariant function on $\bG_\infty^0$ satisfying $\widehat{k}_\xi^0(-\mu) =h_\xi^0(\mu)$.  We define $k_\xi = k_\xi^0 * k_\xi^0$ and $h_\xi = ( h_\xi^0 )^2$, so that $\widehat{k}_\xi(-\mu) =h_\xi(\mu)$.  The existence of $B$ satisfying (\ref{kxi2}) follows from the Paley-Wiener theorem of \cite{Ga}, and conditions (\ref{kxi1}) and (\ref{kxi3}) are given by the following lemmas.

\begin{lemma}

If $\xi_0 \in \ga^*_\textup{un}$ is $(A,\sigma)$-balanced in $\g_{v_0}$, we have

\be
\label{kt}
k_\xi(x) \ll_{A,\sigma} D(\xi) (1 + \| \xi_0 \|^\sigma d(x, K_{v_0}) )^{-1/2}.
\ee

\end{lemma}

\begin{proof}

Let $\ga_{v_0} = \ga \cap \g_{v_0}$, and let $\g_{v_0} = \oplus \g_i$ be the decomposition of $\g_{v_0}$ into $\R$-simple factors.  If $\nu_0 \in \ga_{v_0}^*$, we let $\nu_{0,i}$ be the $\g_i$ component of $\nu_0$.  Let $B_0$ be the projection of $B$ to $\bG_{v_0}$.  We begin by showing that
\be
\label{phiv0bd}
\varphi_{\nu_0}(x_0) \ll (1 + \min \{ \| \nu_{0,i} \| \} d( x_0, K_{v_0} ) )^{-1/2}
\ee
for $\nu_0 \in \ga_{v_0}^*$ and $x_0 \in B_0$.  To do this, we may assume that $\bG$ is of adjoint type, so that $\bG_{v_0}$ decomposes as a product $\prod \bG_i$ of $\R$-almost simple groups.  We let $K_{v_0} = \prod K_i$ and $x_0 = (x_{0,i})_i$.  Applying the bound of \cite[Thm. 2]{BP} or \cite[Prop. 7.2]{MT} on each $\bG_i$ gives
\[
\varphi_{\nu_0}(x_0) \ll \prod (1 + \| \nu_{0,i} \| d( x_{0,i}, K_i) )^{-1/2},
\]
and taking the term in the product with $d( x_{0,i}, K_i)$ largest gives (\ref{phiv0bd}).

We now prove (\ref{kt}).  We may assume that $x \in B$.  Inverting the Harish-Chandra transform gives

\bes
k_\xi(x) = \frac{1}{|W_\R|} \int_{\ga^*} \varphi_\nu(x) h_\xi(-\nu) |c(\nu)|^{-2} d\nu.
\ees
If we apply (\ref{phiv0bd}) on $\bG(F_{v_0})$ and the trivial bound $| \varphi | \le 1$ at other places, this becomes

\bes
k_\xi(x) \ll \int_{\ga^*} (1 + \min \{ \| \nu_{0,i} \| \} d(x, K_{v_0}) )^{-1/2} h_\xi(-\nu) |c(\nu)|^{-2} d\nu.
\ees
The result now follows from the fact that $h_\xi$ is raplidly decaying away from the set $W_\R \Re \xi$, and that $|c(\nu)|^{-2} \ll D(\nu)$.

\end{proof}

\begin{lemma}

If $h_0$ is chosen correctly, then we have $| \widehat{k}^0_\xi(-\xi) | \ge 1$ for all $\xi \in \ga^*_\textup{un}$.

\end{lemma}

\begin{proof}

Let $0 < \delta < 1$.  Let $b \in C^\infty_0(\ga)$ be non-negative, supported in the $\delta$-ball around 0, and satisfy $\int b = 4$.  We further assume that $b = b_0 * b_0$ for some even real-valued $b_0$, which ensures that $\widehat{b}(\nu)$ is non-negative for $\nu \in \ga^*$.  We wish to show that if we choose $h_0 = \widehat{b}$ in our construction of $h_\xi$, then $| h^0_\xi(\xi) | \geqslant 1$.

We let $C > 0$ be a constant depending only on $G$ that may vary from line to line.  We start by showing that $\Re \widehat{b}(\nu) \geqslant -C\delta$ for all $\nu \in \ga^*_\C$ with $\| \Im \nu \| \leqslant \| \rho \|$.  We have
\begin{align*}
\widehat{b}(\nu) & = \int_\ga b(H) e^{-i \nu(H)} dH \\
& = \int_\ga b(H) [ e^{-i \Re \nu(H)} + ( e^{-i \nu(H)} - e^{-i \Re \nu(H)} ) ] dH \\
& = \widehat{b}(\Re \nu) + \int_\ga b(H)  e^{-i \Re \nu(H)} ( e^{\Im \nu(H)} - 1 ) dH.
\end{align*}
As $\| \Im \nu \| \leqslant \| \rho \|$, we have $| e^{\Im \nu(H)} - 1 | \leqslant C \delta$ for all $H \in \text{supp}(b)$, so that
\[
\Re \widehat{b}(\nu) \geqslant \widehat{b}(\Re \nu) - C\delta \geqslant -C\delta
\]
as required.

We now take $h_0 = \widehat{b}$, and construct $h_\xi^0$ as before.  If we choose $\delta$ small, we will have $\Re h_0(i \nu) \geqslant 2$ for all $\nu \in \ga^*$ with $\| \nu \| \leqslant \| \rho \|$, and moreover
\begin{align*}
\Re h_\xi^0(\xi) & = \sum_{w \in W_\R} \Re h_0(w\xi - \Re \xi) \\
& \geqslant \Re h_0(\xi- \Re \xi) - C \delta \geqslant 1.
\end{align*}
This implies $| h^0_\xi(\xi) | \geqslant 1$ as required.

\end{proof}

\subsection{A pre-trace inequality}

We shall use the following amplification inequality, which is often referred to as a pre-trace inequality.

\begin{lemma}
\label{ampineq}

If $\omega \in C_c(\bG(\A))$ and $f \in L^2( \bG(F) \backslash \bG(\A))$ satisfies $\| f \|_2 = 1$, we have

\bes
| [\pi(\omega) f](x)|^2 \le \sum_{\gamma \in \bG(F)} (\omega \omega^*)(x^{-1} \gamma x).
\ees

\end{lemma}

\begin{proof}

We have 

\bes
[\pi( \omega) f](x) = \int_{\bG(\A)} f(g) \omega(x^{-1} g) dg.
\ees
Folding the sum over $\bG(F)$ gives

\bes
[\pi( \omega) f](x) = \int_{ \bG(F) \backslash \bG(\A)} f(g) \sum_{\gamma \in \bG(F)} \omega(x^{-1} \gamma g) dg.
\ees
If we apply Cauchy-Schwartz and expanding the square, we obtain

\begin{align*}
| [\pi( \omega) f](x) |^2 & \le \int_{ \bG(F) \backslash \bG(\A)} \bigg| \sum_{\gamma \in \bG(F)} \omega(x^{-1} \gamma g) \bigg|^2 dg \\
& = \int_{ \bG(F) \backslash \bG(\A)}  \sum_{\gamma_1, \gamma_2 \in \bG(F)} \omega(x^{-1} \gamma_1 g) \overline{\omega(x^{-1} \gamma_2 g)} dg.
\end{align*}
Unfolding again gives

\be
\label{amp1}
| [\pi( \omega) f](x) |^2 \le \sum_{\gamma \in \bG(F)} \int_{\bG(\A)} \omega(x^{-1} \gamma g) \overline{\omega(x^{-1} g)} dg.
\ee
Because $\overline{\omega(x^{-1} g)} = \omega^*(g^{-1} x)$, we have

\begin{align*}
\int_{\bG(\A)} \omega(x^{-1} \gamma g) \overline{\omega(x^{-1} g)} dg & = \int_{\bG(\A)} \omega(x^{-1} \gamma g) \omega^*(g^{-1} x)  dg \\
& = \omega \omega^* ( x^{-1} \gamma x).
\end{align*}
Inserting this into (\ref{amp1}) completes the proof.

\end{proof}

\subsection{Estimating Hecke returns}
\label{Ksmallest}

We now apply the results of Sections \ref{sec5} and \ref{sec6} to derive the bound on Hecke returns that we will use.  Let $x \in \Omega_Y$ be the point at which we want to bound $\psi$.  Let $B \subset \bG(F_\infty)$ be the compact set from Section \ref{sec32} containing the support of $k_\xi$ for all $\xi \in \ga^*_\text{un}$, and for $\delta > 0$ let $B_\delta = \{ g \in B : d( g, K_{v_0}) < \delta \}$.  If $v \neq w$ and $\mu, \nu \in X_*(\bT)$, we define

\[
\cM(v,w,\mu,\nu) = \{ \gamma \in \bG(F) : x^{-1} \gamma x \in B_\delta \cdot (K_v \mu(\varpi_v)K_v) \cdot (K_w \nu(\varpi_w)K_w) \cdot K^{v,w} \}.
\]
Roughly speaking, $\cM(v,w,\mu,\nu)$ is the set of rational elements in the support of the Hecke operator $1_{K_v \mu(\varpi_v)K_v} 1_{K_w \nu(\varpi_w)K_w}$ that move $x$ by at most $\delta$.   We have the following bound for $\# \cM(v,w,\mu,\nu)$.

\begin{prop}
\label{returns}

There exist $C, M > 0$ with the following property.  Let $X > 0$ be given.  There exists $\cQ \subset \cP$ with $\#(\cP - \cQ) \ll \log X$ such that if $v, w \in \cQ$ and $\mu, \nu \in X_*(\bT)$ satisfy $q_v^{\| \mu \|^*}, q_w^{\| \nu \|^*} \le X$ and $\delta < C X^{-M}$, then
\[
\# \cM(v,w,\mu,\nu) \ll \sum_{\alpha \in \Lambda_\mu} q_v^{ 2 \| \alpha \|_H^*} \sum_{\beta \in \Lambda_\nu} q_w^{ 2 \| \beta \|_H^*},
\]
where $\Lambda_\mu = W\mu \cap X_*(\bT_H)$ and likewise for $\Lambda_\nu$.  Note that $\cQ$ depends on $x$.

\end{prop}

The natural first step in the proof is to apply Proposition \ref{diophantine}.  However, before doing this we must bound the heights of elements in $\cM(v,w,\mu,\nu)$, which lets us compare $\cM(v,w,\mu,\nu)$ with the set denoted $\cM(x,\delta,T)$ in Section \ref{sec5}.

\begin{lemma}
\label{height}

There is a constant $C > 0$ such that if $v \in \cP$, $\mu \in X_*(\bT)$ and $g \in K_v \mu(\varpi_v)K_v$, then $\| g \|_v \le q_v^{C \| \mu \|^*}$.

\end{lemma}

\begin{proof}

Consider $\rho$ as a representation of $\bG$, and let $\Omega \subset X^*(\bT)$ be the multiset of weights of this representation. It follows that $\rho(\mu(\varpi_v))$ is semisimple with eigenvalues $\{ \varpi_v^{\langle \omega, \mu \rangle } : \omega \in \Omega \}$. By \cite[Lemma 2.17]{ST}, there is $x \in GL_d(\cO_v)$ such that $x\rho(\bT_v)x^{-1}$ is diagonal, and so if we define $A = \max \{-\langle \omega, \mu \rangle : \omega \in \Omega \}$ then we have $\| \rho(\mu(\varpi_v) ) \|_v = \| x \rho(\mu(\varpi_v))x^{-1} \|_v = q_v^A$. Because $\rho(K_v) \subset SL_N(\cO_v)$, the same holds for $g$.  The observation that $A \le C \| \mu \|^*$ for $C > 0$ depending only on $\bG$ finishes the proof.

\end{proof}
 
\begin{cor}
\label{height2}

There are $C, M > 0$ such that if $v, w \in \cP$ and $\mu, \nu \in X_*(\bT)$ satisfy $q_v^{\| \mu \|^*}, q_w^{\| \nu \|^*} \le X$ and $(x^{-1} \gamma x)_f \in (K_v \mu(\varpi_v)K_v) \cdot (K_w \nu(\varpi_w)K_w) \cdot K^{v,w}$, then $\| \gamma \|_f \le C X^M$.

\end{cor}
 
 \begin{proof}
 
 We have $\| x^{-1} \gamma x \|_f \le C_1 X^{A_2}$ by Lemma \ref{height}, and because $x \in \Omega$ the same is true for $\| \gamma \|_f$ (possibly with a different $C_1$).
 
 \end{proof}

We may now apply Proposition \ref{diophantine}, which gives the following.
 
 \begin{cor}
 
 There exist $C, M > 0$ with the following property.  If $X > 0$, there exist $Q \in \N$, an extension $F_2/F_1$, and $y \in \gG( \cO_2 [1/DQ] )$ with the following properties:  
 
 \begin{itemize}
 
 \item $Q$ has $\ll \log X$ prime factors
 
 \item $|F_2 : F| \le C_2(\bG)$
 
 \item $F_2/F$ is Galois and unramified outside $Q$.
 
 \item If we define $\displaystyle \bL = \bigcap_{\sigma \in \textup{Gal} (F_2/F)} (y \bH y^{-1})^\sigma$ then $\cM(v,w,\mu,\nu) \subset \bL(F)$ provided $v, w \in \cP$ and $\mu, \nu \in X_*(\bT)$ satisfy $q_v^{\| \mu \|^*}, q_w^{\| \nu \|^*} \le X$, and $\delta < C X^{-M}$.

\end{itemize}
 
 \end{cor}
 
 \begin{proof}
 
 Apply Proposition \ref{diophantine} with the data $(\bG, \bH, B_1, B_2)$ chosen to be $(\bG, \bH, B, \Omega_\infty)$, $T = C X^M$ with $C,M > 0$ as in Corollary \ref{height2}, and our chosen embedding $\rho$.  Let $C, M > 0$, $Q \in \N$, $F_2/F_1$, and $y \in \bG( F_2 )$ be the data produced.  Because $\rho(y) \in SL_N( \cO_2[1/DQ])$, and $\rho$ was a closed embedding, this implies that $y \in \gG( \cO_2 [1/DQ] )$.  The first three conditions above clearly hold.  For the last one, Corollary \ref{height2} implies that the set $\cM(x, \delta, C_2 X^{A_2} )$ defined in Section \ref{sec5} contains $\cM(v,w,\mu,\nu)$, and so the last condition holds if $\delta < C X^{-M}$.
 
 \end{proof}

We wish to show that Proposition \ref{returns} holds with this $C$ and $M$, and $\cQ \subset \cP$ the set of  places not dividing $Q$.  We note that $\#(\cP - \cQ) \ll \log X$.  We now let $v, w, \mu, \nu, X,$ and $\delta$ be as in Proposition \ref{returns}, and show that the required bound for $\# \cM(v,w,\mu,\nu)$ holds.  The following lemma reduces this problem to a local one.

\begin{lemma}
\label{heckelocal}

If we define
\[
\cL(v, \mu) = (\bL(F_v) K_v \cap K_v \mu(\varpi_v) K_v) / K_v,
\]
and likewise for $\cL(w,\nu)$, then we have $\# \cM(v,w,\mu,\nu) \ll \# \cL(v,\mu) \# \cL(w,\nu)$.

\end{lemma}

\begin{proof}

Because $x \in \Omega$, our assumptions on $\Omega$ and $\cP$ imply that $x_v \in K_v$ and $x_w \in K_w$.  It follows that if $\gamma \in \cM(v,w,\mu,\nu)$, then $\gamma_v \in \bL(F_v) \cap K_v \mu(\varpi_v) K_v$ and $\gamma_w \in \bL(F_w) \cap K_w \nu(\varpi_w) K_w$.  This gives a map

\bes
\cM(v,w,\mu,\nu) \rightarrow \cL(v,\nu) \times \cL(w,\nu).
\ees
We shall show that the fibers of this map have bounded size.  Suppose $\gamma_1$ and $\gamma_2 \in \cM(v,w,\mu,\nu)$ lie in the same pair of cosets $(g_v K_v, g_w K_w)$.  Then $\gamma_1^{-1} \gamma_2$ must lie in a compact set $C \subset \bG(\A)$ depending only on $K_f$ and $\Omega$, and the result now follows from the fact that $\bG(F) \cap C$ is finite.

\end{proof}

It therefore suffices to prove that $\# \cL(v,\mu) \ll \sum_{\alpha \in \Lambda_\mu} q_v^{ 2 \| \alpha \|_H^*}$ for $v \in \cQ$ and $\mu \in X_*(\bT)$.  We now use $w$ to denote the place of $F_1$ above $v$ in the definition of $\cP$.  The set $\cL(v,\mu)$ is similar to the one denoted $\cL(\mu)$ in Section \ref{sec6}, and they become the same once we enlarge $\bL_v$ to the group $\bL'$ defined below.

\begin{lemma}

Let $w' | w$ be a place of $F_2$, and let $\bL'$ be the subgroup of $\bG \times F_{1,w}$ defined by
\[
\bL' = \bigcap_{\sigma \in \textup{Gal}(F_{2,w'} / F_{1,w}) } (y_{w'} \bH y_{w'}^{-1})^\sigma.
\]
We then have $\bL(F_v) \subset \bL'(F_{1,w})$ under the identification $\bG(F_v) \simeq \bG(F_{1,w})$.

\end{lemma}

\begin{proof}

By the definition of $\bL$ as a scheme theoretic intersection, we have
\[
\bL(F_v) = \bG(F_v) \cap (y \bH y^{-1})( F_2 \otimes_F F_v) \subset \bG( F_2 \times_F F_v),
\]
and likewise
\[
\quad \bL'(F_{1,w}) = \bG(F_{1,w}) \cap (y \bH y^{-1})( F_{2,w'}).
\]
Consider the $F_v$-algebras $F_v \subset F_1 \times_F F_v \subset F_2 \times_F F_v$.  Let $\pi_w$ be the projection $F_1 \times_F F_v \to F_{1,w}$, which realizes the isomorphism $F_v \simeq F_{1,w}$ when restricted to $F_v$.  $\pi_w$ extends to a projection $F_2 \times_F F_v \to F_2 \times_{F_1} F_{1,w}$ after tensoring with $F_2$ over $F_1$.  Applying $\pi_w$ to $\bL(F_v)$ gives
\begin{align*}
\pi_w( \bL(F_v)) & \subset \pi_w( \bG(F_v) ) \cap \pi_w( (y \bH y^{-1})( F_2 \otimes_F F_v) ) \\
& = \bG(F_{1,w}) \cap (y \bH y^{-1})( F_2 \times_{F_1} F_{1,w}) \\
& \subset \bG(F_{1,w}) \cap (y \bH y^{-1})( F_{2,w'}) = \bL'(F_{1,w})
\end{align*}
as required.

\end{proof}

If we define
\[
\cL'(w,\mu) = ( \bL'(F_{1,w}) K_w \cap K_w \mu(\varpi_w) K_w) / K_w,
\]
then we may apply Proposition \ref{buildingcount} to the field $F_{1,w}$, the groups $\gG \times \cO_{1,w}$ and $\gH \times \cO_{1,w}$ with tori $\gT \times \cO_{1,w}$ and $\gT_H \times \cO_{1,w}$, and $y_{w'} \in \gG( \cO_{2,w'})$ to obtain
\be
\label{apply51}
\# \cL(v,\mu) \le \# \cL'(w, \mu) \ll \sum_{\alpha \in \Lambda_\mu} q_v^{ 2 \| \alpha \|_H^*}.
\ee
This completes the proof of Proposition \ref{returns}, once we show that the complexity of $\text{Ad}( \gH \times k_w)$ in $GL(\g_\cO \times k_w) \subset \A^{ (\dim \bG)^2 }$ in the sense of Section \ref{sec:complexity} is bounded uniformly in $v$, so that the implied constant in (\ref{apply51}) is also independent of $v$.  Here $\g_\cO$ denotes the Lie algebra of $\gG$, which is a projective $\cO[1/D]$-module.  To do this, recall that we have fixed a closed embedding of $\gG$ in $SL_N / \cO[1/D]$, and there is a globally defined adjoint map $\gG \to GL(\g_\cO)$ of group schemes over $\cO[1/D]$.  It follows that the adjoint maps on the fibers over $k_w$ have bounded complexity for all $w$, and because the complexity of $\gH \times k_w$ in $\gG \times k_w$ is also bounded, Lemma \ref{composecomplex} gives the desired complexity bound.

\subsection{Conclusion}
\label{sec3amp}

We now combine these ingredients to prove Theorem \ref{main}.  We recall the Maass form $\psi$ with spectral parameter $\lambda$, and the point $x \in \Omega_Y$ at which we want to bound $\psi$.  We let $k_\lambda = k_\lambda^0 * k_\lambda^0$ be as in Section \ref{sec32}.

We next define the finite part of our amplifier.  If $v \in \cP$ and $\mu \in X_*(\bT)$, we define $\tau(v,\mu) \in \cH_v$ to be the function supported on $K_v \mu(\varpi_v) K_v$ and taking the value $q_v^{-\| \mu \|^*}$ there.  It follows from Corollary \ref{Gy1} that $\tau(v, \mu)$ is approximately $L^2$-normalized.  For $\kappa \ge 0$ we define the truncated Hecke algebra $\cH_v^{\le \kappa}$ by
\[
\cH_v^{\le \kappa} = \text{span}_\C \{ \tau(v, \mu) : \| \mu \|^* \le \kappa \}.
\]
Let $\mu \in X_*(\bT)$ and $N > 0$ to be specified later.  Define $\cP_N = \{ v \in \cP : N/2 < q_v < N \}$.  Apply Proposition \ref{satake} to $\mu$ to obtain $\kappa, C > 0$ and $\tau_v = \tau_v^0 + \sigma_v$ for every $v \in \cP$ with the properties:

\begin{itemize}

\item $\tau_v(\psi) = 1$ for all $v \in \cP$.
\item $\tau_v$, $\tau_v^0$, $\sigma_v$, and $\tau_v \tau_v^*$ all lie in $\cH_v^{\le \kappa}$.
\item $\| \tau_v \|, \| \tau_v \tau_v^* \|_2 \ll 1$ and $\| \sigma_v \|_2 \ll q_v^{-1}$.

\item $\displaystyle \tau_v^0 = \sum_{1 \le |j| \le |W|} \sum_{ \| \lambda \|^* < C } c(v,j,\lambda) \tau(v, j \mu + \lambda)$ with $c(v,j,\lambda) \ll 1$.

\end{itemize}
Apply Proposition \ref{returns} with $X = N^{\kappa}$.  Let $\cQ$ be the set of places produced, and let $\cQ_N = \cQ \cap \cP_N$.   The finite part of our amplifier is then $\cT = \sum_{v \in \cQ_N} \tau_v$.

Applying the pre-trace inequality of Lemma \ref{ampineq} to $\psi$ and the test function $k_\lambda^0 \cT$, and using $| \widehat{k}^0_\lambda(-\lambda) | \ge 1$, gives

\[
| \cT(\psi) \psi(x)|^2 \le \sum_{\gamma \in \bG(F)} (\cT \cT^*)(x^{-1} \gamma x) k_\lambda(x^{-1} \gamma x).
\]
We break up the sum depending on whether $x^{-1} \gamma x$ is within $\delta$ of $K_{v_0}$, where $\delta > 0$ will be chosen later.  If $B$ and $B_\delta$ are as in Section \ref{Ksmallest}, we have

\begin{multline}
\label{nearfar}
| \cT(\psi) \psi(x)|^2 \le \sum_{\gamma \in \bG(F)} (\cT \cT^*)(x^{-1} \gamma x) k_\lambda(x^{-1} \gamma x) 1_{B_\delta}(x^{-1} \gamma x) \\
+ \sum_{\gamma \in \bG(F)} (\cT \cT^*)(x^{-1} \gamma x) k_\lambda(x^{-1} \gamma x) (1_B - 1_{B_\delta})(x^{-1} \gamma x).
\end{multline}
We bound the second sum by applying the decay in $k_\lambda$ from (\ref{kt}), and estimating the number of terms trivially.  Because $x \in \Omega$, we have

\bes
\# \{ \gamma \in \bG(F) : (x^{-1} \gamma x)_\infty \in B, (x^{-1} \gamma x)_f \in \text{supp}(\cT \cT^*) \} \ll N^{A_1}
\ees
for some $A_1 = A_1(\bG)$.  When combined with equation (\ref{kt}) this gives

\be
\label{farest}
\sum_{\gamma \in \bG(F)} (\cT \cT^*)(x^{-1} \gamma x) k_\lambda(x^{-1} \gamma x) (1_B - 1_{B_\delta})(x^{-1} \gamma x) \ll D(\lambda) (1 + \| \lambda_0 \|^\sigma \delta)^{-1/2} N^{A_1}.
\ee

We apply the trivial bound $| k_\lambda | \ll D(\lambda)$ to the first sum in (\ref{nearfar}), so that it suffices to estimate
\be
\label{Heckesum}
\sum_{\gamma \in \bG(F)} (\cT \cT^*)(x^{-1} \gamma x) 1_{B_\delta}(x^{-1} \gamma x).
\ee
This requires our bounds for Hecke returns, and in particular the following consequence of Proposition \ref{returns}.  We note that if $v, w \in \cQ_N$ and $\| \mu \|^*, \| \nu \|^* \le \kappa$, then $q_v^{\| \mu \|^*}, q_w^{\| \nu \|^*} \le X$.  Proposition \ref{returns} therefore gives $C, M > 0$ such that for such $v$, $w$, $\mu$, and $\nu$, and $\delta < C N^{-M}$, we have
\be
\label{returns2}
\# \cM(v,w,\mu,\nu) \ll \sum_{\alpha \in \Lambda_\mu} q_v^{ 2 \| \alpha \|_H^*} \sum_{\beta \in \Lambda_\nu} q_w^{ 2 \| \beta \|_H^*}.
\ee
We assume that $\delta < C N^{-M}$ from now on.  If we combine this with the quasi-splitness assumption in ($\mathsf{WS}$), we obtain the following.

\begin{cor}
\label{L2est}

Let $v, w \in \cQ_N$.  If $\omega_v \in \cH_v^{\le \kappa}$ and $\omega_w \in \cH_w^{\le \kappa}$ then

\[
\sum_{\gamma \in \bG(F)} \omega_v \omega_w(x^{-1} \gamma x) 1_{B_\delta}(x^{-1} \gamma x) \ll \| \omega_v \|_2 \| \omega_w \|_2.
\]

\end{cor}

\begin{proof}

If $\mu, \nu \in X_*(\bT)^{\le \kappa}$, Corollary \ref{realsmall} and the quasi-splitness of $\bG(F_{v_0})$ imply that we have $2 \| \alpha \|_H^* \le \| \alpha \|^* = \| \mu \|^*$ for all $\alpha \in \Lambda_\mu$, and likewise for $\beta \in \Lambda_\nu$.  Equation (\ref{returns2}) then implies that $\# \cM(v,w,\mu,\nu) \ll q_v^{\| \mu \|^*} q_w^{ \| \nu \|^*}$.

If we expand $\omega_v$ in terms of the basic Hecke operators $\tau(v,\mu)$, Corollary \ref{Gy1} implies that all the coefficients must be $\ll \| \omega_v \|_2$, and likewise for $\omega_w$.  The corollary now follows from $\# \cM(v,w,\mu,\nu) \ll q_v^{\| \mu \|^*} q_w^{ \| \nu \|^*}$.

\end{proof}

We next expand out $\cT \cT^*$ in (\ref{Heckesum}), and examine the contribution from the various terms.  The diagonal terms $\tau_v \tau_v^*$ can be estimated by combining the bound $\| \tau_v \tau_v^* \|_2 \ll 1$ with Corollary \ref{L2est}, and make a total contribution of $O(N)$.  Each off-diagonal term contains $\tau_v^0 (\tau_w^0)^*$, as well as three other terms involving $\sigma_v$ and $\sigma_w$.  Because $\| \sigma_v \|_2, \| \sigma_w \|_2 \ll q_v^{-1} \le 2/N$, the contribution from the terms containing a $\sigma$ is $O(N)$ by Corollary \ref{L2est}.

It remains to estimate the terms containing $\tau_v^0 (\tau_w^0)^*$ with $v \neq w$, which requires the second part of condition $(\mathsf{WS})$.  We substitute the formula
\[
\tau_v^0 = \sum_{1 \le |j| \le |W|} \sum_{ \| \lambda \|^* < C } c(j,\lambda) \tau(j \mu + \lambda),
\]
and likewise for $\tau_w^0$, and only consider the contributions from the two terms with $j = 1$ as the others may be treated in the same way.

We first consider the case where $\dim \bT_H < \dim \bT$, so that $W X_*(\bT_H) \subset X_*(\bT)$ is contained in a finite union of lower dimensional subspaces.  We may choose $\mu$ such that none of the cocharacters $\pm \mu + \lambda$ with $\| \lambda \|^* \le C$ lie in $W X_*(\bT_H)$.  This implies that when we use (\ref{returns2}) to estimate the contributions
\[
\sum_{\gamma \in \bG(F)} \tau(v,\mu + \lambda_1) \tau(w, \mu + \lambda_2)^*(x^{-1} \gamma x) 1_{B_\delta}(x^{-1} \gamma x)
\]
from these terms to (\ref{Heckesum}), the sums on the right hand side are empty and the contribution is zero.  Combining the contributions from the other terms in the expansion of $\cT \cT^*$ gives
\[
\sum_{\gamma \in \bG(F)} (\cT \cT^*)(x^{-1} \gamma x) 1_{B_\delta}(x^{-1} \gamma x) \ll N.
\]

The other case is where $\dim \bT_H = \dim \bT$, but there is $\nu \in X_*(\bT)$ such that
\be
\label{normineq}
\| \nu \|^* > 2 \; \underset{w \in W}{\max} \, \| w \nu \|^*_H.
\ee
Because $\| \cdot \|^*$ and $\| \cdot \|^*_H$ are piecewise linear, we may again choose $\mu$ such that the inequality (\ref{normineq}) holds for all $\pm \mu + \lambda$ with $\| \lambda \|^* \le C$.  Applying (\ref{returns2}), we see that there is $\eta > 0$ such that
\[
\sum_{\gamma \in \bG(F)} \tau(v,\mu + \lambda_1) \tau(w, \mu + \lambda_2)^*(x^{-1} \gamma x) 1_{B_\delta}(x^{-1} \gamma x) \ll N^{-\eta}.
\]
Summing over $v$ and $w$ and including the other terms in the expansion of $\cT \cT^*$ gives
\[
\sum_{\gamma \in \bG(F)} (\cT \cT^*)(x^{-1} \gamma x) 1_{B_\delta}(x^{-1} \gamma x) \ll N^{2-\eta}.
\]

In either case, we have
\[
\sum_{\gamma \in \bG(F)} (\cT \cT^*)(x^{-1} \gamma x) k_\lambda(x^{-1} \gamma x) 1_{B_\delta}(x^{-1} \gamma x) \ll D(\lambda) N^{2-\eta}.
\]
Combining this with (\ref{nearfar}) and (\ref{farest}) gives

\bes
| \cT(\psi) \psi(x)|^2 \ll D(\lambda) N^{2 - \eta} + D(\lambda) (1 + \| \lambda_0 \|^\sigma \delta)^{-1/2} N^{A_1}.
\ees
Applying $\tau_v(\psi) = 1$ and the bound $|\cQ_N| \gg_{\epsilon} N^{1 - \epsilon}$ gives

\be
\label{psiest}
|\psi(x)|^2 \ll_{\epsilon} D(\lambda) N^{- \eta + \epsilon} + D(\lambda) (1 + \| \lambda_0 \|^\sigma \delta)^{-1/2} N^{A_1}.
\ee
If we choose $\delta = \| \lambda_0 \|^{-\sigma/2}$ and $N \sim \| \lambda_0 \|^c$ for $c = c(\bG) > 0$ sufficiently small, then we have $\delta \le C N^{-M}$, and the inequality (\ref{psiest}) becomes $\psi(x) \ll D(\lambda)^{1/2} (1 + \| \lambda_0 \|)^{- \epsilon}$ for some $\epsilon = \epsilon(\bG, \sigma)$, which completes the proof.

\section{Diophantine approximation}
\label{sec5}

This section establishes the existence of the group $\bL$ used in $\mathsection$\ref{Ksmallest}.  It may be read independently from the rest of the paper, and draws on unpublished notes of Peter Sarnak and Akshay Venkatesh which we thank the authors for sharing with us.  We note that the main result of this section controls returns to any subvariety of $\bG$, not just subgroups.

\subsection{Notation}

Let $F$ be a number field.  Let $E$ be an extension of $F$, and fix infinite places $w_0 | v_0$ of $E$ and $F$ respectively.  Let $\bG/F$ be an affine algebraic group, and let $\bH \subset \bG$ be a subvariety defined over $E$.  We let $\rho : \bG \rightarrow SL_N$ be an $F$-embedding.  For $\gamma \in \bG(F)$, define $\| \gamma \|_f$ to be the LCM of the denominators of the norms of the entries of $\rho(\gamma)$.  We let $d( \cdot, \cdot)$ be the distance function on $\bG(F_{v_0})$ obtained from the standard Euclidean distance on $\A^{N^2}(F_{v_0})$.

\subsection{Main result and method of proof}

Fix a compact set $B_1 \subset \bG(F_{\infty})$, and let $H_{v_0} = \bH(E_{w_0}) \cap \bG(F_{v_0})$.  For $x \in \bG(F_{\infty})$, $T > 2$, and $\delta > 0$, we define

\bes
\cM(x, \delta, T) = \{ \gamma \in \bG(F) : \| \gamma \|_f \le T, \quad x^{-1} \gamma x \in B_1, \quad d(x^{-1} \gamma x, H_{v_0}) < \delta \}.
\ees
In other words, $\cM(x, \delta, T)$ is roughly the set of $\gamma \in \bG(F)$ that are within $\delta$ of $x H_{v_0} x^{-1}$, lie in $x B_1 x^{-1}$, and have denominators bounded by $T$.  Our main result is that, if $\delta$ is small, all such $\gamma$ lie in a variety $\bL / F$ that is stably conjugate to a subvariety of $\bH$.

\begin{prop}
\label{diophantine}

Let $B_2 \subset \bG(F_{\infty})$ be compact.  There exists $M = M(\bG, \bH) > 0$ and $C = C(B_1, B_2) > 0$ such that if $x \in B_2$ and $T > 2$, then there exist $Q \in \N$, an extension $E'/E$, and $y \in \bG( E')$ with the following properties.

\begin{enumerate}

\item $Q$ has $\ll_{B_1, B_2} \log T$ prime factors.

\item $|E' : E| \ll_{\bG, \bH} 1$, and $E'/F$ is Galois and unramified outside $Q$.

\item
\label{yintegral}
The denominators of $\rho(y)$ all divide $Q$.

\item If we define

\bes
\bL = \bigcap_{\sigma \in \textup{Gal} (E'/F)} (y \bH y^{-1})^\sigma,
\ees
then $\cM(x, \delta, T) \subset \bL(F)$ provided $\delta \le C T^{-M}$.

\end{enumerate}
The data $E'$, $Q$, and $y$ depend on $x$, $\delta$, and $T$.

\end{prop}

We may illustrate the proof of Proposition \ref{diophantine} using a toy example.  Suppose one has a collection of points in $\Z^2$ lying in a fixed ball.  One wants to show that if all the points lie close to a line, then they actually lie on a line.  One does this by considering a triangle formed by three of the points; the area must be small, by the near-collinearity, but it must also be a half-integer.  Therefore any three points are collinear, so they all must be.

We adapt this argument as follows.  For any subset $S \subset \cM(x, \delta, T)$, define the variety $\bX_S/E \subset \bG$ by $\bX_S = \{ g \in \bG : S \subset g \bH g^{-1} \}$, and let $\bX = \bX_{\cM(x, \delta, T)}$.  In terms of our toy model, one may think of $\bX_S$ as the set of lines containing the points in $S$.  Proving Proposition \ref{diophantine} is essentially the same as showing that $\bX \neq \emptyset$, because if $y \in \bX(\overline{F})$ then $\cM(x, \delta, T) \subset y \bH y^{-1}$, and we may descend $y \bH y^{-1}$ to $\bL$.

The main step in showing that $\bX \neq \emptyset$ is to prove that, if $S$ is finite, $\bX_S \neq \emptyset$ if $\delta$ is sufficiently small in terms of $T$ and $|S|$.  This is analogous to the argument with the area of a triangle.  The next step is to find $S$ of controlled size such that $\bX = \bX_S$, which is analogous to the fact that if any three points from our collection are collinear, they all must be.  The final step is to show that if one has a nonempty algebraic set of bounded complexity, then it must contain an algebraic point of bounded complexity, which gives the required control on $y$.

\subsection{Nonemptiness of $\bX_S$}

We now show that if $S \subset \cM(x, \delta, T)$ is finite, and $\delta$ is sufficiently small in terms of $T$ and $|S|$, then $\bX_S( \overline{E})$ is nonempty.

\begin{prop}
\label{approx}

Let $l > 0$.  There are $C = C(B_1, B_2, l) > 0$ and $M = M(\bG, \bH, l) > 0$ such that if $x \in B_2$, $\delta < C T^{-M}$, and $S \subset \cM(x, \delta, T)$ has $|S| = l$, then $\bX_S( \overline{E}) \neq \emptyset$.

\end{prop}

\begin{proof}

Let $\pi: \bH^l \times \bG \rightarrow \bG^l$ be the map over $E$ given by

\bes
\pi(h_1, \ldots, h_l, g ) = ( g h_1 g^{-1}, \ldots, g h_l g^{-1}).
\ees
Let $S = \{ \gamma_1, \ldots, \gamma_l \}$, and define $z = (\gamma_1, \ldots, \gamma_l)$.  It may be seen that $\bX_S$ is the projection of $\pi^{-1}(z)$ to $\bG$, and so it suffices to show that $\pi^{-1}(z) \neq \emptyset$.

To explain why this should be true, suppose that the image of $\pi$ is a closed subvariety $\bI \subset \bG^l$.  In this case, we know that $z$ is a rational point of $\bG^l$ of bounded height, and the assumption that $\gamma_i \in \cM(x, \delta, T)$ implies that $z$ is close to $\bI(E_{w_0})$.  It is then easy to show that $z \in \bI(E)$ as required.

If $\bI$ is a constructible set rather than a subvariety, then this argument does not work - there may be rational points that lie in the closure of $\bI(E_{w_0})$ but not in $\bI(E)$.  However, if we assume that $z$ is close not just to $\bI(E_{w_0})$, but to $\pi(K_X)$ for some compact $K_X \subset \bX(E_{w_0})$, then the principle still holds, and is stated precisely as Proposition \ref{abstractapprox}.

We wish to apply Proposition \ref{abstractapprox} with $\bX = \bH^l \times \bG$, $\bY = \bG^l$, and $f = \pi$.  We give $\bX$ and $\bY$ the standard affine embeddings coming from $\rho$, and let $d_Y$ be the distance function on $\bY(E_{w_0})$ obtained from this embedding.  We choose $K_Y = (B_2 B_1 B_2^{-1})^l$, so that $K_Y \subset \bY(F_\infty) \subset \bY(E_\infty)$, and we have $z \in K_Y \cap \bY(E)$.

We next choose a compact set $K_X \subset \bX(E_{w_0})$ such that $d_Y( f(K_X), z) < C_1 \delta$ for some $C_1 = C_1(B_1, B_2) > 0$.  Let $B_{1,v_0}$ and $B_{2,v_0}$ denote the projection of $B_1$ and $B_2$ to $\bG(F_{v_0})$, and let $B_{3,v_0} \subset \bG(F_{v_0})$ denote the set of points within distance 1 of $B_{1,v_0}$.  We define $K_X = B_{3,v_0}^l \times B_{2,v_0} \subset \bX(E_{w_0})$.

To show that $K_X$ has the required property, we know that there are $h_i \in H_{v_0}$ such that $d(x^{-1} \gamma_i x, h_i) < \delta$ for all $i$.  We have assumed that $x^{-1} \gamma_i x \in B_1$ for all $i$, and that $x \in B_2$, which implies that $\widetilde{x} = (h_1, \ldots, h_l, x) \in K_X$.  Because $\gamma_i$, $x$, and $h_i$ lie in compact sets depending only on $B_1$ and $B_2$ we have $d(\gamma_i, x h_i x^{-1}) < C_1(B_1, B_2) \delta$ for all $i$, so that $d_Y( f(\widetilde{x}), z) < C_1 \delta$ as required.

Applying Proposition \ref{abstractapprox} to these data gives $C, M > 0$ such that if $\delta < C T^{-M}$ then $z \in \pi( \bX( \overline{E}) )$, so that $\pi^{-1}(z) \neq \emptyset$ as required.

\end{proof}

If $z \in \A^n(E)$, we define $\| z \|_f$ to be the LCM of the denominators of the norms of the coordinates of $z$.

\begin{prop}
\label{abstractapprox}

Let $\bX \subset \A^n$ and $\bY \subset \A^m$ be affine varieties defined over $E$, and let $f : \bX \to \bY$ be a map defined over $E$.  Let $w_0$ be an infinite place of $E$, and let $d$ be the standard distance function on $\A^m(E_{w_0})$.  Let $K_X \subset \bX(E_{w_0})$ and $K_Y \subset \bY(E_\infty)$ be compact.  There are $C, M > 0$ such that if $z \in \bY(E) \cap K_Y$ satisfies $\| z \|_f < T$, $d(z, f(K_X)) < \delta$, and $\delta < C T^{-M}$, then $z \in f( \bX(\overline{E}) )$.

\end{prop}

\begin{proof}

Let $I = f( \bX(\overline{E}) )$.  Chevalley's theorem implies that $I$ is a constructible subset of $\bY(\overline{E})$ defined over $E$.  This means that there is a finite decreasing chain of subvarieties $\bY = \bV_0 \supseteq \bV_1 \ldots \supseteq \bV_{2a} \supseteq \bV_{2a+1} = \emptyset$ of $\bY$ defined over $E$ and such that 

\bes
I = \bigcup_{i=1}^a \bV_{2i-1}(\overline{E}) \setminus \bV_{2i}(\overline{E}),
\ees
where $\bV_{2a}$ may be empty.

Let $\cO_X$ and $\cO_Y$ be the $E$-algebras of functions on $\bX$ and $\bY$.  For each $0 \le i \le 2a+1$, we let $\{ p_{i,j} \in \cO_Y : 1 \le j \le D(i) \}$ be a finite collection of equations defining $\bV_i$.  We have $f^{-1}(\bV_{2i}(\overline{E}) ) = f^{-1}(\bV_{2i+1}(\overline{E}) )$ for all $0 \le i \le a$, and so if we define the ideals $N_i = \langle f^* p_{i,j} : 1 \le j \le D(i) \rangle \subset \cO_X \otimes \overline{E}$ we see that $f^{-1}(\bV_{2i}(\overline{E}) )$ is the vanishing set of both $N_{2i}$ and $N_{2i+1}$.  The Nullstellensatz implies that the radicals of $N_{2i}$ and $N_{2i+1}$ in $\cO_X \otimes \overline{E}$ must be equal, and so there is $L = L(\bG, \bH, l) \in \N$ and elements

\bes
\{ c(i,j,k) \in \cO_X \otimes \overline{E} : 0 \le i \le a, 1 \le j \le D(2i+1), 1 \le k \le D(2i) \}
\ees
such that

\be
\label{null}
(f^* p_{2i+1,j})^L = \sum_{k=1}^{D(2i)}  c(i,j,k) f^* p_{2i,k}.
\ee

If $z \notin I$, there must be some $0 \le r \le a$ such that $z \in \bV_{2r}(\overline{E})$ but $z \notin \bV_{2r+1}(\overline{E})$.  As $z \notin \bV_{2r+1}(\overline{E})$, there must be some $s$ such that $p_{2r+1, s}(z) \neq 0$.  We have $p_{2r+1,s}(z) \in E$.  Our assumption that $z \in K_Y$ implies that $|p_{2r+1, s}(z)|_w \ll 1$ for all infinite places $w$.  Our assumption on $\| z \|_f$ implies that the denominator of $p_{2r+1, s}(z)$ is bounded by a power of $T$, and combining this with $|p_{2r+1, s}(z)|_w \ll 1$ implies that there exist $C_1$ and $M_1$ such that

\be
\label{w0lowerbd}
|p_{2r+1, s}(z)|_{w_0} \ge 2C_1 T^{-M_1}.
\ee

We now use (\ref{null}) to show that $z \notin \bV_{2r}(\overline{E})$.  There is $\kappa > 0$ such that

\be
\label{lip}
|p_{i,j}(y) - p_{i,j}(y')| \le \kappa d(y,y')
\ee
for all $i$ and $j$, and all $y, y' \in f(K_X) \cup K_{Y,w_0}$ where $K_{Y,w_0}$ denotes the projection of $K_Y$ to $\bY(E_{w_0})$.

By assumption, there is $x \in K_X$ with $d( f(x), z) \le \delta$, and applying (\ref{lip}) to $f(x)$ and $z$ gives $|p_{2r+1, s}( f(x) )| \ge 2C_1 T^{-M_1} - C_2 \kappa \delta$.  By decreasing $C$ and increasing $M$ if necessary, we may assume that

\bes
|f^* p_{2r+1, s}( x)| = |p_{2r+1, s}(f(x))| \ge C_1 T^{-M_1}.
\ees
Combining this with (\ref{null}) gives

\bes
\left| \sum_{k=1}^{D(2i)}  c(r,s,k)( x) f^* p_{2r,k}( x) \right| \ge C_1^L T^{-LM_1}.
\ees
If we let $C_3$ be an upper bound for all $|c(i,j,k)|$ on $K_X$, there must be some $k$ such that

\bes
|p_{2r,k}( f(x))| \ge \frac{C_1^L T^{-LM_1}}{D(2i) C_3}.
\ees
Combined with (\ref{lip}) this gives

\bes
|p_{2r, k}(z)| \ge \frac{C_1^L T^{-LM_1}}{D(2i) C_3} - C_2 \kappa \delta.
\ees
By shrinking $C$ and increasing $M$ if necessary, this implies that $p_{2r, k}(z) \neq 0$, which contradicts $z \in \bV_{2i}(\overline{E})$.

\end{proof}

\subsection{Proof of Proposition \ref{diophantine}}

We now find $S$ of controlled size such that $\bX = \bX_S$, to which we apply Proposition \ref{approx}.  The fact that some finite $S$ with this property exists is immediate by Noetherianness, but in order for it to be useful we need $|S|$ to be bounded only in terms of $\bG$ and $\bH$.

\begin{lemma}
\label{reduce}

There exists an integer $n = n(\bG, \bH)$ and $S \subset \cM(x, \delta, T)$ with $|S| \le n$ such that $\bX_S = \bX$

\end{lemma}

\begin{proof}

Let $\cO_G$ be the $F$-algebra of functions on $\bG$.  For $d \ge 0$, let $\cO_G^d$ be the subspace of $\cO_G$ consisting of the restrictions of degree $d$ polynomials on $\A^{N^2}$ under the natural embedding $\bG \to SL_N \to \A^{N^2}$.  $\bX$ is cut out by the conditions $\{ y \in g \bH g^{-1} : y \in \cM(x, \delta, T) \}$, which are expressed as functions in $\cO_G^d$ for $d = d(\bG, \bH)$ sufficiently large.  Let $V \subset \cO_G^d$ be the $E$-vector subspace spanned by these functions.  It has dimension at most $\dim \cO_G^d$, and so there exists $S \subset \cM(x, \delta, T)$ with $|S| \le \dim \cO_G^d$ such that $\bX = \bX_S$.

\end{proof}

Let $n$ be the integer provided by Lemma \ref{reduce}, and apply Proposition \ref{approx} for every $l$ between 1 and $n$.  This gives $C = C(B_1, B_2) > 0$ and $M = M(\bG, \bH) > 0$ such that $\bX(\overline{E}) \neq \emptyset$ if $x \in B$ and $\delta < C T^{-M}$.  It remains to show that we can find $y \in \bX(\overline{E})$ of bounded complexity.

We shall do this using Lemma \ref{cxbound} below, which states that an algebraic set in $\overline{\Q}^n$ cut out by polynomials of low complexity contains an algebraic point of low complexity.  We feel that such a result should be standard, but as we are unable to find a suitable version in the literature, we shall give a proof.

Before stating the Lemma, let us specify what we mean by ``low complexity''.  Let $K$ be a number field with ring of integers $\cO_K$.  We say that a (multivariable) polynomial $f$ with coefficients in $\cO_K$ has complexity $(D,X)$ if it has degree at most $D$, and all of its coefficients have absolute value at most $X$ under all archimedean norms on $K$.  We shall say that an algebraic number $y$ has complexity $(D,X)$ over $K$ if it is a root of a polynomial over $\cO_K$ of complexity $(D,X)$.  We say that a polynomial has complexity $\lesssim_* (D,X)$ (where $*$ is some additional data) if it has complexity $(D', C_1 X^{C_2})$ for $D'$, $C_1$, and $C_2$ depending on $D$ and $*$, and likewise for algebraic numbers.

\begin{lemma}
\label{cxbound}

Let $K$ be a number field.  Let $f_1, \ldots, f_k \in \cO_K[x_1, \ldots, x_n]$ be a set of polynomials of complexity $(D,X)$.  Let $Z$ be the common zero locus of the $f_i$ in $\overline{\Q}^n$.  If $Z$ is nonempty, there exists $y \in Z$ with complexity $\lesssim_{k,n} (D,X)$ over $K$.

\end{lemma}

\begin{proof}

We perform induction on $n$.  Let $\pi : \overline{\Q}^n \to \overline{\Q}$ be the projection to the first co-ordinate.  We first show that it suffices to find $y_1 \in \pi(Z)$ with complexity $\lesssim_{k,n} (D,X)$ over $K$.  Suppose such a $y_1$ exists, and let $Z' = \pi^{-1}(y_1) \subset \overline{\Q}^{n-1}$.  We wish to show that $Z'$ has a point of low complexity.  Let $K'$ be the number field generated by $y_1$ over $K$, and define $f_i'(x_2, \ldots, x_n) = f_i(y_1, x_2, \ldots, x_n) \in K'[x_2, \ldots, x_n]$ so that $Z'$ is the zero locus of the $f_i'$.  By clearing denominators we may assume that $f_i' \in \cO_K'[x_2, \ldots, x_n]$, and that the $f_i$ have complexity $\lesssim_{k,n} (D,X)$ over $K'$.  Our inductive hypothesis implies that there is a point $y' \in Z'$ with complexity $\lesssim_{k,n} (D,X)$ over $K'$.  By taking the product of the polynomials defining the coordinates of $y'$ over their embeddings above $K$, one sees that $(y_1, y') \in Z$ has complexity $\lesssim_{k,n} (D,X)$ over $K$ as required.

We now produce $y_1 \in \pi(Z)$ with complexity $\lesssim_{k,n} (D,X)$.  We do this by proving an effective version of Chevalley's theorem for $\pi(Z)$.  Let $y_1 \in \overline{\Q}$, and define $f_i' \in \overline{\Q}[x_2, \ldots, x_n]$ as before.  For $d \ge 0$, let $V_d \subset \overline{\Q}[x_2, \ldots, x_n]$ be the set of polynomials of degree $\le d$.  By the effective Nullstellensatz \cite{Ko}, there exists $d = d(D,k,n)$ such that $\pi^{-1}(y_1) = \emptyset$ if and only if there exist $g_i \in V_d$ such that $\sum f'_i g_i = 1$.  In other words, one has a $\overline{\Q}$-linear map

\bes
A : V_d^k \to V_{d+D}, \quad A(g_1, \ldots, g_k) = \sum f_i' g_i,
\ees
and one has $\pi^{-1}(y_1) = \emptyset$ if and only if the equation $A{\bf x} = {\bf 1}$ has a solution, where ${\bf 1}$ is the vector corresponding to $1 \in V_{d+D}$.

If we write $A$ as a matrix with respect to the monomial basis of $V_d^k$ and $V_{d+D}$, we see that the entries are of the form $p_{ij}(y_1)$, where $p_{ij} \in \cO_K[x_1]$ have complexity $\lesssim_{k,n} (D,X)$.  We shall therefore identify $A$ with this matrix of polynomials.

By performing row and column operations on $A$ with coefficients in $\cO_K[x_1]$, we may reduce the system $(A:{\bf 1})$ to $(B:{\bf v})$, where $B$ is a diagonal matrix and ${\bf v}$ is a vector with entries in $\cO_K[x_1]$.  Moreover, all entries of $B$ and ${\bf v}$ have complexity $\lesssim_{k,n} (D,X)$.  The system $A(y_1) {\bf x} = {\bf 1}$ has a solution with ${\bf x} \in V_d^k$ if and only if $B(y_1) {\bf x} = {\bf v}(y_1)$ does, except possibly when $y_1$ lies in the finite set $W_1$ consisting of the roots of the polynomials in $\cO_K[x_1]$ that we multiplied by during the reduction.

Let $W_2$ be the set of $y_1$ where $B(y_1) {\bf x} = {\bf v}(y_1)$ has a solution, so that $W_2 - W_1 \subset \pi(Z) \subset W_1 \cup W_2$.  $W_2$ is clearly constructable, and is either a finite set of points with complexity $\lesssim_{k,n} (D,X)$, or the complement of such a set.  In the first case, we see that $\pi(Z)$ is a finite set of points of complexity $\lesssim_{k,n} (D,X)$, and we are done.  In the second, the number of points in $\overline{\Q} - (W_2 - W_1)$ must be bounded in terms of $k$, $n$, and $D$.  It follows that there exists $y_1 \in \Z \cap \pi(Z)$ with size bounded in terms of these data, which completes the proof.

\end{proof}

We now finish the proof using Lemma \ref{cxbound}.  By Lemma \ref{reduce}, $\bX$ is cut out by a set of polynomials over $E$ whose cardinality and degrees are bounded in terms of $\bG$ and $\bH$.  In addition, our assumptions that $x \in B_2$, $x^{-1} \gamma x \in B_1$, and $\| \gamma \|_f \le T$ imply that the heights of the coefficients of these polynomials are bounded by $C' T^{M'}$ for some $C' = C'(B_1, B_2)$ and $M' = M'(\bG, \bH)$.

We may therefore apply Lemma \ref{cxbound} to produce a point $y \in \bX(\overline{E})$ of complexity $(D,X)$, where $D$ is bounded in terms of $\bG$ and $\bH$, and $X \le C_1 T^{M_1}$ for $C_1 = C_1(B_1, B_2)$ and $M_1 = M_1(\bG, \bH)$.  If $E'$ is the extension over which $y$ is defined, $|E':F|$ is bounded in terms of $\bG$ and $\bH$.  Moreover, there is $Q$ with $\ll_{B_1, B_2} \log T$ prime factors such that $E'/F$ is unramified away from $Q$, and the denominators of $\rho(y)$ all divide $Q$.  We may also assume that $E'/F$ is Galois.  The definition of $\bX$ then gives

\bes
\cM(x, T, \delta) \subset y \bH y^{-1}(E'),
\ees
and because $\cM(x, T, \delta) \subset \bG(F)$ we may descend $y \bH y^{-1}$ to the variety $\bL / F$.

\section{Estimating intersections in buildings}
\label{sec6}

This section contains the proof of inequality (\ref{apply51}), used in the proof of Proposition \ref{returns} bounding Hecke returns.

\subsection{Notation and statement of main result}

Let $K$ be a $p$-adic field with ring of integers $\cO$ and residue field $k$.  Let $\varpi$ be a uniformizer, and let $q$ be the order of $k$.  Let $\widetilde{K}$ be the maximal unramified extension of $K$, with integers $\widetilde{\cO}$ and residue field $\widetilde{k}$.  Let $\Gamma$ be the Galois group of $\tK / K$, which is canonically identified with the Galois group of $\tk / k$.

We shall work with affine group schemes over $\cO$ in this section, which we denote by upper case Gothic letters.  Any unfamiliar terminology regarding these objects will be defined in Sections \ref{sec:schemes1} and \ref{sec:schemes2}.  If $\gX$ is an affine group scheme over $\cO$, we denote its special and generic fibers by $\gX_k$ and $\gX_K$.  Let $\gG$ and $\gH$ be two smooth connected affine reductive group schemes over $\cO$.  We assume that $\gH$ is embedded as a closed subgroup scheme of $\gG$, and that $\gG_K$ and $\gH_K$ are both split.  We also assume that there are smooth closed subgroup schemes $\gT < \gG$ and $\gT_H < \gH$ such that $\gT_H < \gT$, and $\gT_K$ and $\gT_{H,K}$ are maximal split tori in $\gG_K$ and $\gH_K$\footnote{Note that the assumptions we have made on $\gG$ and $\gH$ imply that such $\gT$ and $\gT_H$ exist, by an argument we recall at the end of Section \ref{sec:schemes2}.  However, we wish to work with the tori given to us by the global setup of Section \ref{sec3}. \label{foot:tori}}.

Let $W$ be the Weyl group of $\gT$ in $\gG$, and let $\| \cdot \|_H^*$ be the seminorm on $X_*(\gT_H)$ defined in Section \ref{alggps}.  Let $y \in \gG(\tO)$, and define

\bes
L = \bigcap_{\sigma \in \Gamma} (y \gH_K y^{-1})^\sigma,
\ees
which is a subgroup of $\gG_K$ defined over $K$.  We define

\bes
\cL(\mu) = \# (L(K) \gG(\cO) \cap \gG(\cO) \mu(\varpi) \gG(\cO)) / \gG(\cO)
\ees
for $\mu \in X_*(\gT)$.  This can be thought of as the size of the intersection of $L(K)$ with a sphere in the building of $\gG_K$.  The main result of this section is a bound for this size.  

\begin{prop}
\label{buildingcount}

Let $\mu \in X_*(\gT)$, and define $\Lambda_\mu = X_*(\gT_H) \cap W\mu$.  We have $\cL(\mu) \ll \sum_{\lambda \in \Lambda_\mu} q^{ 2\| \lambda \|_H^*}$, where the implied constant depends only on the complexity of the image of $\gH_k$ in $\A_k^{(\dim \gG)^2}$ under the adjoint mapping $\gG_k \to GL( \textup{Lie}(\gG_k) ) \subset \A_k^{(\dim \gG)^2}$.  Moreover, this dependence is ineffective.

\end{prop}

See Section \ref{sec:complexity} for the definition of complexity used in the proposition.  We emphasize that the implied constant does not depend on $y$, $K$, or even the characteristic of $k$.  In the case $y = 1$ (or equivalently, $y \in \gH(\tO)$), we have an explicit formula for $\cL(\mu)$.

\begin{prop}
\label{y1}

Let $\mu \in X_*(\gT)$, and $\Lambda_\mu = X_*(\gT_H) \cap W\mu$.  For any $\lambda \in X_*(\gT_H)$, let $Q_{H,\lambda}$ denote the corresponding parabolic subgroup of $\gH_k$.  $\Lambda_\mu$ is invariant under $W_H$, and we have
\[
\# (\gH(K) \gG(\cO) \cap \gG(\cO) \mu(\varpi) \gG(\cO)) / \gG(\cO) = \sum_{\lambda \in \Lambda_\mu / W_H} q^{ 2 \| \lambda \|^*_H} \frac{ \# \gH(k) / q^{\textup{dim} \gH} }{ \# Q_{H,\lambda}(k) / q^{ \textup{dim} Q_{H,\lambda}} }.
\]

\end{prop}

Specializing to $\gH = \gG$, we obtain the following classic result.

\begin{cor}
\label{Gy1}

Let $\lambda \in X_*(\gT)$, and let $Q_\lambda$ be the parabolic subgroup of $\gG_k$ associated to $\lambda$.  We have
\[
\# \gG(\cO) \lambda(\varpi) \gG(\cO) / \gG(\cO) = q^{ 2 \| \lambda \|^*} \frac{ \# \gG(k) / q^{\textup{dim} \gG} }{ \# Q_\lambda(k) / q^{ \textup{dim} Q_\lambda} }.
\]

\end{cor}

\subsection{Complexity of algebraic sets}
\label{sec:complexity}

We shall use results of Breuillard, Green, and Tao \cite[Sec. 3]{BGT} on the complexity of algebraic sets, which we now recall.  The first three results below define the complexity of a variety, and of a map between varieties, and say that the image of a variety of bounded complexity by a map of bounded complexity is a constructible set of bounded complexity.

\begin{definition}[\cite{BGT}, Definition 3.1]
\label{setcomplex}

Let $M \ge 1$ be an integer.

\begin{enumerate}[(i)]

\item An affine variety over $\tk$ of complexity at most $M$ in $\A^n$ is a subset $V \subset \A^n(\tk)$ of the form
\[
V = \{ x \in \A^n(\tk) : P_1(x) = \ldots = P_m(x) = 0 \}
\]
where $0 \le n, m \le M$, and $P_1, \ldots, P_m : \A^n(\tk) \to \tk$ are polynomials of degree at most $M$.

\item  A projective variety over $\tk$ of complexity at most $M$ in $\bbP^n$ is a subset $V \subset \bbP^n(\tk)$ of the form
\[
V = \{ x \in \bbP^n(\tk) : P_1(x) = \ldots = P_m(x) = 0 \}
\]
where $0 \le n, m \le M$, and $P_1, \ldots, P_m : \tk^{n+1} \to \tk$ are homogeneous polynomials of degree at most $M$.

\item A quasiprojective variety over $\tk$ of complexity at most $M$ in $\bbP^n$ is a set of the form $V \setminus W \subset \bbP^n(\tk)$, where $V, W$ are projective varieties of
complexity at most $M$.

\item A constructible set over $\tk$ of complexity at most $M$ in $\bbP^n$ is a boolean combination of at most $M$ projective varieties of complexity at most $M$.

\end{enumerate}

\end{definition}

As in \cite{BGT}, we may consider both affine and projective varieties as quasiprojective varieties, and abbreviate quasiprojective variety to variety.

\begin{definition}[\cite{BGT}, Definition 3.3]
\label{mapcomplex}

Let $V \subset \bbP^n(\tk)$ and $W \subset \bbP^m(\tk)$ be varieties, and let $M \ge 1$.  A map $f : V \to W$ is said to be regular with complexity at most $M$ if $V, W$ are individually of complexity at most $M$ in $\bbP^n$ and $\bbP^m$, and if one can cover $V$ by some varieties $V_1, \ldots , V_r$ of complexity at most $M$ for some $r \le M$ such that

\begin{enumerate}[(i)]

\item for each $1 \le j \le r$, $f(V_j)$ is contained in an affine space $\A^m_{i_j} (\tk) \subset \bbP^m(\tk)$;
\item the map $f |{V_j}$ has the form $(P_{j,1}/Q_{j,1}, \ldots , P_{j,m}/Q_{j,m})$, where the $P_{j,l}, Q_{j,l}$ are homogeneous polynomial maps from $\tk^{n+1}$ to $\tk$ with $\textup{deg}(P_{j,l}) = \textup{deg}(Q_{j,l}) \le M$, and the $Q_{j,l}$ are non-vanishing on $V_j$.

\end{enumerate}

\end{definition}

\begin{lemma}[\cite{BGT}, Lemma 3.4]
\label{composecomplex}


Let $V$ and $W$ be varieties of complexity at most $M$ in $\bbP^n$ and $\bbP^m$ respectively, and let $f : V \to W$ be regular of complexity at most $M$.  The image $f(V)$ is a constructible set of complexity $O_M(1)$ in $\bbP^m$. In particular, by \cite[Lemma 3.2]{BGT}, the Zariski closure of this image is a variety of complexity $O_M(1)$ in $\bbP^m$.

\end{lemma}

These notions of complexity will be used together with the following bound on rational points.

\begin{lemma}
\label{pointcount}

If $V$ is a projective variety over $\tk$ of complexity at most $M$ and dimension $d$, then $\# V(k) \ll_M q^d$.

\end{lemma}

\begin{proof}

Lemma A.4 of \cite{BGT} allows us to assume that $V$ is irreducible, and \cite[Lemma 3.5]{BGT} implies that its degree is bounded in terms of $M$.  The lemma now follows from \cite[Lemma 1]{LW}; note that the proof of that lemma assumes that $V$ is defined over $k$, but this is not needed.

\end{proof}

The constants in Lemmas \ref{composecomplex} and \ref{pointcount} are ineffective (although this could probably be overcome with more work), which is the source of the ineffectiveness of Theorem \ref{buildingcount}.  

\subsection{Background on group schemes}
\label{sec:schemes1}

Let $\gX$ be an affine group scheme over $\cO$.  We denote the $\cO$-algebra of $\gX$ by $\cO[\gX]$.  We denote the generic and special fibers of $\gX$ by $\gX_K$ and $\gX_k$ respectively.  We say that $\gX$ is connected if both its generic and special fibers are.

If $n \ge 1$, we define the $n$th principal congruence subgroup of $\gX(\tO)$ to be $X_n = \ker \gX(\tO) \to \gX(\tO / \p^n)$, and set $X_0 = \gX(\tO)$.  If $\gY$ is another affine group scheme with a morphism $f : \gY \to \gX$, the commutative diagram
\[
\begin{CD}
\gY(\tO)     @>>>  \gX(\tO)\\
@VVV        @VVV\\
\gY(\tO / \p^n)   @>>>  \gX(\tO / \p^n)
\end{CD}
\]
implies that $f(Y_n) \subset X_n$ for all $n$.  Applying this to the product and inverse maps on $\gX$, we see that $X_n$ is a group.

$\gX$ is called flat if $\cO[\gX]$ is flat (equivalently, free) as an $\cO$ module, and smooth if it is flat and $\gX_K$ and $\gX_k$ are smooth.  We shall use two consequences of smoothness in this section.  First, if $\gX$ is smooth then for any $n$ the reduction map $\gX( \tO) \to \gX( \tO / \p^n)$ is surjective, as $\tO$ is Henselian.  Secondly, smoothness of $\gX$ implies that its Lie algebra $\mathfrak{x}$ behaves well under passage to the fibers.  In particular, smoothness implies that $\mathfrak{x}$ is a free $\cO$ module, and we have $\text{Lie}(\gX_K) = \mathfrak{x} \otimes_\cO K$ and $\text{Lie}(\gX_k) = \mathfrak{x} \otimes_\cO k$.  We define $\overline{\mathfrak{x}} = \mathfrak{x} \otimes_\cO \widetilde{k}$.

If $\gX$ is smooth, then we have $X_0 / X_1 \simeq \gX_k(\widetilde{k})$, and the following lemma gives a similar statement for $X_n / X_{n+1}$.

\begin{lemma}

If $\gX$ is smooth, then $X_n / X_{n+1} \simeq \overline{\mathfrak{x}}$ for any $n \ge 1$.

\end{lemma}

\begin{proof}

We define $\kappa_n = \text{ker}( \gX( \tO / \p^{n+1}) \to \gX( \tO / \p^n) )$.  We first use smoothness to show that $X_n / X_{n+1} \simeq \kappa_n$, and then show that a general affine group scheme satisfies $\kappa_n \simeq \overline{\mathfrak{x}}$.  To prove the first claim, consider the sequence
\[
X_{n+1} \to X_n \to \gX( \tO / \p^{n+1}) \to \gX( \tO / \p^n).
\]
The sequence is exact at $X_n$ by definition, and exact at $\gX( \tO / \p^{n+1})$ because of the Henselian property of $\gX$.  This implies the claim.

If we denote the identity in $\gX(\cO)$ by $e$, then $\kappa_n$ is the set of $\cO$ algebra homomorphisms $\cO[\gX] \to \tO / \p^{n+1}$ whose reduction mod $\p^n$ is the same as $e$.  If we let $\text{Hom}_\cO( \cO[\gX], \tO / \p^{n+1} )_e \subset \text{Hom}_\cO( \cO[\gX], \tO / \p^{n+1} )$ be the $\cO$ module homomorphisms that reduce to $e$ mod $\p^n$, then there is an isomorphism

\begin{align*}
\text{Hom}_k( \cO[\gX], \widetilde{k}) & \simeq \text{Hom}_\cO( \cO[\gX], \tO / \p^{n+1} )_e, \\
d & \mapsto e + \varpi^n d.
\end{align*}
Moreover, it may be checked that $e + \varpi^n d$ is an algebra homomorphism if and only if $d$ is a derivation of $e: \cO[\gX] \to \widetilde{k}$, so that we have an isomorphism $\text{ker}( \gX( \tO / \p^{n+1}) \to \gX( \tO / \p^n) ) \simeq \overline{ \mathfrak{x}}$ as required.

\end{proof}

Let $\gY$ be an affine group scheme, and let $\iota : \gY \to \gX$ be a closed embedding.  We recall that this means that the corresponding map $\iota^* : \cO[ \gX] \to \cO[ \gY]$ is surjective.  This implies that if $R \subset R'$ are two $\cO$ algebras, then $\iota : \gY(R) \to \gX(R)$ is injective and $\iota(\gY(R)) = \iota(\gY(R')) \cap \gX(R)$.

We define the standard multiplicative and additive group over $\cO$ to be the usual integral models of $\mathbb{G}_m$ and $\mathbb{G}_a$ with coordinate rings $\cO[ X, X^{-1}]$ and $\cO[X]$ respectively.  A standard torus is a product of standard multiplicative groups.  If $\gT$ is a standard torus, then the group of integral cocharacters $X_*(\gT)$ is identified with $X_*(\gT_K)$ and $X_*(\gT_k)$ under passage to the fibers.

\subsection{Background on reductive group schemes}
\label{sec:schemes2}

Let $\gG$ be a smooth affine group scheme over $\cO$.  We say that $\gG$ is reductive if both its fibers are.  We say that a subgroup scheme $\gT < \gG$ is a split maximal torus if it is closed in $\gG$, isomorphic to a standard torus over $\cO$, and is a maximal torus in one (equivalently, both) of the fibers of $\gG$.  It follows from e.g. \cite[Prop. 2.1.2]{Co} that any two split maximal tori $\gT, \gT'$ of $\gG$ are conjugate, because their special fibers are and $\text{Transp}_\gG( \gT, \gT')$ is smooth.

There is a notion of $\gG$ being split as a reductive group over $\cO$.  We will not give the full definition here, which may be found in \cite[Def. 5.1.1]{Co}, and instead we shall only state those consequences of it that we shall use.

\begin{enumerate}

\item
\label{split1}
$\gG$ has a split maximal torus $\gT$.

\item
\label{split2}
There is a root system $\Phi \subset X^*(\gT)$, which gives the root systems of $(\gT_K,\gG_K)$ and $(\gT_k, \gG_k)$.  We will refer to $\Phi$ as the root system of $(\gT, \gG)$.

\item
\label{split3}
For each $\alpha \in \Phi$ there is a root space $\g_\alpha \subset \g$, which is a free $\cO$ module of rank one on which $\gT$ acts by $\alpha$, and there is a root space decomposition $\g = \gt \oplus \bigoplus_{\alpha \in \Phi} \g_\alpha$.

\item
\label{split4}
For each $\alpha \in \Phi$ there is a group scheme $\gU_\alpha$, isomorphic to the standard additive group over $\cO$, and a closed embedding of $\gU_\alpha$ in $\gG$ that induces the inclusion $\g_\alpha \hookrightarrow \g$ and which identifies $\gU_{\alpha,K}$ and $\gU_{\alpha,k}$ with the root subgroups of $\alpha$.

\item
\label{split5}
Let $\Phi^+ \subset \Phi$ be a system of positive roots.  There is a smooth closed subgroup scheme $\gU^+ < \gG$ with connected unipotent fibers, and such that for any ordering of $\Phi^+$ the multiplication map induces an isomorphism $\prod_{\alpha \in \Phi^+} \gU_\alpha \simeq \gU^+$.  It follows that the fibers of $\gU^+$ are the maximal unipotent subgroups of the fibers of $\gG$ corresponding to $\Phi^+$.  We likewise define $\gU^-$ for the system $-\Phi^+$.

\item 
\label{split6}
The multiplication map $\gU^- \times \gT \times \gU^+$ is an isomorphism onto an open subscheme $\mathfrak{C}$ of $\gG$ containing the identity.

\end{enumerate}

The Weyl group of $\Phi$ will also be referred to as the Weyl group of $\gT$ in $\gG$.  To show where these facts may be found in \cite{Co}, statements (\ref{split1})-(\ref{split3}) are consequences of the definition of splitness in Definition 5.1.1, together with the definition of the root spaces $\g_\alpha$ in Definition 4.1.1.  Statement (\ref{split4}) is Theorem 4.1.4, and statements (\ref{split5}) and (\ref{split6}) are Theorem 5.1.13.

These facts imply the following decomposition of the congruence subgroups $G_n$.

\begin{lemma}
\label{rootspace}

If $n \ge 1$, the multiplication map

\bes
\pi_n : U_n^- \times T_n \times U_n^+ \to G_n
\ees
is a bijection.

\end{lemma}

\begin{proof}

We let $\pi : \gU^- \times \gT \times \gU^+ \to \gG$ be the product map, which induces $\pi_n$ on $\tO$ points.  The injectivity of $\pi_n$ follows from the fact that $\pi$ is the composition of an isomorphism $\gU^- \times \gT \times \gU^+ \simeq \mathfrak{C}$ and an open inclusion $\mathfrak{C} \hookrightarrow \gG$, which are bijective and injective on $\tO$ points respectively.

To show surjectivity, the fact that $\mathfrak{C}$ contains the identity means that $\mathfrak{C}(\tO)$ contains $G_1$, and hence $G_n$.  If $y \in G_n$, it follows that there exists $x \in (\gU^- \times \gT \times \gU^+)(\tO)$ such that $\pi(x) = y$.  Let $\overline{x} \in (\gU^- \times \gT \times \gU^+)(\tO / \p^n)$ be the reduction of $x$ mod $\p^n$.  If $x \notin U_n^- \times T_n \times U_n^+$, then $\overline{x}$ will be different from the identity, and so will $\pi( \overline{x}) \in \gG( \tO / \p^n)$.  This contradicts the fact that $\pi(\overline{x})$ must be the reduction of $y$.

\end{proof}

If $\gG, \gH, \gT$, and $\gT_H$ are as in the statement of Proposition \ref{buildingcount}, the following theorem from \cite[Prop. 1.3]{Co2} implies that these groups are all split, and that $\gT$ and $\gT_H$ are split maximal tori in $\gG$ and $\gH$ respectively, so that the above discussion applies to them.

\begin{theorem}
\label{stdscheme}

If $G'$ is a split connected reductive group over $K$, then there is a unique smooth reductive group scheme $\gG'$ over $\cO$ with $\gG'_K \simeq G'$.  In particular, $\gG'$ is split.

\end{theorem}

The group scheme $\gG'$ in this theorem is sometimes referred to as the Chevalley group scheme of type $G'$.

We finish this section by proving the claim in footnote \ref{foot:tori}, i.e. that if $\gG$ and $\gH$ are smooth connected affine reductive, with $\gH$ closed in $\gG$ and $\gG_K$ and $\gH_K$ split, then there exist split maximal tori $\gT < \gG$ and $\gT_H < \gH$ with $\gT_H < \gT$.  First, Theorem \ref{stdscheme} implies that $\gH$ is split, so it has a split maximal torus $\gT_H$.  By \cite[Lemma 2.2.4]{Co}, there exists a smooth closed subgroup scheme $\mathfrak{Z} < \gG$, called the centralizer of $\gT_H$ in $\gG$, which by Definition 2.2.1 of \cite{Co} has the property that for any $\cO$-algebra $R$ the subset of $\gG(R)$ preserving $\gT_H$ under conjugacy is $\mathfrak{Z}(R)$.  It may be seen that the fibers of $\mathfrak{Z}$ are connected, reductive, and split, and applying Theorem \ref{stdscheme} again we see that $\mathfrak{Z}$ is split over $\cO$.  If we let $\gT$ be a split maximal torus of $\mathfrak{Z}$, then we have $\gT_H < \gT$ by combining $\gT_H(\tO) < \gT(\tO)$ with \cite[Prop. 1.7.6]{BT2}.

\subsection{Congruence flag varieties}

We now begin the proof of Proposition \ref{buildingcount}, starting with an outline of the method.  The set $\gG( \tO ) \mu(\varpi) \gG( \tO ) / \gG( \tO )$ can be identified with a $\gG(\tO)$-orbit in the building of $\gG_{\widetilde{K}}$, and $\cL(\lambda)$ is roughly the size of the intersection of $y \gH(\widetilde{K}) y^{-1}$ with the points of this orbit coming from $\gG(K)$.  If we define
\be
\label{PFdef}
P_\mu = \gG( \tO ) \cap \mu(\varpi) \gG( \tO ) \mu(\varpi)^{-1}, \quad F_\mu = \gG( \tO ) / P_\mu,
\ee
then $F_\mu$ naturally parametrizes $\gG( \tO ) \mu(\varpi) \gG( \tO ) / \gG( \tO )$, and in Lemma \ref{Lmucontrol} we bound $\cL(\mu)$ in terms of an intersection inside $F_\mu$.  In Section \ref{sec:filter} we bound this intersection using the natural congruence filtration on $F_\mu$.  Section \ref{sec:ratlflag} contains an extra argument involving flag varieties over $k$ that is needed to handle the first step of this filtration.

If $\lambda \in X_*(T_H)$, we define the objects $P_\lambda^H$ and $F_\lambda^H$ for $\gH$ as in (\ref{PFdef}).  Note that $P_\lambda^H = P_\lambda \cap \gH(\tO)$, so there is a natural injection $\iota : F_\lambda^H \rightarrow F_\lambda$.  The Galois group $\Gamma$ stabilises $P_\lambda$, hence acts on $F_\lambda$, and we may define

\begin{align*}
I_\lambda & = \{ h \in F_\lambda^H : y \iota(h) \in F_\lambda^\Gamma \}.
\end{align*}
The sets $I_\lambda$ control $\cL(\mu)$ as follows.

\begin{lemma}
\label{Lmucontrol}

We have $\cL(\mu) \le \sum_{\lambda \in \Lambda_\mu} \# I_\lambda$.

\end{lemma}

\begin{proof}

The map $\gG(K) / \gG( \cO ) \rightarrow \gG(\tK) / \gG( \tO )$ is an injection, and the images of $L(K) \gG(\cO)$ and $\gG(\cO) \lambda( \varpi) \gG(\cO)$ under this map are contained in
\begin{align*}
B_1 & = y \gH( \tK) y^{-1} \gG(\tO) / \gG( \tO ) = y \gH( \tK) \gG(\tO) / \gG( \tO ) \\
\text{and} \quad B_2 & = \{ g \in \gG( \tO) \lambda(\varpi) \gG( \tO) / \gG( \tO ) : \sigma(g) = g, \sigma \in \Gamma \}
\end{align*}
respectively, so that $\cL(\mu) \le \#( B_1 \cap B_2)$.  We may translate both of these sets by $y^{-1}$, to obtain
\[
y^{-1} B_1 = \gH( \tK) \gG(\tO) / \gG( \tO ), \quad y^{-1} B_2 = \{ g \in \gG( \tO) \lambda(\varpi) \gG( \tO) / \gG( \tO ) : \sigma(yg) = yg, \sigma \in \Gamma \}.
\]
Applying Lemma \ref{Cartancomp} gives

\bes
y^{-1}( B_1 \cap B_2) = \bigcup_{\lambda \in \Lambda_\mu} \{ g  \in \gH(\tO) \lambda(\varpi) \gG( \tO ) / \gG( \tO ) : \sigma (y g) = (y g), \sigma \in \Gamma \}.
\ees
It may be checked that the map

\begin{align*}
F_\lambda^H & \rightarrow \gH(\tO) \lambda(\varpi) \gG( \tO ) / \gG( \tO ) \\
h & \mapsto h \lambda(\varpi) \gG( \tO )
\end{align*}
is a bijection, and the condition that the coset $yh \lambda(\varpi) \gG( \tO )$ is fixed by $\Gamma$ is just that $y \iota(h)$ is fixed by $\Gamma$ as an element of $F_\lambda$.  This gives
\[
\{ g  \in \gH(\tO) \lambda(\varpi) \gG( \tO ) / \gG( \tO ) : \sigma (y g) = (y g), \sigma \in \Gamma \} \simeq I_\lambda,
\]
which completes the proof.

\end{proof}

\begin{lemma}
\label{Cartancomp}

We have $\gH(\tK) \gG( \tO ) \cap \gG( \tO ) \mu(\varpi) \gG( \tO ) = \bigcup_{\lambda \in \Lambda_\mu} \gH(\tO) \lambda(\varpi) \gG( \tO )$.

\end{lemma}

\begin{proof}

It is clear that the right hand side is contained in the left.  For the reverse inclusion, let $hg \in \gG( \tO ) \mu(\varpi) \gG( \tO )$ with $h \in \gH(\tK)$ and $g \in \gG( \tO )$.  Applying the Cartan decomposition on $\gH(\tK)$ gives $h \in \gH(\tO) \lambda(\varpi) \gH(\tO)$ for some $\lambda \in X_*(\gT_H)$, and comparing this with the Cartan decomposition on $\gG(\tK)$ gives $\lambda \in W\mu$ as required.

\end{proof}

\subsection{Filtrations on coset spaces}
\label{sec:filter}

By Lemma \ref{Lmucontrol}, it suffices to prove
\be
\label{Slambdabd}
\# I_\lambda \ll q^{ 2 \| \lambda \|_H^* }
\ee
for all $\lambda \in X_*(T_H)$.  In this section we shall prove this using a congruence filtration on $I_\lambda$ defined using the natural filtrations on $F_\lambda$ and $F_{H, \lambda}$.  For $i \ge 0$, define
\[
P(\lambda,i) = P_\lambda G_i, \quad F(\lambda,i) = G_0 / P(\lambda,i),
\]
so that $F(\lambda, i)$ stabilize at $F_\lambda$ for $i$ large, and we have maps $\pi_i : F(\lambda, i) \to F(\lambda, i-1)$ for $i \ge 2$.  We define $P_H(\lambda,i)$, $F_H(\lambda, i)$, and $\pi_{H,i}$ similarly for $\gH$.  There are maps $\iota : F_H(\lambda, i) \to F(\lambda, i)$ for all $i$, which we prove are injective in Lemma \ref{iotainject}.  We define the sets $I(\lambda, i)$ by
\bes
I(\lambda,i) = \{ h \in F_H(\lambda,i) : y \iota(h) \in F(\lambda,i)^\Gamma \}.
\ees
The sets $I(\lambda, i)$ stabilize at $I_\lambda$, and the projections $\pi_{H,i}$ map $I(\lambda, i)$ to $I(\lambda, i-1)$.  We shall bound $I(\lambda,i)$ by bounding $I(\lambda, 1)$, and the fibers of $I(\lambda, i)$ under $\pi_{H,i}$.  First, we need to understand $F(\lambda, 1)$ and the fibers of $\pi_i$ on $F(\lambda, i)$, and show that $\iota$ is injective on $F_H(\lambda, i)$.

Let us introduce some notation.  $\lambda$ defines a map $\mathbb{G}_m / k \to \gT_{H,k} < \gT_k$, and composing this with the adjoint action of $\gT_{H,k}$ on $\overline{\g}$ and $\overline{\gh}$ gives these spaces $\Z$-gradings $\overline{\g} = \oplus_j \overline{\g}_j$ and $\overline{\gh} = \oplus_j \overline{\gh}_j$ such that $\mathbb{G}_m / k$ acts on $\overline{\g}_j$ by $x \mapsto x^j$.  We have $\overline{\gh}_j = \overline{\gh} \cap \overline{\g}_j$ for all $j$.  For $i \ge 0$, define

\begin{align*}
\g(\lambda,i) & := \bigoplus_{j \le i} \overline{\g}_j = \overline{\gt} + \bigoplus_{ \alpha \in \Delta : \langle \alpha, \lambda \rangle \le i } \overline{\g}_\alpha, \\
\gh(\lambda,i) & := \bigoplus_{j \le i} \overline{\gh}_j = \overline{\gt}_H + \bigoplus_{ \alpha \in \Delta_H : \langle \alpha, \lambda \rangle \le i } \overline{\gh}_\alpha.
\end{align*}
We have $\gh(\lambda, i) = \overline{\gh} \cap \g(\lambda, i)$.  The following lemma shows that these subspaces describe the higher steps of the congruence filtrations of $P_\lambda$ and $P_{H, \lambda}$.

\begin{lemma}
\label{glambdai}

For $i \ge 2$, we have $P(\lambda, i) \cap G_{i-1} / G_i = \g(\lambda,i)$ under the isomorphism $G_{i-1} / G_i \simeq \overline{\g}$, and likewise for $\gH$.

\end{lemma}

\begin{proof}

We begin by describing the intersection $G_n \cap \lambda(\varpi) G_0 \lambda(\varpi)^{-1}$ for any $n \ge 1$ in terms of root spaces.  We recall the objects $\Phi$, $\Phi^+$, $\gU_\alpha$, etc. and the properties (\ref{split1})-(\ref{split6}) they satisfy from Section \ref{sec:schemes2}.  We may assume that $\Phi^+$ is chosen so that $\lambda \in X_*^+(\gT)$.  We will show that

\be
\label{flagint}
G_n \cap \lambda(\varpi) G_0 \lambda(\varpi)^{-1} = \prod_{\alpha \in \Phi^-} U_{\alpha, n} \times T_n \times \prod_{\alpha \in \Phi^+} U_{\alpha,\max(n, \langle \alpha, \lambda \rangle ) }.
\ee
The inclusion of the right hand side in the left hand side follows from the formula
\be
\label{rootaction}
t U_{\alpha,m} t^{-1} = U_{\alpha, m + v(\alpha(t)) } \quad \text{for} \quad t \in \gT(\widetilde{K}),
\ee
where $v : K^\times \to \Z$ is the valuation.  To show the reverse inclusion, let $g \in G_n \cap \lambda(\varpi) G_0 \lambda(\varpi)^{-1}$.  Write $g = u^- t u^+ \in U_n^- T_n U_n^+$ using Lemma \ref{rootspace}.  Then $g^{\lambda(\varpi)^{-1}} = (u^-)^{\lambda(\varpi)^{-1}} t (u^+)^{\lambda(\varpi)^{-1}} \in G_0$.  Equation (\ref{rootaction}) gives $(u^-)^{\lambda(\varpi)^{-1}} t  \in G_0$, so we must have $(u^+)^{\lambda(\varpi)^{-1}} \in \gU^+( \widetilde{K}) \cap G_0$.  Because $\gU^+$ is closed in $\gG$, we have $\gU^+(\widetilde{K}) \cap G_0 = U^+_0$.

We may show that multiplication gives a bijection $\prod_{\alpha \in \Phi^+} U_{\alpha, n} \simeq U^+_n$ as in Lemma \ref{rootspace}, and we write $u^+ = \prod_{\alpha \in \Phi^+} u_\alpha$ with $u_\alpha \in U_{\alpha, n}$.  We have $(u^+)^{\lambda(\varpi)^{-1}} = \prod_{\alpha \in \Phi^+} (u_\alpha)^{\lambda(\varpi)^{-1}} \in U^+_0$, and comparing the bijections $\prod_{\alpha \in \Phi^+} U_{\alpha,0} \simeq U^+_0$ and $\prod_{\alpha \in \Phi^+} \gU_\alpha(\widetilde{K}) \simeq \gU^+(\widetilde{K})$ we see that this implies $(u_\alpha)^{\lambda(\varpi)^{-1}} \in U_{\alpha, 0}$ for all $\alpha \in \Phi^+$.  Equation (\ref{rootaction}) gives $u_\alpha \in U_{\alpha, \langle \alpha, \lambda \rangle } \cap U_{\alpha, n}$, so that $g$ lies in the right hand side of (\ref{flagint}).

To finish the proof, $P(\lambda, i) \cap G_{i-1} / G_i$ is the image of $G_{i-1} \cap \lambda(\varpi) G_0 \lambda(\varpi)^{-1}$ in $G_{i-1} / G_i$.  The proposition now follows from (\ref{flagint}) and the fact that the identifications $U_{\alpha, i-1} / U_{\alpha, i} \simeq \overline{\gu}$ and $T_{i-1} / T_i \simeq \overline{\gt}$ are functorial for the inclusions of these groups in $\gG$.

\end{proof}

The analog of Lemma \ref{glambdai} for $i = 1$ is to describe the image of $P(\lambda, 1)$ inside $G_0 / G_1 \simeq \gG( \widetilde{k})$.  If $Q_\lambda$ is the parabolic subgroup of $\gG_k$ associated to $\lambda$, \cite[Prop. 3.8]{Hr} states that $P(\lambda, 1) / G_1 \simeq Q_\lambda(\widetilde{k})$.  We likewise define $Q_{H,\lambda}$ to be the parabolic subgroup of $\gH_k$ associated to $\lambda$, and have $P_H(\lambda, 1) / H_1 \simeq Q_{H, \lambda}( \widetilde{k})$.  We have $\text{Lie}(Q_\lambda) \otimes \tk = \g(\lambda, 0)$ and $\text{Lie}(Q_{H,\lambda}) \otimes \tk = \gh(\lambda, 0)$, and because a parabolic subgroup is the normalizer of its Lie algebra, it follows that $Q_{H,\lambda}(\widetilde{k}) = \gH(\widetilde{k}) \cap Q_\lambda(\widetilde{k})$.

\begin{lemma}
\label{iotainject}

The map $\iota : F_H(\lambda, i) \to F(\lambda, i)$ is an injection for any $i$.

\end{lemma}

\begin{proof}

It is equivalent to show that $P(\lambda, i) \cap H_0 = P_H(\lambda, i)$, and we prove this by induction on $i$.  For the base case $i = 1$ we must show that $P_\lambda G_1 \cap H_0 = P_{H,\lambda} H_1$.  We have a commutative diagram
\[
\begin{CD}
H_0  @>>> \gH(\widetilde{k}) \\
@VVV   @VVV \\
G_0  @>>> \gG(\widetilde{k}). \\
\end{CD}
\]
The discussion above implies that $P_\lambda G_1 \cap H_0$ is the preimage of $Q_\lambda(\widetilde{k})$ under the bottom left pair of arrows.  On the other hand, $P_{H,\lambda} H_1$ is the preimage of $Q_{H, \lambda}( \widetilde{k})$ under the top arrow, and $Q_{H,\lambda}(\widetilde{k}) = \gH(\widetilde{k}) \cap Q_\lambda(\widetilde{k})$ completes the proof.

We now suppose that $P_\lambda G_{i-1} \cap H_0 = P_{H,\lambda} H_{i-1}$ for some $i \ge 2$, and show that $P_\lambda G_i \cap H_0 = P_{H,\lambda} H_i$.  One inclusion is clear, so we let $h \in P_\lambda G_i \cap H_0$ and wish to show that $h \in P_{H,\lambda} H_i$.  We have $h \in P_\lambda G_{i-1} \cap H_0$, and so our inductive hypothesis implies that $h \in P_{H,\lambda} H_{i-1}$.  We may therefore write $h = h_1 h_2$ with $h_1 \in P_{H,\lambda}$ and $h_2 \in H_{i-1}$.  As $h = h_1 h_2 \in P_\lambda G_i$ and $h_1 \in P_{H,\lambda} \subset P_\lambda$, this implies that $h_2 \in P_\lambda G_i \cap H_{i-1}$.  We have a commutative diagram
\[
\begin{CD}
H_{i-1}  @>>> H_{i-1} / H_i \simeq \overline{\gh} \\
@VVV   @VVV \\
G_{i-1}  @>>> G_{i-1} / G_i \simeq \overline{\g} \\
\end{CD}
\]
and $P_\lambda G_i \cap H_{i-1}$ is the inverse image of $P(\lambda, i) \cap G_{i-1} \simeq \g(\lambda,i)$ under the bottom left pair of arrows.  Because $\gh(\lambda, i) = \overline{\gh} \cap \g(\lambda_i)$, we may apply Lemma \ref{glambdai} for $\gH$ to show that  $P_\lambda G_i \cap H_{i-1} = P_{H,\lambda} H_i \cap H_{i-1}$.  This implies that $h_2 \in P_{H,\lambda} H_i$, so that $h \in P_{H,\lambda} H_i$ as required.

\end{proof}

We now use these results to describe $F(\lambda, 1)$ and the fibers of $\pi_i$.  Note that Lemma \ref{glambdai} implies that $\g(\lambda, i)$ is stable under the adjoint action of $P(\lambda, i)$, so that the subspace $\textup{Ad}_z( \g(\lambda, i) )$ with $z \in F(\lambda, i-1)$ in the following proposition is well defined.

\begin{prop}
\label{Ffibers}

We have the following identifications:

\begin{enumerate}

\item
\label{fiber1}
$F(\lambda, 1) = \gG( \tk) / Q_\lambda(\tk)$.

\item 
\label{fiber2}
If $i \ge 2$ and $z \in F(\lambda,i-1)$, the action of $G_{i-1} / G_i \simeq \overline{\g}$ on $\pi_i^{-1}(z)$ by left multiplication is transitive, and makes $\pi_i^{-1}(z)$ a torsor for $\overline{\g} / \textup{Ad}_z( \g(\lambda, i) )$.

\end{enumerate}

\noindent
Moreover, these identifications also hold for $\gH$ (with $\gh(\lambda, i)$ in place of $\g(\lambda, i)$) in a way that is compatible with the natural inclusions on both sides.

\end{prop}

\begin{proof}

Part (\ref{fiber1}) is immediate.  For part (\ref{fiber2}), the claim about the transitive action is clear.  If $z P(\lambda, i) \in F(\lambda, i)$, the stabilizer of $z P(\lambda, i)$ in $G_{i-1}$ is equal to $G_{i-1} \cap z P(\lambda, i) z^{-1}$, so the statement follows from Lemma \ref{glambdai}.

\end{proof}

We now apply these results to bound $I(\lambda, 1)$ and the fibers of $\pi_{H,i}$ on $I(\lambda, i)$.  To deal with $I(\lambda, 1)$, Proposition \ref{Ffibers} (\ref{fiber1}) gives an identification of $I(\lambda, 1)$ with $y \iota( (\gH_k / Q_{H,\lambda})( \widetilde{k}) ) \cap (\gG_k / Q_\lambda)(k)$, and (\ref{flagint1}) implies that $\# I(\lambda, 1) \ll q^{ \dim( \overline{\gh} / \gh(\lambda, 0) ) }$, where the dependence of the implied constant is the same as in Proposition \ref{buildingcount}.  The fibers of $\pi_i$ are controlled by the following lemma.

\begin{lemma}
\label{fibre2}

If $i \ge 2$, the fibers of $\pi_{H,i} : I(\lambda,i) \rightarrow I(\lambda,i-1)$ have size at most $q^{ \dim ( \overline{\gh} / \gh(\lambda, i-1) )}$.

\end{lemma}

\begin{proof}

If we choose $z_0 \in I(\lambda,i-1)$, the fiber above $z_0$ we wish to bound is $\pi_{H,i}^{-1}(z_0) \cap I(\lambda, i)$.  We may suppose that this fiber is nonempty, and let $z \in \pi_{H,i}^{-1}(z_0) \cap I(\lambda, i)$ be an element.  Combining Proposition \ref{Ffibers} (\ref{fiber2}) with the choice of basepoint $z$ gives a bijection $\pi_{H,i}^{-1}(z_0) \simeq \overline{\gh} / \text{Ad}_{z_0}( \gh(\lambda, i-1) )$, and we wish to determine $\pi_{H,i}^{-1}(z_0) \cap I(\lambda, i)$ in terms of this bijection.

If $X \in \overline{\gh} / \text{Ad}_{z_0}( \gh(\lambda, i-1) )$, then $X \cdot z$ lies in $I(\lambda, i)$ if and only if $y \iota(X \cdot z)$ lies in $F(\lambda,i)^\Gamma$.  The compatibility assertion in Proposition \ref{Ffibers} implies that $y \iota(X \cdot z) = \text{Ad}_y(X) \cdot y \iota(z)$, where now
\[
\text{Ad}_y(X) \in \text{Ad}_y \overline{\gh} / \text{Ad}_{y \iota(z_0)} \gh(\lambda, i-1) \subset \overline{\g} / \text{Ad}_{y \iota(z_0)} \g(\lambda, i-1).
\]
Moreover, the assumption $z \in I(\lambda, i)$ means that $y \iota(z) \in F(\lambda, i)^\Gamma$.  This implies that $\text{Ad}_{y \iota(z_0)}( \g(\lambda, i-1))$ is $\Gamma$-invariant, so that it makes sense to talk about the $\Gamma$-action on $\text{Ad}_y(X) \in  \overline{\g} / \text{Ad}_{y \iota(z_0)}( \g(\lambda, i-1))$.  Combining these, we have
\[
\sigma( y \iota(X \cdot z) ) = \sigma( \text{Ad}_y(X)) \cdot y \iota(z) \quad \text{for all} \quad \sigma \in \Gamma.
\]
We therefore have $y \iota(X \cdot z) \in F(\lambda,i)^\Gamma$ if and only if $\sigma( \text{Ad}_y(X)) = \text{Ad}_y(X)$ for all $\sigma \in \Gamma$, i.e if $\text{Ad}_y(X)$ is a $k$-rational vector in $\text{Ad}_y \overline{\gh} / \text{Ad}_{y \iota(z_0)}( \gh(\lambda, i-1))$ with respect to the rational structure on $\overline{\g} / \text{Ad}_{y \iota(z_0)} \g(\lambda, i-1)$.  It is clear that the number of such $X$ is at most $q^{ \dim ( \overline{\gh} / \gh(\lambda, i-1) )}$.

\end{proof}

The bound (\ref{Slambdabd}), and hence Proposition \ref{buildingcount}, now follows by combining $\# I(\lambda, 1) \ll q^{ \dim( \overline{\gh} / \gh(\lambda, 0) ) }$ with Lemma \ref{fibre2} and

\begin{lemma}

We have $\sum_{i=0}^\infty \dim( \overline{\gh} / \gh(\lambda, i) ) = 2 \| \lambda^* \|_H$.

\end{lemma}

\begin{proof}

We have $\dim( \overline{\gh} / \gh(\lambda, i) ) = \# \{ \alpha \in \Delta_H : \langle \alpha, \lambda \rangle > i \}$.  Each $\alpha$ makes a contribution of $\langle \alpha, \lambda \rangle$ to the sum if $\langle \alpha, \lambda \rangle \ge 0$, and 0 otherwise.

\end{proof}

\subsection{Complexity bounds for flag varieties}
\label{sec:ratlflag}

We now show that
\be
\label{flagint1}
\# y \iota( (\gH_k / Q_{H,\lambda})( \widetilde{k}) ) \cap (\gG_k / Q_\lambda)(k) \ll q^{ \dim( \gH_k / Q_{H,\lambda}) }.
\ee
This bound is what one would naively expect from dimension considerations, and it follows from Lemma \ref{pointcount} once we bound the complexity of $\iota(\gH_k / Q_{H,\lambda})$ in $\gG_k / Q_\lambda$.

As we shall only work over $\tk$ in this section, we simplify notation and denote $\gG_k$, $\gH_k$, $Q_\lambda$ and $Q_{H,\lambda}$ by $G$, $H$, $Q$, and $Q_H$ respectively, with Lie algebras $\g, \gh, \gq$, and $\gq_H$.  We recall that $Q_H(\tk) = Q(\tk) \cap H(\tk)$ and $\gq_H = \gh \cap \gq$.  Let $N = \binom{\dim G}{\dim Q} - 1$, and identify $\bbP^N$ with $\bbP \bigwedge^{\dim Q} \g$ so that $G$ acts on $\bbP^N$.  Let $x = \bigwedge^{\dim Q} \gq \in \bbP^N(k)$, and let $\phi_0 : G \to \bbP^N$ be the orbit map $g \mapsto g.x$.  The induced map on points factors through a map $\phi_G : (G/Q)(\tk) \to \bbP^N(\tk)$, and we may likewise define $\phi_H : (H/Q_H)(\tk) \to \bbP^N(\tk)$.  The universal property of quotients implies that $\iota$, $\phi_H$, and $\phi_G$ may be upgraded to morphisms of varieties, rather than just maps on points.  We define $J = \phi_H( H / Q_H)$, which is closed and irreducible as $H/Q_H$ is complete and irreducible.

It may be seen that if $p \in (H/Q_H)(\tk)$ satisfies $y \iota(p) \in (G/Q)(k)$, then $y \phi_H(p) = \phi_G (y \iota(p)) \in \bbP^N(k)$.  In particular, we have
\[
y \iota( (H / Q_H)( \widetilde{k}) ) \cap (G / Q)(k) \subset y \phi_H( (H/Q_H)(\tk)) \cap \bbP^N(k) = y J(\tk) \cap \bbP^N(k).
\]
Because $\dim J \le \dim(H/Q_H)$, the bound (\ref{flagint1}) now follows by applying Lemma \ref{pointcount} to the variety $y J(\tk) \subset \bbP^N$, once we have a bound on the complexity of $J$ in $\bbP^N$.  If we let $\text{Ad}$ be the adjoint map $H \to GL(\g)$, then $J$ is the image of $\text{Ad}(H)$ under the orbit map $GL(\g) \to \bbP^N$.  Because the complexity of the orbit map is clearly bounded in terms of $\dim G$, Lemma \ref{composecomplex} implies that the complexity of $J$ is bounded in terms of that of $\text{Ad}(H)$.  This completes the proof of (\ref{flagint1}).

\subsection{Proof of Proposition \ref{y1}}
\label{y1proof}

We first prove the claim that $\Lambda_\mu$ is invariant under $W_H$.  If we let $\lambda \in \Lambda_\mu$ and $w_H \in W_H$, we wish to show that $w_H \lambda \in \Lambda_\mu$, or equivalently that $w_H \lambda$ and $\lambda$ lie in the same $W$ orbit.  We may find a representative $w_H^* \in N_{\gH(K)}( \gT_H)$ for $w_H$, so that $w_H \lambda = w_H^* \lambda (w_H^*)^{-1}$.  After considering $w_H^*$ as an element of $\gG(K)$ this implies that $\lambda$ and $w_H \lambda$ are conjugate in $\gG(K)$, and this implies they are conjugate under $W$ by a standard argument which we now recall.

Let $\mathfrak{Z}$ denote the centralizer in $\gG_K$ of the image of the cocharacter $w_H^* \lambda (w_H^*)^{-1}$.  Because $w_H^* \lambda (w_H^*)^{-1}$ lies in both $\gT_K$ and $w_H^* \gT_K (w_H^*)^{-1}$, these are both split maximal tori in $\mathfrak{Z}$.  Because $\mathfrak{Z}$ is connected and reductive, $\gT_K$ and $w_H^* \gT_K (w_H^*)^{-1}$ are conjugate by some $z \in \mathfrak{Z}(K)$.  This implies that $n = z w_H^* \in N_{\gG(K)}( \gT_K)$, and $n \lambda n^{-1} = z w_H^* \lambda (w_H^*)^{-1} z^{-1} = w_H^* \lambda (w_H^*)^{-1}$ as required.

By the invariance of $\Lambda_\mu$, we have
\be
\label{balldecomp}
( \gH(K) \gG(\cO) \cap \gG(\cO) \mu(\varpi) \gG(\cO) ) / \gG(\cO) = \bigcup_{ \lambda \in \Lambda_\mu / W_H} \gH(\cO) \lambda(\varpi) \gG(\cO) / \gG(\cO)
\ee
as in Lemma \ref{Cartancomp}.  We also see that $\gH(\cO) \lambda(\varpi) \gG(\cO) / \gG(\cO)$ is the image of $\gH(\cO) \lambda(\varpi) \gH(\cO) / \gH(\cO)$ under the injection $\gH(K) / \gH(\cO) \to \gG(K) / \gG(\cO)$, which implies that the union on the right hand side of (\ref{balldecomp}) is disjoint, and
\[
\# ( \gH(K) \gG(\cO) \cap \gG(\cO) \mu(\varpi) \gG(\cO) ) / \gG(\cO) = \sum_{ \lambda \in \Lambda_\mu / W_H} \#( \gH(\cO) \lambda(\varpi) \gH(\cO) / \gH(\cO)).
\]
We count $\#( \gH(\cO) \lambda(\varpi) \gH(\cO) / \gH(\cO))$ using the methods of Section \ref{sec:filter}.  For $n \ge 1$ let $H^0_n = \text{ker}( \gH(\cO) \to \gH(\cO / \p^n) )$, and for $\lambda \in X_*(\gT_H)$ define
\[
P_\lambda^0 = \gH(\cO) \cap \lambda(\varpi) \gH(\cO) \lambda(\varpi)^{-1}, \quad F_\lambda^0 = \gH(\cO) / P_\lambda^0, \quad F^0(\lambda, i) = \gH(\cO) / P_\lambda^0 H^0_i.
\]
We now let $\pi_i$ be the projection $F^0(\lambda, i) \to F^0(\lambda, i-1)$, and let $Q_{H,\lambda}$ be the parabolic subgroup of $\gH_k$ associated to $\lambda$.  As in Proposition \ref{Ffibers} we have $F^0(\lambda, 1) \simeq \gH(k) / Q_{H,\lambda}(k)$, and the fibers of $\pi_i$ each have the same cardinality as that of $\gh(k) / ( \gh(k) \cap \gh(\lambda, i) )$, which is $q^{\dim \overline{\gh} / \gh(\lambda, i)}$.  This implies that
\[
\#( \gH(\cO) \lambda(\varpi) \gH(\cO) / \gH(\cO)) = q^{ 2 \| \lambda \|^*_H} \frac{ \# \gH(k) / q^{\textup{dim} \gH} }{ \# Q_{H,\lambda}(k) / q^{ \textup{dim} Q_{H,\lambda}} },
\]
which completes the proof of Proposition \ref{y1}.

\section{Constructing an amplifier}
\label{sec7}

This section contains the construction of the amplifier used in Section \ref{sec3amp}.  Let $F$ be a $p$-adic field with integer ring $\cO$, uniformizer $\varpi$, and residue field of cardinality $q$.  Let $\bG$ be a split semisimple algebraic group over $F$.  Let $\bB = \bT \bU$ be a Borel subgroup defined over $F$, where $\bT$ is a maximal split torus and $\bU$ is the unipotent radical of $\bB$.  Let $K$ be a hyperspecial subgroup of $\bG(F)$ corresponding to a point in the apartment of $\bT$.  Let $\Delta \subset X^*(\bT)$ be the roots of $\bT$ in $\bG$, and let $\Delta^+$ be the positive roots for $\bB$.  Let $\rho$ be the half sum of the roots in $\Delta^+$.  Define $X_*^+(\bT)$ to be the set

\bes
X_*^+(\bT) = \{ \mu \in X_*(\bT) : \langle \mu, \alpha \rangle \ge 0, \alpha \in \Delta^+ \}.
\ees
Let $W$ be the Weyl group of $(\bG,\bT)$.  Define the norm $\| \cdot \|^*$ on $X_*(\bT)$ as in Section \ref{alggps}.  Define $\cH = C_0( K \backslash \bG(F) / K)$.  If $\lambda \in X_*(\bT)$, define $\tau(\lambda) = q^{- \| \lambda \|^*} 1_{K\lambda(\varpi)K} \in \cH$.  For $\kappa \ge 0$, define the truncated Hecke algebra $\cH^{\le \kappa} = \text{span} \{ \tau(\mu) : \mu \in X_*^+(\bT), \| \mu \|^* \le \kappa \}$.  The main result of this section is the following.

\begin{prop}
\label{satake}

There exist $C_1, C_2, C_3 > 0$ depending only on $| F : \Q_p |$ and $\dim(\bG)$ with the following property.  Let $\theta$ be a character of $\cH$, let $\mu \in X_*(\bT)$, and let $\kappa = C_1 \| \mu \|^* + C_2$.  There exists $\tau = \tau_0 + \sigma \in \cH$ with the following properties:

\begin{enumerate}

\item\label{tau1} $\theta(\tau) = 1$.
\item\label{tau2} $\tau$, $\tau^0$, $\sigma$, and $\tau \tau^*$ all lie in $\cH^{\le \kappa}$.
\item\label{tau3} $\| \tau \|_2, \| \tau \tau^* \|_2 \ll 1$ and $\| \sigma \|_2 \ll q^{-1}$.

\item\label{tau4} $\displaystyle \tau^0 = \sum_{1 \le |j| \le |W|} \sum_{ \| \lambda \|^* < C_3 } c(j,\lambda) \tau(j \mu + \lambda)$ with $c(j,\lambda) \ll 1$.

\end{enumerate}

\noindent
The implied constants in (\ref{tau3}) depend only on $| F : \Q_p |$, $\dim(\bG)$, and $\| \mu \|^*$.

\end{prop}

\subsection{Background on the spherical transform}

We let $dg$ be the Haar measure on $\bG(F)$ that gives mass 1 to $K$.  Let $\hG$ and $\hT$ be the Langlands dual groups of $G$ and $T$, and let $\hT_c$ be the maximal compact subgroup of $\hT$.  We have $\hT \simeq \text{Hom}( \bT(F) / \bT(\cO), \C^\times)$, where $\bT(\cO)$ denotes the maximal compact subgroup of $\bT(F)$, via the maps
\be
\label{torusisoms}
\text{Hom}( \bT(F) / \bT(\cO), \C^\times) \simeq \text{Hom}( X_*(\bT), \C^\times) \simeq \text{Hom}( X^*(\hT), \C^\times) \simeq X_*(\hT) \otimes_\Z \C^\times \simeq \hT.
\ee
Here, the second isomorphism is induced by the map $X_*(\bT) \to \bT(F)$ sending $\mu$ to $\mu(\varpi)$.

Given $\nu \in \hT$, we may define a character $\theta_\nu : \cH \rightarrow \C$ in the following way.  Consider $\nu$ as a character $\chi_\nu$ of $\bT(F)$, and let $\pi_\nu$ be the unique spherical subquotient of the normalised induction of $\chi_\nu$ from $\bB$ to $\bG$.  We then choose a nonzero $v \in \pi_\nu^K$, and define $\theta_\nu$ by $\theta_\nu(k) v = \pi_\nu(k) v$ for $k \in \cH$.  The following facts are standard, see e.g. $\mathsection$6, $\mathsection$7, and $\mathsection$10.4 of \cite{Bo}.

\begin{enumerate}[(i)]

\item $\nu \mapsto \theta_\nu$ defines a bijection between $\hT / W$ and characters of $\cH$.

\item For every $\nu \in \hT / W$, there is a function $\varphi_\nu \in C^\infty(K \backslash \bG(F) / K)$ such that $\theta_\nu(k) = \int k(g) \varphi_\nu(g) dg$.

\item If $k \in \cH$, we define $\widehat{k}(\nu) = \int k(g) \varphi_\nu(g) dg$.  There is a probability measure $\mu_\text{pl}$ on $\hT_c$ such that

\be
\label{sphinvert}
k(g) = \int_{\hT_c} \widehat{k}(\nu) \varphi_{-\nu}(g) d\mu_\text{pl}(\nu)
\ee
and

\be
\label{planch}
\| k \|_2^2 = \int_{\hT_c} | \widehat{k}(\nu) |^2 d\mu_\text{pl}(\nu).
\ee

\item We have $\widehat{k^*}(\nu) = \overline{\widehat{k}}(\nu)$.

\end{enumerate}

We shall use a formula for the functions $\varphi_\nu$ when $\nu$ is nonsingular, due to Macdonald \cite{Mc2} when $\bG$ is simply connected, and Casselman \cite{Cs} for general $\bG$.  We define the function $c(\nu)$ on the nonsingular set in $\hT$ by
\[
c(\nu) = \prod_{\alpha \in \Delta^+} \frac{1 - q^{-1} \alpha^\vee(\nu)}{1 - \alpha^\vee(\nu)}.
\]
The formula of Macdonald and Casselman then states that if $\nu$ is nonsingular and $\mu \in X_*^+(\bT)$, then
\be
\label{Mcdonald}
\varphi_\nu( \mu(\varpi)) = Q^{-1} q^{-\langle \mu, \rho \rangle} \sum_{w \in W} c(-w\nu) \mu( w \nu),
\ee
where $Q = 1 + O_{\text{dim}(\bG)}(q^{-1})$ is a constant depending on $q$ and the root system of $G$.  Indeed, this follows from \cite[Theorem 4.2]{Cs}, once we apply Remark 1.1 there that $a_\alpha = \alpha^\vee(\varpi)$, and the identity $\chi_\nu( \alpha^\vee(\varpi)) = \alpha^\vee(\nu)$ which follows from (\ref{torusisoms}).

We may express $\mu_\text{pl}$ in terms of $c(\nu)$ as

\[
d\mu_\text{pl} = C |c(\nu)|^{-2} d\nu,
\]
where $d\nu$ is a Haar probability measure on $\hT_c$, and $C$ is chosen to make $\mu_\text{pl}$ a probability measure.  This follows from the formula
\[
d\mu_\text{pl}( \nu) = C \frac{ \det( 1 - \text{ad}( \nu) | \text{Lie}( \hG) / \text{Lie}(\hT) )}{ \det( 1 - q^{-1} \text{ad}( \nu) | \text{Lie}( \hG) / \text{Lie}(\hT) )} d\nu,
\]
which Shin and Templier deduce in \cite[Prop. 3.3]{ST} from the results of \cite{Sd}, and observing that the ratio on the right hand side is equal to $c(\nu)^{-1} c(-\nu)^{-1} = |c(\nu)|^{-2}$.

We will also need the following Paley-Wiener theorem for the spherical transform.  Define a partial order on $X_*(\bT)$ by saying that $\lambda \ge \mu$ if $\lambda - \mu$ is a nonnegative linear combination of positive coroots.  Define $\cH_T = C_0( \bT(F) / \bT(\cO) )$, and if $\lambda \in X_*(\bT)$ define $\tau_T(\lambda) = 1_{\lambda(\varpi)\bT(\cO)} \in \cH_T$.  If $\lambda \in X^+_*(\bT)$, define the truncated Hecke algebras

\begin{align*}
\cH_\lambda & = \text{span} \langle \tau(\mu) : \mu \in X^+_*(\bT), \mu \le \lambda \rangle, \\
\cH_{T,\lambda} & = \text{span} \langle \tau_T(\mu) : \mu \in X_*(\bT), w \mu \le \lambda \text{ for all } w \in W \rangle.
\end{align*}
There is a natural isomorphism $\cH_T \simeq \C[X^*(\hT)]$, which may be viewed as a Fourier transform, and we define $\C[X^*(\hT)]_\lambda$ to be the subspace corresponding to $\cH_{T,\lambda}$.

\begin{lemma}
\label{paleywiener}

The map $k \mapsto \widehat{k}$ defines an isomorphism between $\cH_\lambda$ and $\C[X^*(\hT)]_\lambda^W$.

\end{lemma}

\begin{proof}

The spherical transform is the composition of the Satake isomorphism $\cS : \cH \rightarrow \cH_T^W$ with the Fourier transform $\cH_T^W \simeq \C[X^*(\hT)]^W$.  The result follows from the fact that $\cS$ gives an isomorphism between $\cH_\lambda$ and $\cH_{T,\lambda}^W$, see for instance \cite[p.148]{Ca}.

\end{proof}

\subsection{Proof of Proposition \ref{satake}}

Choose $\mu \in X^+_*(\bT)$.  Let $\nu_0 \in \hT$ correspond to the character $\theta$.  Following \cite[Sec. A.4]{SV1}, we define $h \in \C[ X^*(\hT) ]$ by

\bes
h = \sum_{1 \le |j| \le |W|} \overline{\sum_{w_1 \in W} j w_1 \mu( \nu_0)} \sum_{w_2 \in W} j w_2 \mu.
\ees
By applying the following lemma with $z_w = w \mu( \nu_0)$ for $w \in W$, we see that
\[
h(\nu_0) = \sum_{1 \le |j| \le |W|} \left| \sum_{w \in W} j w \mu( \nu_0) \right|^2 \ge C( \bG).
\]

\begin{lemma}

If $z_1, \ldots, z_m \in \C^\times$, there is $c(m) > 0$ such that
\bes
\sum_{1 \le |j| \le m} | z_1^j + \ldots + z_m^j | \ge c(m) > 0.
\ees

\end{lemma}

\begin{proof}

Let $z = (z_1, \ldots, z_m)$.  If we have
\[
\sum_{1 \le j \le m} | z_1^j + \ldots + z_m^j | = 0
\]
then $z = 0$.  Therefore, by compactness of the set $\| z \| = 1$, and homogeneity, we have
\[
\sum_{1 \le j \le m} | z_1^j + \ldots + z_m^j | \ge c(m) \min \{ \| z \|, \| z \|^m \}.
\]
If we define $z^{-1} = (z_1^{-1}, \ldots, z_m^{-1})$, then the lemma follows by applying the above inequality to $z$ and $z^{-1}$, together with $\| z \| \| z^{-1} \| \ge \langle z, z^{-1} \rangle = m$.

\end{proof}

We define $\tau \in \cH$ by $\widehat{\tau} = h / h(\nu_0)$.  To define $\tau^0$ and $\sigma$, we note that
\be
\label{cvdecomp}
c(-\nu)^{-1} = \prod_{\alpha > 0} \frac{1 - \alpha^\vee(-\nu) }{1 - q^{-1} \alpha^\vee(-\nu) } = \prod_{\alpha > 0} (1 - \alpha^\vee(-\nu)) + E(\nu)
\ee
where $\| E|_{\hT_c} \|_\infty \ll q^{-1}$.  We then set $\tau^0$ and $\sigma$ to be the functions on $K \backslash G / K$ defined by
\begin{align}
\label{tau0def}
\tau^0( \lambda(\varpi)) & = \frac{ q^{-\langle \rho, \lambda \rangle} }{ h(\nu_0)} \int_{\hT_c} h(\nu) \prod_{\alpha > 0} (1 - \alpha^\vee(-\nu)) \lambda( -\nu) d\nu, \\
\notag
\sigma( \lambda(\varpi)) & = \frac{ q^{-\langle \rho, \lambda \rangle} }{ h(\nu_0)} \int_{\hT_c} h(\nu) E(\nu) \lambda( -\nu) d\nu,
\end{align}
for $\lambda \in X_*^+(\bT)$.  These definitions are partly explained by the following lemma.

\begin{lemma}

We have $\tau = \tau^0 + \sigma$.

\end{lemma}

\begin{proof}

Equation (\ref{cvdecomp}) implies that

\[
\tau^0(\lambda(\varpi)) + \sigma( \lambda(\varpi)) = \frac{ q^{-\langle \rho, \lambda \rangle} }{ h(\nu_0)} \int_{\hT_c} h(\nu) c(-\nu)^{-1} \lambda( -\nu) d\nu
\]
for $\lambda \in X_*^+(\bT)$.  Because $d\lambda_\text{pl} = | c(\nu) |^{-2} d\nu$, we may rewrite this as
\[
\tau^0(\lambda(\varpi)) + \sigma( \lambda(\varpi)) = \frac{ q^{-\langle \rho, \lambda \rangle} }{ h(\nu_0)} \int_{\hT_c} h(\nu) c(\nu) \lambda( -\nu) d\lambda_\text{pl}(\nu).
\]
The $W$-invariance of $h$ means that we may introduce an average over $W$ into the integral, and the lemma now follows from the formula (\ref{Mcdonald}) for $\varphi_{-\nu}$ and the spherical inversion formula for $\tau$.

\end{proof}

We must now show that the three functions we have defined have the required properties.  Condition (\ref{tau1}) is immediate.

\noindent
$\bullet$ {\bf Condition (\ref{tau4}):} We begin with the following bound for $h$.

\begin{lemma}
\label{hbound}

There are constants $c(j) \ll 1$ such that
\be
\label{hexpand}
h / h(\nu_0) = \sum_{1 \le |j| \le |W|} c(j) \sum_{w \in W} j w\mu.
\ee
In particular, $\| h / h(\nu_0)|_{\hT_c} \|_\infty \ll 1$.

\end{lemma}

\begin{proof}

If we choose $c(j) = (\sum_{w \in W} j w \mu( \nu_0) )/h(\nu_0)$, then the expansion (\ref{hexpand}) holds.  On one hand, we have
\[
|c(j)| \le |\sum_{w \in W} j w \mu( \nu_0)| / C(\bG).
\]
On the other hand, substituting the formula for $h(\nu_0)$ and dropping all terms except the one corresponding to $j$, we have
\[
|c(j)| \le \frac{ \left| \sum_{w \in W} j w \mu( \nu_0) \right| }{ \left| \sum_{w \in W} j w \mu( \nu_0) \right|^2 } = \left| \sum_{w \in W} j w \mu( \nu_0) \right|^{-1}.
\]
Combining these gives the lemma.

\end{proof}

If we substitute the expansion of Lemma \ref{hbound} in (\ref{tau0def}), we obtain
\[
\tau^0( \lambda(\varpi)) = q^{-\langle \rho, \lambda \rangle} \sum_{1 \le |j| \le |W|} c(j) \sum_{w \in W} \int_{\hT_c} j w \mu(\nu) \prod_{\alpha > 0} (1 - \alpha^\vee(-\nu)) \lambda( -\nu) d\nu, \quad \lambda \in X_*^+(\bT).
\]
We expand the product over $\alpha$ as
\[
\prod_{\alpha > 0} (1 - \alpha^\vee) = \sum_{\beta \in \Xi} r(\beta) \beta
\]
for a subset $\Xi \subset X_*(\bT)$ and $r(\beta) \neq 0$.  The contribution from a term $j$, $w \in W$ and $\beta \in \Xi$ to $\tau^0(\lambda(\varpi))$ is zero unless $\lambda = j w\mu - \beta$, in which case it is $q^{-\langle \rho, \lambda \rangle} c(j) r(\beta)$.  It follows that $\tau_0$ is the sum over $1 \le |j| \le |W|$, $w \in W$, and $\beta \in \Xi$ satisfying $j w\mu - \beta \in X_*^+(\bT)$ of $\tau(j w\mu - \beta) c(j) r(\beta) = \tau(j \mu - w^{-1} \beta) c(j) r(\beta)$.  We have $\| w^{-1} \beta \|^* < C_3$ for some $C_3 > 0$, so $\tau_0$ has an expansion as in condition (\ref{tau4}).

\noindent
$\bullet$ {\bf Condition (\ref{tau3}):} It follows from Lemma \ref{paleywiener} that $\tau, \tau \tau^* \in \cH^{\le \kappa}$ if $\kappa = C_1 \| \mu \|^* + C_2$ for suitable $C_1$ and $C_2$.  Condition (\ref{tau4}) implies that the same is true for $\tau^0$, and $\tau = \tau^0 + \sigma$ implies it for $\sigma$.

\noindent
$\bullet$ {\bf Condition (\ref{tau2}):} The bounds $\| \tau \|_2, \| \tau \tau^* \|_2 \ll 1$ follow from the Plancherel theorem and the bound $\| h / h(\nu_0)|_{\hT_c} \|_\infty \ll 1$ of Lemma \ref{hbound}.  To prove $\| \sigma \|_2 \ll q^{-1}$, we have
\[
\sigma( \lambda(\varpi)) = q^{-\langle \rho, \lambda \rangle} \int_{\hT_c} \frac{h(\nu)}{h(\nu_0)} E(\nu) \lambda( -\nu) d\nu, \quad \lambda \in X_*^+(\bT),
\]
and combining this with $\| E|_{\hT_c} \|_\infty \ll q^{-1}$ and $\| h / h(\nu_0)|_{\hT_c} \|_\infty \ll 1$ gives $\sigma( \lambda(\varpi)) \ll q^{-\langle \rho, \lambda \rangle - 1}$.  The required bound follows from this, together with the asymptotic $\text{vol}( K \lambda(\varpi) K) \sim q^{2\langle \rho, \lambda \rangle}$ from Corollary \ref{Gy1}, and $\sigma \in \cH^{\le \kappa}$.

\section{Small subgroups}
\label{sec8}

In Sections \ref{sec8} and \ref{algsec}, we prove that condition ($\mathsf{WS}$) is equivalent to the cocharacter inequalities used in Section \ref{sec3}.  We also prove a related equivalence that is used in \cite{BM}.  Our main result is Theorem \ref{smallequiv}, which is stated in terms of semisimple groups with involution over a general field of characteristic 0.  In Corollary \ref{realsmall} we apply this theorem to a real group with its Cartan involution, to obtain the equivalence we need.

The structure of these sections is as follows.  In Section \ref{sec8invol} we recall some results about semisimple groups $G$ with involution $\theta$, and in Section \ref{sec8small} we define three conditions on a connected reductive subgroup $H$ of $G$, called small, quasi-small, and weakly small.  In Section \ref{sec8results} we state our main result, which characterizes when $(G^\theta)^0 < G$ satisfies these smallness conditions in terms of $\theta$.  In section \ref{sec8reduction} we reduce the main theorem to a statement about complex semisimple Lie algebras, and deduce it in many cases from work of Benoist and Kobayashi \cite{BK}.  We verify the remaining cases by hand in Section \ref{algsec}.

\subsection{Involutions of semisimple groups}
\label{sec8invol}

Let $G$ be a semisimple group over an algebraically closed field $F$ of characteristic 0.  Let $\theta$ be an involution of $G$.  If $T < G$ is a $\theta$-stable torus, we say that $T$ is $\theta$-split if $\theta$ acts on $T$ by inversion.  We say that $G$ is $\theta$-split if $G$ has a $\theta$-split maximal torus.  We say that $G$ is $\theta$-quasi-split if it has a Borel subgroup $B$ such that $B$ and $\theta(B)$ are opposed.

In \cite{Hl}, it is shown that pairs $(G,\theta)$, where $G$ is semisimple over $F$ and $\theta$ is a class of involutions under conjugacy in $G$, are classified by the same combinatorial data over any algebraically closed $F$ of characteristic not 2.  We briefly recall this data, called the index, and a weaker form called the diagram.

If $\Psi$ is a root datum with an involution $\theta$, the index of $\theta$ is a set of combinatorial data that determine $\Psi$ and $\theta$ up to the action of the Weyl group of $\Psi$.  If $(G,\theta)$ is a semisimple group with involution, Helminck defines its index by choosing a $\theta$-stable maximal torus such that $T^-$ is maximal $\theta$-split, letting $\Psi$ be the root datum of $T$ in $G$, and defining the index of $(G,\theta)$ to be that of $(\Psi,\theta)$.  In \cite[Theorem 3.11]{Hl}, Helminck shows that the index determines $(G,\theta)$ up to inner isomorphism, and he also determines the set of indices arising from semisimple groups with involution.  He associates to an index a weaker invariant called the diagram of $(G,\theta)$, which determines the pair up to isogeny.  If $G$ is a real semisimple group, we define the diagram of $G$ to be the diagram of $G_\C$ together with the Cartan involution of $G$.

\begin{lemma}
\label{index}

The index of $(G,\theta)$ is invariant under extension of algebraically closed fields, and determines whether $(G,\theta)$ is (quasi-) split.

\end{lemma}

\begin{proof}

Let $T$ be a maximal torus of $G$ such that $T^-$ is maximal $\theta$-split.  Let $F'$ be an algebraically closed field containing $F$.  By \cite[Prop. 2]{Vu}, $T^-$ is still a maximal $\theta$-split torus of $G \times F'$, and so we can also use $T$ to define the index of $(G \times F', \theta)$.  It follows that the index is the same over $F$ and $F'$.  Because the index determines the action of $\theta$ on $T$, it determines $T^-$ and hence whether $G$ is $\theta$-split.  The index also determines the roots of $T^-$ in $G$, and hence whether $Z_G(T^-)$ is a torus.  By \cite[p. 21, Corollaire]{Vu}, this determines whether $(G,\theta)$ is quasi-split.

\end{proof}

\begin{lemma}
\label{realsplit}

Let $G$ be a semisimple group over $\R$ with Cartan involution $\theta$.  $(G,\theta)$ is (quasi-) split if and only if $G$ is (quasi-) split over $\R$.

\end{lemma}

\begin{proof}

Let $\g = \p + \gk$ be the Cartan decomposition induced by $\theta$, and let $\ga \subset \p$ be a maximal Abelian subalgebra.  By \cite[Ch. 24.6 (e)]{Bo}, there is a maximal $\R$-split torus $T$ of $G$ with Lie algebra $\ga$.  It follows that $T$ is $\theta$-split, and because $Z_\g(\ga) = \ga \oplus Z_\gk(\ga)$, $T_\C$ must also be a maximal $\theta$-split torus of $G_\C$.  These imply that $G$ is split over $\R$ if and only if it is $\theta$-split.

Because $Z_G(T)$ is the Levi of a minimal $\R$-parabolic subgroup of $G$, $G$ is quasi-split if and only if $Z_G(T)$ is a torus.  Because $T_\C$ is a maximal $\theta$-split torus of $G_\C$, \cite[p. 21, Corollaire]{Vu} implies that $G$ is also $\theta$ quasi-split if and only $Z_G(T)$ is a torus.

\end{proof}

\subsection{Smallness}
\label{sec8small}

Let $G$ be a connected reductive group over an algebraically closed field $F$ of characteristic 0.  Let $X_*(G)$ be the set of cocharacters of $G$.  We define a function $\| \cdot \|^*$ on $X_*(G)$ as follows.  If $\mu : \mathbb{G}_m \to G$ is a cocharacter, we obtain a representation $\text{Ad} \circ \mu : \mathbb{G}_m \to GL(\g)$, and we define $\| \mu \|^* \in X^*(\mathbb{G}_m) \otimes \R \simeq \R$ to be half the sum of the positive weights of this representation.  If $T < G$ is a maximal torus, the restriction of $\| \cdot \|^*$ to $X_*(T)$ is also equal to $\| \mu \|^* = \underset{w \in W}{\max} \langle w \mu, \rho \rangle$, where $W$ is the Weyl group and $\rho$ is the half sum of the positive roots.  $\| \cdot \|^*$ is a seminorm on $X_*(T)$, and a norm if $G$ is semisimple; the condition $\| \mu \|^* = \| -\mu \|^*$ holds because $\rho$ and $-\rho$ lie in the same Weyl orbit.

We use the function $\| \cdot \|^*$ to define the following three conditions on reductive subgroups of $G$.

\begin{definition}
\label{small}

Let $H < G$ be connected reductive groups, and define the functions $\| \cdot \|^*$ and $\| \cdot \|^*_H$ on $X_*(G)$ and $X_*(H)$ as above.

\begin{itemize}

\item We say that $H$ is small in $G$ if $\| \mu \|^* > 2 \| \mu \|^*_H$ for all nonzero $\mu \in X_*(H)$.

\item We say that $H$ is quasi-small in $G$ if $\| \mu \|^* \ge 2 \| \mu \|^*_H$ for all $\mu \in X_*(H)$.

\item We say that $H$ is weakly small in $G$ if it is quasi-small, and either $H$ has lower rank than $G$, or there is $\mu \in X_*(H)$ such that $\| \mu^g \|^* > 2 \| \mu^g \|^*_H$ for any $g \in G$ with $\mu^g \in X_*(H)$.

\end{itemize}

\end{definition}

It will be convenient to rephrase these definitions in terms of cocharacters of tori.

\begin{lemma}
\label{smalltori}

Let $G$ and $H$ be as in Definition \ref{small}.  Let $T_H < T$ be maximal tori in $H$ and $G$.

\begin{itemize}

\item $H$ is small in $G$ if and only if $\| \mu \|^* > 2 \| \mu \|^*_H$ for all nonzero $\mu \in X_*(T_H)$.

\item $H$ is quasi-small in $G$ if and only if $\| \mu \|^* \ge 2 \| \mu \|^*_H$ for all $\mu \in X_*(T_H)$.

\item $H$ is weakly small in $G$ if and only if it is quasi-small and either $\dim T_H < \dim T$, or there is $\mu \in X_*(T)$ such that $\| \mu \|^* > 2 \underset{w \in W}{\max} \| w \mu \|^*_H$.

\end{itemize}

\end{lemma}

\begin{proof}

The first two equivalences are immediate from the fact that any $\mu \in X_*(H)$ can be conjugated into $T_H$, and the two seminorms are invariant under conjugation.  For the third, it suffices to assume that $T = T_H$ and prove the equivalence of the following:

\begin{enumerate}[(i)]

\item
\label{mu1}
There is $\mu \in X_*(H)$ such that if $\mu^g \in X_*(H)$, then $\| \mu^g \|^* > 2 \| \mu^g \|^*_H$.

\item
\label{mu2}
There is $\mu \in X_*(T_H)$ such that $\| \mu \|^* > 2 \| w \mu \|^*_H$ for all $w \in W$.

\end{enumerate}

Suppose $\mu \in X_*(H)$ satisfies (\ref{mu1}).  If $h \in H$ is such that $\mu^h \in X_*(T_H)$, we clearly have $\| \mu^h \|^* > 2 \| w \mu^h \|^*_H$ for all $w \in W$ so that (\ref{mu2}) holds.

Suppose $\mu \in X_*(T_H)$ satisfies (\ref{mu2}).  If $g \in G$ satisfies $\mu^g \in X_*(H)$, we wish to show that $\| \mu^g \|^* > 2 \| \mu^g \|^*_H$.  By conjugating in $H$, we may assume that $\mu^g \in X_*(T_H)$.  As the subtori $\mu(\mathbb{G}_m)$ and $\mu^g(\mathbb{G}_m)$ of $T$ are conjugate in $G$, they are conjugate under $W$ by the argument of Section \ref{y1proof}.  We therefore have $w \mu = \mu^g$, and so (\ref{mu2}) implies that $\| \mu^g \|^* > 2 \| \mu^g \|^*_H$ as required.

\end{proof}

We will need the following simple observation:

\begin{lemma}
\label{smallBC}

The conditions of Definition \ref{small} are invariant under extension of algebraically closed fields.

\end{lemma}

\begin{proof}

We use the formulations of Lemma \ref{smalltori}.  Let $F'$ be an algebraically closed field containing $F$.  If $T_H < T$ are maximal tori in $H$ and $G$, then they are still maximal in $H \times F'$ and $G \times F'$, and the seminorms and Weyl groups appearing in Lemma \ref{smalltori} are preserved under field extension.  It follows that the smallness conditions are preserved.

\end{proof}

\begin{remark}

Throughout the rest of the paper, we shall only use the smallness conditions in the form given by Lemma \ref{smalltori}.  We only gave them in the form of Definition \ref{small} to make it clear that they were independent of the choice of tori in Lemma \ref{smalltori}.

\end{remark}

If the field of definition is not algebraically closed, we say that $H$ is small in $G$ if this is true after passing to an algebraic closure.  With this, the condition that $H$ is small in $G$ is invariant under any extension of fields.

\subsection{Statement of results}
\label{sec8results}

Our main result is the following:

\begin{theorem}
\label{smallequiv}

Let $G$ be a semisimple group with involution $\theta$, and let $H = (G^\theta)^0$.  We have the following equivalences.

\begin{itemize}

\item $H$ is small in $G$ if and only if $G$ is $\theta$-split.

\item (Benoist-Kobayashi) $H$ is quasi-small in $G$ if and only if $G$ is $\theta$ quasi-split.

\item $H$ is weakly small in $G$ if and only if $(G,\theta)$ is quasi-split and does not have the same diagram as a product of the groups $SU(n,n-1)$.

\end{itemize}

\end{theorem}

We have attributed the second equivalence to Benoist-Kobayashi because it is implicit in Theorem 4.1 and Example 5.7 of \cite{BK}, as will be explained later.  By applying Theorem \ref{smallequiv} to a real group with its Cartan involution, combined with Lemma \ref{realsplit}, we have:

\begin{cor}
\label{realsmall}

Let $G/\R$ be a connected semisimple group, and let $K/\R$ be a subgroup such that $K(\R)$ is a maximal connected compact subgroup of $G(\R)$.

\begin{itemize}

\item $K$ is small in $G$ if and only if $G$ is $\R$-split.

\item $K$ is quasi-small in $G$ if and only if $G$ is $\R$ quasi-split.

\item $K$ is weakly small in $G$ if and only if $G$ satisfies $(\mathsf{WS})$.

\end{itemize}

\end{cor}

The conditions appearing in Theorem \ref{smallequiv} also arise in the study of the spectra of symmetric varieties, see for instance \cite{BK,GO} and Sections 2 and 3 of \cite{BM}.

\subsection{Preliminary reductions}
\label{sec8reduction}

In this section, we will reduce Theorem \ref{smallequiv} to a statement about complex semisimple Lie algebras with involution.  We begin with the following reduction.

\begin{lemma}

It suffices to prove Theorem \ref{smallequiv} for groups of adjoint type over $\C$.

\end{lemma}

\begin{proof}

Let $F$ be algebraically closed of characteristic 0, and let $(G_F,\theta)$ be a semisimple group with involution over $F$.  By the results of \cite{Hl}, there exists a semisimple group with involution $(G_{\overline{\Q}}, \theta)$ over $\overline{\Q}$ such that $(G_F,\theta)$ is isomorphic to the base change of $(G_{\overline{\Q}}, \theta)$.  By Lemmas \ref{index} and \ref{smallBC}, we may pass all the conditions of Theorem \ref{smallequiv} from $(G_F,\theta)$, to $(G_{\overline{\Q}}, \theta)$, to $(G_{\overline{\Q}} \times \C, \theta)$.  Finally, all conditions of Theorem \ref{smallequiv} are invariant under isogeny, which is clear for the splitness and diagram conditions, and follows from Lemma \ref{smalltori} for the smallness conditions.

\end{proof}

The set of complex adjoint groups with involution is naturally in bijection with complex semisimple Lie algebras with involution, and we next adapt the definitions appearing in Theorem \ref{smallequiv} to the Lie algebra setting.

\begin{lemma}
\label{Liedata1}

Let $\g_\C$ be a complex semisimple Lie algebra with involution $\theta$, and let $\gh_\C = \g_\C^\theta$.  There exist Cartan subalgebras $\gt_{H,\C} \subset \gt_\C$ of $\gh_\C$ and $\g_\C$, and real structures $V$ and $V_H$ on $\gt_\C$ and $\gt_{H,\C}$, such that $V_H \subset V$ and all roots of $\gt_\C$ in $\g_\C$ are real on $V$.

\end{lemma}

\begin{proof}

Let $G$ be the complex adjoint group with Lie algebra $\g_\C$.  We also use $\theta$ to denote the involution of $G$ with differential $\theta$ on $\g_\C$, and let $H = (G^\theta)^0$ so that $\gh_\C = \text{Lie}(H)$.  Let $T_H < T$ be maximal tori in $H$ and $G$, let $\gt_{H,\C} \simeq X_*(T_H) \otimes \C$ and $\gt_\C \simeq X_*(T) \otimes \C$ be their Lie algebras, and let $V_H = X_*(T_H) \otimes \R$ and $V = X_*(T) \otimes \R$.

\end{proof}

\begin{lemma}
\label{Liedata2}

Let $\g_\C$ be a complex semisimple Lie algebra with involution $\theta$, and let $\gh_\C = \g_\C^\theta$.  Any data $\gt_{H,\C}$, $\gt_\C$, $V_H$, and $V$ satisfying the conditions of Lemma \ref{Liedata1} arises from two complex tori as in the proof of that Lemma.

\end{lemma}

\begin{proof}

Let $G$, $\theta$, and $H$ be as in the proof of Lemma \ref{Liedata1}.  By the conjugacy of Cartan subalgebras, there exist complex tori $T_H < T$ in $H$ and $G$ with Lie algebras $\gt_{H,\C}$ and $\gt_\C$.  This determines real structures $V_H'$ and $V'$ on $\gt_{H,\C}$ and $\gt_\C$, but we must have $V_H = V_H'$ and $V = V'$ because the roots of $\gt_\C$ in $\g_\C$ span $\gt_\C^*$.

\end{proof}

Now let $\g_\C$, $\theta$, $V_H$, and $V$ be as in Lemma \ref{Liedata1}.  Lemma \ref{Liedata2} implies that $V_H$ and $V$ are invariant under the relevant Weyl groups.  We define functions $\| \cdot \|^*$ and $\| \cdot \|^*_H$ on $V$ and $V_H$ in analogy with the case of tori.  We adapt the definitions of (weakly, quasi-) small to $\gh_\C$ in $\g_\C$ in the natural way.  We also say that $\g_\C$ is $\theta$ (quasi-)split if and only if the associated pair $(G,\theta)$ is.  By Lemmas \ref{smalltori} and \ref{Liedata2}, $(\g_\C, \theta)$ is (weakly, quasi-) small if and only if $(G,\theta)$ is, so that we have reduced Theorem \ref{smallequiv} to the following statement about complex Lie algebras.

\begin{prop}
\label{Liesmall}

Let $(\g_\C, \theta)$ be a semisimple complex Lie algebra with involution.  Then $(\g_\C, \theta)$ is (quasi-)small if and only if it is (quasi-)split, and weakly small if and only if it is quasi-split and not a product of simple factors of type $\mathfrak{su}(n,n-1)$.

\end{prop}

We shall prove Proposition \ref{Liesmall} by reducing to the case where $\g_\C$ is simple and $\theta$ quasi-split, and then computing the remaining cases in Section \ref{algsec}.

The pair $(\g_\C, \theta)$ breaks up as a direct sum of simple pairs, which are either of the form $(\g_\C', \theta')$ with $\g_\C'$ simple, or of the form $(\g_\C' \oplus \g_\C', \theta')$ where $\g_\C'$ is simple and $\theta'$ switches the two factors.  The conditions (quasi-)split and (quasi-)small hold for $(\g_\C, \theta)$ if and only if they hold for every simple factor.  Likewise, the extra conditions defining weakly small hold if and only if they hold for just one simple factor, and likewise for the condition of not being a product of type $\mathfrak{su}(n,n-1)$.  It follows that it suffices to prove the equivalences of Proposition \ref{Liesmall} for simple $(\g_\C, \theta)$.

Suppose $(\g_\C, \theta)$ has the form $(\g_\C' \oplus \g_\C', \theta')$ where $\g_\C'$ is simple and $\theta'$ switches the two factors.  In this case, $\gh_\C$ is the diagonal copy of $\g_\C'$, and we have $\| \cdot \|^* = 2 \| \cdot \|^*_H$ so that $(\g_\C, \theta)$ is weakly and quasi-small but not small.  Correspondingly, the associated $(G,\theta)$ is weakly and quasi-split but not split, so Proposition \ref{Liesmall} holds in these cases.

We shall therefore assume that $\g_\C$ is simple from now on.  Because quasi-splitness, resp. quasi-smallness, is the weakest of the three conditions of its type appearing in Proposition \ref{Liesmall}, the following lemma allows us to assume that $(\g_\C,\theta)$ is quasi-split.

\begin{lemma}

$(\g_\C, \theta)$ is quasi-small if and only if it is quasi-split.

\end{lemma}

\begin{proof}

This follows from work of Benoist and Kobayashi \cite{BK} on the spectra of reductive homogeneous spaces, after translating their results into our language.

Let $\g$ be the real Lie algebra corresponding to $(\g_\C, \theta)$, and let $G$ be the real adjoint group with Lie algebra $\g$.  Let $K$ be a maximal connected compact subgroup of $G$, with Lie algebra $\gk$.  Let $T_K \subset T$ be maximal $\R$-tori in $K$ and $G$.  The space $V_K = X_*(T_K) \otimes_\Z \R$ is naturally isomorphic to a maximal $\R$-split abelian subalgebra of $\gk_\C$ (considered as a real Lie algebra).  Let $\gq_\C$ be the $\gk_\C$-stable complement to $\gk_\C$ in $\g_\C$.  We define the functions $\rho_\gk$ and $\rho_\gq$ on $V_K$ as in \cite[Section 3.1 and 4.1]{BK}, by considering $\gk_\C$ and $\gq_\C$ as real representations of $\gk_\C$.  (Note that $\rho_\gk$ is denoted $\rho_\gh$ in \cite{BK}.)  We then have $\rho_\gk = 4\| \cdot \|_K^*$ and $\rho_\gq = 4\| \cdot \|^* - 4\| \cdot \|_K^*$.  By \cite[Theorem 4.1]{BK}, the representation of $G(\C)$ on $L^2(G(\C) / K(\C) )$ is tempered if and only if $\rho_\gq(t) \ge \rho_\gk(t)$ for all $t \in V_K$, which is equivalent to $(\g_\C, \theta)$ being quasi-small.  Combining this with \cite[Example 5.7]{BK}, we see that $(\g_\C, \theta)$ is quasi-small if and only if $G$, and hence $(\g_\C, \theta)$, are quasi-split.

\end{proof}

\section{Lie algebra computations}
\label{algsec}

We now finish the proof of Proposition \ref{Liesmall}, by checking it directly when $\g_\C$ is simple and $\theta$ quasi-split.  Rather than listing pairs $(\g_\C, \theta)$, we shall list the associated real Lie algebra $\g$.

With notation as in Lemma \ref{Liedata1}, let $\Delta$ and $\Delta_K$ be the roots of $V$ in $\g_\C$ and $V_K$ in $\gk_\C$.  Let $W$ and $W_K$ be the Weyl groups they generate.  We choose systems of positive roots $\Delta^+$ and $\Delta_K^+$, with corresponding closed positive Weyl chambers $V^+ \subset V$ and $V_K^+ \subset V_K$.  We let $\rho$ and $\rho_K$ be the half sums of $\Delta^+$ and $\Delta^+_K$.  For each $\g$ we consider, we shall choose an identificaion $V_K \simeq \R^m$, with basis $x_1, \ldots, x_m$ and dual basis $\xi_1, \ldots, \xi_m$, and likewise for $V \simeq \R^n$ with basis $y_1, \ldots, y_n$ and dual basis $\eta_1, \ldots, \eta_n$.

\subsection{Split classical Lie algebras}

We begin with the cases where $\g$ is split and classical.  In this case, all three equivalences of Proposition \ref{Liesmall} follow if we know $(\g_\C, \theta)$ is small.  We must therefore prove that

\be
\label{small2}
\| t \|^* > 2\| t \|_K^* \text{ for all nonzero } t \in V_K.
\ee
If $V_K = V$, we have $W_K \subset W$.  In the remaining cases, when $\g$ is $\gsl(n)$ or $\mathfrak{so}(2k+1,2k+1)$, we also have $W_K \subset W$ in the sense that the stabiliser of $V_K$ in $W$ contains $W_K$.  We may therefore assume that $t \in V_K^+$ and $t \neq 0$.  We let $t' = Wt \cap V^+$, so that (\ref{small2}) is equivalent to proving that $\langle t', \rho \rangle > 2 \langle t, \rho_K \rangle$ for nonzero $t \in V_K^+$.  Our choices of $\Delta^+$ and $\Delta_K^+$ will be the standard ones in all classical cases.

\subsection{$\gsl(2k)$}

We have $V_K \simeq \R^{k}$ and $V \simeq \R^{2k-1}$.  For convenience, we shall identify $V$ with the space $\{ x \in \R^{2k} : \sum x_i = 0 \}$.  With this identification, we have

\begin{align*}
i(x_j) & = y_j - y_{2k + 1 - j} \\
\Delta_K & = \{ \pm \xi_i \pm \xi_j : 1 \le i \neq j \le k \} \\
\Delta & = \{ \eta_i - \eta_j : 1 \le i \neq j \le 2k \}.
\end{align*}
The assumption that $t \in V_K^+$ is equivalent to saying that the coordinates $t_i$ of $t$ are in decreasing order, and $t_{k-1} + t_k \ge 0$.  If $w \in W$ is the element that switches $x_k$ and $x_{k+1}$, then $w$ stabilises $V_K$ and acts on it by changing the sign of $t_k$.  We may therefore assume without loss of generality that $t_k \ge 0$.  This implies that $t' = (t_1, \ldots, t_k, -t_k, \ldots, -t_1)$, and that

\begin{align*}
\langle t, \rho_K \rangle & = (2k-2)t_1 + (2k-4)t_2 + \ldots + 2t_{k-1} \\
\langle t', \rho \rangle & = (4k-2)t_1 + (4k-6)t_2 + \ldots + 2t_k
\end{align*}
so that

\bes
\langle t', \rho \rangle - 2 \langle t, \rho_K \rangle = 2t_1 + 2t_2 + \ldots + 2t_k > 0.
\ees

\subsection{$\gsl(2k+1)$}

We have $V_K \simeq \R^{k}$ and $V \simeq \{ x \in \R^{2k+1} : \sum x_i = 0 \}$.  We have

\begin{align*}
i(x_j) & = y_j - y_{2k + 2 - j} \\
\Delta_K & = \{ \pm \xi_i \pm \xi_j : 1 \le i \neq j \le k \} \cup \{ \pm \xi_j : 1 \le j \le k \} \\
\Delta & = \{ \eta_i - \eta_j : 1 \le i \neq j \le 2k+1 \}.
\end{align*}
The condition $t \in V_K^+$ is equivalent to saying that $t_i$ are non-negative and decreasing, so that $t' = (t_1, \ldots, t_k, 0, -t_k, \ldots, -t_1)$.  This gives

\begin{align*}
\langle t, \rho_K \rangle & = (2k-1)t_1 + (2k-3)t_3 + \ldots + t_k \\
\langle t', \rho \rangle & = 4kt_1 + (4k-4)t_2 + \ldots 4t_k
\end{align*}
so that

\bes
\langle t', \rho \rangle - 2 \langle t, \rho_K \rangle = 2t_1 + 2t_2 + \ldots + 2t_k > 0.
\ees

\subsection{$\mathfrak{sp}(2k)$}

We have $V_K = V \simeq \R^k$, and

\begin{align*}
\Delta_K & = \{ \xi_i - \xi_j : 1 \le i \neq j \le k \}  \\
\Delta & = \{ \pm \xi_i \pm \xi_j : 1 \le i \neq j \le k \} \cup \{ \pm 2\xi_j : 1 \le j \le k \}.
\end{align*}
The condition $t \in V_K^+$ is equivalent to saying that $t_i$ are decreasing, while $t_1', \ldots, t_k'$ is the decreasing rearrangement of $|t_1|, \ldots, |t_k|$.  Let $j$ be the largest number with $t_j \ge 0$.  Because

\bes
\langle t, \rho_K \rangle = (n-1) t_1 + \ldots + (n - 2j + 1)t_j + (n - 2j - 1)t_{j+1} + \ldots + (1-n)t_k,
\ees
we see that $\langle t, \rho_K \rangle$ is obtained by summing $t_1', \ldots, t_k'$ with some permutation of the weights $n-1, \ldots, n+1 - 2j$ and $n-1, \ldots, n+1 - 2(n-j)$.  We obtain an upper bound for this sum by replacing each weight by its absolute value and arranging them in decreasing order.  If $l = \lfloor n/2 \rfloor$, this gives

\bes
\langle t, \rho_K \rangle \le (n-1)(t_1' + t_2') + (n-3)(t_3' + t_4') + \ldots + (n - 2l +1)(t'_{2l-1} + t'_{2l}).
\ees
We also have

\bes
\langle t', \rho \rangle = 2nt_1' + 2(n-1)t_2' + \ldots + 2t_n',
\ees
and combining these gives

\bes
\langle t', \rho \rangle - 2 \langle t, \rho_K \rangle = 2t_1' + 2t_3' + \ldots + 2t_{2m+1}' > 0
\ees
where $m = \lfloor (n-1)/2 \rfloor$.

\subsection{$\mathfrak{so}(2k,2k)$}

We have $V_K = V \simeq \R^{2k}$, and

\begin{align*}
\Delta_K & = \{ \pm \xi_i \pm \xi_j : 1\le i \neq j \le k \} \cup \{ \pm \xi_i \pm \xi_j : k+1 \le i \neq j \le 2k \} \\
\Delta & = \{ \pm \xi_i \pm \xi_j : 1 \le i \neq j \le 2k \}.
\end{align*}
The condition $t \in V_K^+$ is equivalent to saying that the two sequences $t_1, \ldots, t_k$ and $t_{k+1}, \ldots, t_{2k}$ are decreasing, and that $t_{k-1} + t_k \ge 0$ and $t_{2k-1} + t_{2k} \ge 0$.  We define $v \in V_K$ by saying that $v_1, \ldots, v_{2k}$ is the decreasing rearrangement of $t_1, \ldots, t_{2k}$.  We then have

\begin{align*}
\langle t, \rho_K \rangle & = (2k-2)(t_1 + t_{k+1}) + (2k-4)(t_2 + t_{k+2}) + \ldots + 2(t_{k-1} + t_{2k-1}) \\
& \le (2k-2)(v_1 + v_2) + (2k-4)(v_3 + v_4) + \ldots + 2(v_{2k-3} + v_{2k-2}),
\end{align*}
and $v \in Wt$ so that

\bes
\langle t', \rho \rangle \ge \langle v, \rho \rangle = (4k-2)v_1 + (4k-4)v_2 + \ldots + 2v_{2k-1}.
\ees
Combining these gives

\be
\label{so2k2k}
\langle t', \rho \rangle - 2 \langle t, \rho_K \rangle \ge 2v_1 + 2v_3 + \ldots + 2v_{2k-1}.
\ee
If $v_{2k-1} \ge 0$, then the RHS of (\ref{so2k2k}) is $\ge 0$ with equality iff $t_j = 0$ for all but one $j$.  If equality occurs, the conditions $t_{k-1} + t_k \ge 0$ and $t_{2k-1} + t_{2k} \ge 0$ then imply that $t = 0$, a contradiction.  We may therefore assume that $v_{2k-1} < 0$, which implies that $t_k < 0$ and $t_{2k} < 0$.  We define $v' \in V_K$ by changing the signs of the last two coordinates of $v$.  We then have $v' \in Wt$ and

\bes
\langle t', \rho \rangle \ge \langle v', \rho \rangle = (4k-2)v_1 + (4k-4)v_2 + \ldots + 4v_{2k-2} - 2v_{2k-1},
\ees
so that

\bes
\langle t', \rho \rangle - 2 \langle t, \rho_K \rangle \ge 2v_1 + 2v_3 + \ldots + 2v_{2k-3} - 2v_{2k-1} > 0.
\ees

\subsection{$\mathfrak{so}(2k+1,2k)$}

We have $V_K = V \simeq \R^{2k}$, and

\begin{align*}
\Delta_K & = \{ \pm \xi_i \pm \xi_j : 1\le i \neq j \le k \} \cup \{ \pm \xi_j : 1 \le j \le k \} \cup \{ \pm \xi_i \pm \xi_j : k+1 \le i \neq j \le 2k \} \\
\Delta & = \{ \pm \xi_i \pm \xi_j : 1 \le i \neq j \le 2k \} \cup \{ \pm \xi_j : 1 \le j \le 2k \}.
\end{align*}
The condition $t \in V_K^+$ is equivalent to saying that the two sequences $t_1, \ldots, t_k$ and $t_{k+1}, \ldots, t_{2k}$ are decreasing, and that $t_k \ge 0$ and $t_{2k-1} + t_{2k} \ge 0$.  As in the previous case, we define $v \in Wt$ to be the decreasing rearrangement of $t$.  We then have

\begin{align*}
\langle t, \rho_K \rangle & = (2k-1)t_1 + (2k-3)t_2 + \ldots + t_k + (2k-2)t_{k+1} + (2k-4)t_{k+2} + \ldots + 2t_{2k-1} \\
& \le (2k-1)v_1 + (2k-2)v_2 + \ldots + v_{2k-1}
\end{align*}
and

\bes
\langle t', \rho \rangle \ge \langle v, \rho \rangle = (4k-1)v_1 + (4k-3)v_2 + \ldots + v_{2k},
\ees
so that

\bes
\langle t', \rho \rangle - 2 \langle t, \rho_K \rangle \ge v_1 + v_2 + \ldots + v_{2k}.
\ees
If $t_{2k} \ge 0$, then we are done.  If $t_{2k} < 0$, we may define $v' \in Wt$ by changing the sign of the last coordinate of $v$.  This gives

\bes
\langle t', \rho \rangle \ge \langle v', \rho \rangle = (4k-1)v_1 + (4k-3)v_2 + \ldots 3v_{2k-1} - v_{2k},
\ees
and it follows that

\bes
\langle t', \rho \rangle - 2 \langle t, \rho_K \rangle \ge v_1 + v_2 + \ldots v_{2k-1} - v_{2k} > 0.
\ees

\subsection{$\mathfrak{so}(2k+1,2k+1)$}

We have $V_K \simeq \R^{2k}$, $V \simeq \R^{2k+1}$, and

\begin{align*}
i(x_j) & = y_j \\
\Delta_K & = \{ \pm \xi_i \pm \xi_j : 1\le i \neq j \le k \} \cup \{ \pm \xi_i \pm \xi_j : k+1 \le i \neq j \le 2k \} \cup \{ \pm \xi_j : 1 \le j \le 2k \}\\
\Delta & = \{ \pm \eta_i \pm \eta_j : 1 \le i \neq j \le 2k+1 \}.
\end{align*}
The condition $t \in V_K^+$ is equivalent to saying that the two sequences $t_1, \ldots, t_k$ and $t_{k+1}, \ldots, t_{2k}$ are decreasing and non-negative.  It follows that $t'$ is the decreasing rearrangement of $t_1, \ldots, t_{2k}, 0$, which gives

\begin{align*}
\langle t, \rho_K \rangle & = (2k-1)(t_1 + t_{k+1}) + (2k-3)(t_2 + t_{k+2}) + \ldots + (t_k + t_{2k}) \\
& \le (2k-1)(t_1' + t_2') + (2k-3)(t_3' + t_4') + \ldots + (t_{2k-1}' + t_{2k}')
\end{align*}
and
\bes
\langle t', \rho \rangle = 4kt_1' + (4k-2)t_2' + \ldots + 2t_{2k}',
\ees
so that

\bes
\langle t', \rho \rangle - 2 \langle t, \rho_K \rangle \ge 2t_1' + 2t_3' + \ldots + 2t_{2k-1}' > 0.
\ees

\subsection{Quasi-split classical Lie algebras}

We now deal with the cases where $\g$ is classical, and quasi-split but not split, namely $\gso(2k+2,2k)$, $\gso(2k+3,2k+1)$, $\gsu(k,k)$, and $\gsu(k+1,k)$.  In the first three cases we must show that $(\g_\C, \theta)$ is weakly small but not small, and in the last that $(\g_\C, \theta)$ is not small or weakly small.

\subsection{$\gso(2k+2,2k)$}

We have $V_K = V \simeq \R^{2k+1}$, and

\begin{align*}
\Delta_K & = \{ \pm \xi_i \pm \xi_j : 1 \le i \neq j \le k+1 \} \cup \{ \pm \xi_i \pm \xi_j : k+2 \le i \neq j \le 2k+1 \} \\
\Delta & = \{ \pm \xi_i \pm \xi_j : 1 \le i \neq j \le 2k+1 \}.
\end{align*}
Because $\dim V_K = \dim V$, we must produce $t \in V_K$ for which $\| t \|^* = 2 \| t \|_K^*$, and another for which $\| t \|^* > 2 \underset{w \in W}{\max} \| w t \|^*_K$.  For the first, choose $t = (1, 0, \ldots, 0) \in V_K^+$.  We then have $t' = t$, $\langle t, \rho_K \rangle = 2k$ and $\langle t', \rho \rangle = 4k$.  For the second, choose $t = (1,1, 0, \ldots, 0)$.  Then $t = t'$, and $\langle t', \rho \rangle = 8k-2$.  The coordinates of $\rho_K$ in the standard basis are $(2k, 2k-2, \ldots, 0, 2k-2, \ldots, 0)$, and so the maximum value of $2 \| w t \|_K^*$ is
\[
4k t_1 + 2(2k-2) t_2 + 2(2k-2) t_3 + \ldots = 8k-4
\]
as required.

\subsection{$\gso(2k+3,2k+1)$}

We have $V_K \simeq \R^{2k+1}$ and $V \simeq \R^{2k+2}$.  The embedding $\iota$ and roots are given by

\begin{align*}
\iota(x_j) & = y_j, \quad 1 \le j \le 2k+1 \\
\Delta_K & = \{ \pm \xi_i \pm \xi_j : 1 \le i \neq j \le k+1 \} \cup \{ \pm \xi_i \pm \xi_j : k+2 \le i \neq j \le 2k+1 \} \cup \{ \pm \xi_i : 1 \le i \le 2k+1 \} \\
\Delta & = \{ \pm \eta_i \pm \eta_j : 1 \le i \neq j \le 2k+2 \}.
\end{align*}
Because $\dim V_K < \dim V$, it suffices to find $t \in V_K$ with $\| t \|^* = 2 \| t \|_K^*$.  Choose $t = (1, 0, \ldots, 0) \in V_K^+$.  We then have $t' = t$, $\langle t, \rho_K \rangle = 2k+1$, and $\langle t', \rho \rangle = 4k+2$.

\subsection{$\gsu(k,k)$}

We have $V = V_K \simeq \R^{2k-1}$.  For convenience, we shall identify $V$ with the space $\{ x \in \R^{2k} : \sum x_i = 0 \}$.  With this identification, we have

\begin{align*}
\Delta_K & = \{ \xi_i - \xi_j : 1 \le i \neq j \le k \} \cup \{ \xi_i - \xi_j : k+1 \le i \neq j \le 2k \} \\
\Delta & = \{ \xi_i - \xi_j : 1 \le i \neq j \le 2k \}.
\end{align*}
Because $\dim V_K = \dim V$, we must produce $t \in V_K$ for which $\| t \|^* = 2 \| t \|_K^*$, and another for which $\| t \|^* > 2 \underset{w \in W}{\max} \| w t \|^*_K$.  For the first, we choose $t \in V_K^+$ to be the vector obtained by concatenating two copies of $(1, 0, \ldots, 0, -1) \in \R^k$, so that $\langle t, \rho_K \rangle = 4k-4$.  We have $t' = (1,1,0, \ldots, 0,-1,-1) \in V^+$, so that $\langle t', \rho \rangle = 8k - 8 = 2\langle t, \rho_K \rangle$.

For the second, we choose $t \in V_K^+$ to be the concatenation of $(2, 0, \ldots, 0, -2) \in \R^k$ and $(1, 0, \ldots, 0, -1) \in \R^k$.  We have $t' = (2,1,0, \ldots, 0 -1, -2)$, and $\langle t', \rho \rangle = 6k-6$.  The coordinates of $\rho_K$ in the standard basis are $(k-1, k-3, \ldots, 1-k, k-1, k-3, \ldots, 1-k)$, and so the maximum value of $2 \| w t \|_K^*$ is
\[
2(k-1) t_1 + 2(k-1) t_2 + \ldots - 2(k-1) t_{2k-1} - 2(k-1) t_{2k} = 6k-5
\]
as required.

\subsection{$\gsu(k+1,k)$}

We have $V = V_K \simeq \R^{2k}$, and we again identify $V$ with the space $\{ x \in \R^{2k+1} : \sum x_i = 0 \}$.  We have

\begin{align*}
\Delta_K & = \{ \xi_i - \xi_j : 1 \le i \neq j \le k+1 \} \cup \{ \xi_i - \xi_j : k+2 \le i \neq j \le 2k+1 \} \\
\Delta & = \{ \xi_i - \xi_j : 1 \le i \neq j \le 2k+1 \}.
\end{align*}
Because $\dim V_K = \dim V$, both smallness and weak smallness will fail if for any $t \in V$ we have $\| t \|^* = 2 \underset{w \in W}{\max} \| w t \|^*_K$.  Because both sides of this equation are invariant under $W$, we may assume that $t \in V^+$ so that $t_1, \ldots, t_{2k+1}$ are decreasing and we have
\[
\| t \|^* = \langle t, \rho \rangle = 2k t_1 + (2k-2) t_2 + \ldots - 2k t_{2k+1}.
\]
The coordinates of $\rho_K$ in the standard basis of $V^*$ are $(k, k-2, \ldots, -k, k-1, k-3, \ldots, 1-k)$, and so the maximum value of $2 \| w t \|_K^*$ is
\[
2k t_1 + (2k-2) t_2 + \ldots - 2k t_{2k+1} = \| t \|^*
\]
as required.

\subsection{Exceptional Lie algebras with $V = V_K$}

In this section, we describe how one may use Magma to deal with the exceptional Lie algebras with $V = V_K$.  This condition is satisfied by the quasi-split form of $\mathfrak{e}_6$, and the split exceptional real algebras other than $\mathfrak{e}_6$.  We treat the split form of $\mathfrak{e}_6$ by hand in Section \ref{e6}.

As $V_K = V$, we have $W_K \subset W$.  Because $\| \cdot \|^*$ is $W$-invariant, proving smallness is equivalent to showing that $\| t \|^* > 2 \underset{w \in W}{\max} \| wt \|^*_K$ for all nonzero $t \in V$.  We may assume that $t \in V^+$, so that this inequality is equivalent to $\langle t, \rho \rangle > 2 \langle wt, \rho_K \rangle$ for all $w \in W$.  By linearity, it suffices to test this when $t \in V^+$ is a coweight $\omega \in V^+$.  Likewise, if there is one coweight $t \in V^+$ such that $\langle t, \rho \rangle > 2 \langle wt, \rho_K \rangle$ for all $w \in W$ then $(\g_\C, \theta)$ is weakly small.

The classification of automorphisms of complex simple Lie algebras allows us to read the root system $\Delta_K \subset \Delta$ from the extended Dynkin diagram of $\g$, see for instance \cite[Thm. 5.15]{He}.  This description of $\Delta_K$ allows us to compute $\max \{ \langle w\omega, \rho_K \rangle : w \in W \}$ for each coweight $\omega$ using an algebra package such as Magma.  We give Magma code that performs this calculation in the case of $\fre_8$, and indicate the modifications needed in the other cases.  We begin by constructing $W$.

\begin{verbatim}
W:=CoxeterGroup("E8");
\end{verbatim}
\cite[Thm. 5.15]{He} describes a set of simple roots in $\Delta$ whose reflections generate $W_K$, and we use this to compute $W_K$ and $\Delta^+_K$.

\begin{verbatim}
WH:=ReflectionSubgroup(W, {2,3,4,5,6,7,8,120});
DHplus:=[RootPosition(W, Root(WH,n)) : n in [1..(#Roots(WH)/2)]];
\end{verbatim}
We next calculate $\langle \omega_i, \rho \rangle$ for $1 \le i \le 8$, where $\omega_i$ are the coweights of $\Delta$.

\begin{verbatim}
for i in [1..8] do

s:=0;
for n in [1 .. (#Roots(W)/2)] do
    s := s+Root(W,n)[i];
end for;
s;

end for;
\end{verbatim}
We next compute $\max \{ \langle w\omega_i, \rho_K \rangle : w \in W \}$.  We reduce the time needed to do this by letting $W_i$ be the stabiliser of $\omega_i$ in $W$ and computing $\max \{ \langle w\omega_i, \rho_K \rangle : w \in W / W_i \}$.  We find a set of coset representatives for $W / W_i$.

\begin{verbatim}
for i in [1..8] do

Stab:=ReflectionSubgroup(W, Exclude({1,2,3,4,5,6,7,8},i));
Tran:=Transversal(W,Stab);
\end{verbatim}
Finally, we use our set of representatives to compute $\max \{ \langle w\omega_i, \rho_K \rangle : w \in W / W_i \}$.

\begin{verbatim}
Worbit:=[];
for w in Tran do
t:=0;
for n in DHplus do
    t := t+Root(W,n^(w^(-1)))[i];
end for;
Append(~Worbit,t);
end for;
Sort(Worbit)[#Tran];

end for;
\end{verbatim}
This produces two lists of 8 numbers, and one checks that all the numbers in the first list are more than double the corresponding number in the second list.  In the cases of $\g_2$, $\mathfrak{f}_4$, $\mathfrak{e}_6$, and $\mathfrak{e}_7$, one replaces the second line with

\begin{verbatim}
WH:=ReflectionSubgroup(W, {1,6});
WH:=ReflectionSubgroup(W, {2,3,4,24});
WH:=ReflectionSubgroup(W, {1,3,4,5,6,36});
WH:=ReflectionSubgroup(W, {1,3,4,5,6,7,63});
\end{verbatim}
and modifies the others in the obvious way.  In every case except $\mathfrak{e}_6$, all the numbers in the first list are again more than double the corresponding number in the second list, while for $\mathfrak{e}_6$ there is at least one pair for which the inequality holds.


\subsection{$\fre_6$}
\label{e6}

In the case when $\g$ is the split form of $\fre_6$, we have $\gk_\C = \mathfrak{sp}(4)_\C$.  We choose the basis $\xi_1, \ldots, \xi_4$ for $V_K^*$ in such a way that

\bes
\Delta_K = \{ \pm \xi_i \pm \xi_j : 1 \le i \neq j \le 4 \} \cup \{ \pm 2\xi_i : 1 \le i \le 4 \}.
\ees
We let $\p$ be the $\gk$-invariant complement to $\gk$ in $\g$.  Theorem 5.15 of \cite{He} implies that the representation of $\gk$ on $\p$ has highest weight $\xi_1 + \xi_2 + \xi_3 + \xi_4$.  It follows that we have $\p_\C \simeq \bigwedge^4 \text{std} / ( \omega \wedge \bigwedge^2 \text{std} )$, where $\text{std}$ denotes the standard representation of $\gk_\C$ on $\C^8$ and $\omega$ is the invariant symplectic form on $\C^8$ fixed by $\gk_\C$.  If we define $\Sigma$ to be the multiset $\{ \alpha|_{V_K} : \alpha \in \Delta \}$, it follows that

\bes
\Sigma = \Delta_K \cup \{ \pm \xi_1 \pm \xi_2 \pm \xi_3 \pm \xi_4 \} \cup \{ \pm \xi_i \pm \xi_j : 1 \le i \le 4 \}
\ees
as multisets.  In particular, any $\alpha \in \Delta$ has a nonzero restriction to $V_K$.  Let $\Sigma^+ = \{ \alpha \in \Sigma : \alpha( (5,3,2,1) ) > 0 \}$, so that $\Sigma = \Sigma^+ \cup -\Sigma^+$.  There is a set of positive roots $\Delta^+ \subset \Delta$ such that $\Sigma^+ = \{ \alpha|_{V_K} : \alpha \in \Delta^+ \}$, and so if $\rho$ and $\rho_\Sigma$ are the the half sums of $\Delta^+$ and $\Sigma^+$ then $\rho|_{V_K} = \rho_\Sigma$.

\begin{lemma}

We have $\max\{ \langle w t, \rho \rangle : w \in W \} \ge \max \{ \langle w_K t, \rho_\Sigma \rangle : w_K \in W_K \}$.

\end{lemma}

\begin{proof}

Let $w_K \in W_K$.  We have $w_K \Sigma^+ = \{ \alpha \in \Sigma : \langle \alpha, w_K (5,3,2,1) \rangle > 0 \}$, so as before there is a set of positive roots $\Phi^+ \subset \Delta$ with $w_K \Sigma^+ = \{ \alpha|_{V_K} : \alpha \in \Phi^+\}$.  If $\rho_\Phi$ is the half sum of $\Phi^+$, then $\rho_\Phi|_{V_K} = w_K \rho_\Sigma$, and so $\langle t, w_K \rho_\Sigma \rangle = \langle t, \rho_\Phi \rangle \le \max\{ \langle wt, \rho \rangle : w \in W \}$ as required.

\end{proof}

It therefore suffices to prove that

\bes
\max\{ \langle w_K t, \rho_\Sigma \rangle : w_K \in W_K \} > 2 \max \{ \langle w_K t, \rho_K \rangle : w_K \in W_K \}.
\ees
We may assume that $t \in V_K^+$, in which case we must show that $\max\{ \langle w_K t, \rho_\Sigma \rangle : w_K \in W_K \} > 2\langle t, \rho_K \rangle$.  We have

\begin{align*}
\langle t, \rho_\Sigma \rangle & = 20t_1 + 12t_2 + 8t_3 + 4x_4 \\
\langle t, \rho_K \rangle & = 8t_1 + 6t_2 + 4t_3 + 2t_4 \\
\langle t, \rho_\Sigma \rangle - 2\langle t, \rho_K \rangle & = 4t_1 > 0,
\end{align*}
which completes the proof in this case.

\end{document}